\documentclass[12pt,letterpaper]{article}

\usepackage{amsmath,amssymb,amsthm,mathtools}
\usepackage[margin=0.8in,a4paper]{geometry}    
\usepackage{cite} 
\usepackage{caption}
\usepackage{float}
\usepackage{tikz}
\usepackage{authblk}
\usepackage[colorlinks=true,linkcolor=blue,citecolor=blue,urlcolor=blue]{hyperref}
\usepackage{parskip}
\usepackage{enumitem}
 \usepackage{booktabs}
\usepackage{csquotes}
 \usepackage{lmodern}
\usepackage{orcidlink}
\usepackage{xcolor}
\usepackage{algorithm}
\usepackage{algpseudocode}
\usepackage{pgfplotstable}

\theoremstyle{plain}
\newtheorem{theorem}{Theorem}
\newtheorem{corollary}{Corollary}
\newtheorem{lemma}{Lemma}
\newtheorem{proposition}{Proposition}

\theoremstyle{definition}

\newtheorem{conjecture}{Conjecture}

\theoremstyle{remark}
\newtheorem{remark}{Remark}

\pgfplotsset{compat=1.18}

\DeclareRobustCommand{\seqnum}[1]{%
  \ifmmode
    \text{\href{https://oeis.org/#1}{\textcolor{blue}{{\normalfont\ttfamily #1}}}}%
  \else
    \href{https://oeis.org/#1}{\textcolor{blue}{{\normalfont\ttfamily #1}}}%
  \fi
}

\author{El-Mehdi Mehiri \orcidlink{0000-0002-7164-3658}\\
Mines Saint-Etienne, CMP, Department of Manufacturing Sciences and Logistics, F-13120 Gardanne, France.\\
\url{elmehdi.mehiri@emse.fr}\\
\url{mehiri314@gmail.com}}

\title{\textbf{The Parity-Constrained Four-Peg Tower of Hanoi Problem and Its State Graph\thanks{\textcopyright   E.-M.Mehiri 2026}
}}

\date{19-06-2026}

\begin{document}
 
 \maketitle

\begin{abstract}
 \noindent We introduce and study a parity-constrained variant of the four-peg Tower of
Hanoi problem. In this model, two pegs are neutral, while the two remaining
pegs are reserved respectively for even-labelled and odd-labelled discs. Starting
from the classical initial tower, we consider four natural transfer objectives
corresponding to different target configurations of the full tower and of the
even and odd subtowers.

  \noindent For these four objectives, we propose a system of recursive algorithms based on
parity separation and classical three-peg transfers. These algorithms lead to a
coupled system of recurrence relations for their move counts. The resulting
candidate sequences are then transformed into simplified and higher-order
recurrences, from which explicit closed formulas are obtained. The formulas
exhibit a periodic structure and have the same exponential order of growth,
strictly slower than that of the classical three-peg Tower of Hanoi.

 \noindent  The main open point is the optimality of the proposed recursive algorithms.
Equivalently, one has to prove that certain canonical configurations, including
the one-move behaviour of the largest disc, are unavoidable in every shortest
solution. This difficulty is closely analogous to the structural difficulties
encountered in the Reve's puzzle and the Frame--Stewart conjecture. We therefore
formulate the optimality of the proposed algorithms as a conjecture. We also
discuss computational evidence, the number of shortest solutions, a linear
variant in which only adjacent peg moves are allowed, and the associated state
graph of the parity-constrained problem.

\end{abstract}



\section{Introduction}

The classical Tower of Hanoi puzzle asks for the minimum number of moves required to transfer a stack of $n$ discs from one peg to another, under the rules that only one disc may be moved at a time and that no larger disc may be placed on top of a smaller one. In the three-peg setting, the problem admits the well-known optimal solution
\begin{equation}
    h_3^n = 2^n-1,
\end{equation}
and the sequence $(h_3^n)_{n\geq 0}$ appears in the \emph{On-Line Encyclopedia of Integer Sequences} as \texttt{A000225} \cite{OEIS}. For general
background on the Tower of Hanoi and its many variations, we refer to the monograph of
Hinz, Klav\v zar, and Petr \cite{HinzMythsMaths2018}.

Although this classical version has been understood for more than a century, the four-peg Tower of Hanoi problem, also known as \emph{Reve's puzzle}, proved substantially more difficult. Its optimal solution was only confirmed in 2014 by Bousch \cite{Bousch2014}, who established the optimality of the Frame--Stewart recurrence
\begin{equation}
    h_{4}^{0}=0,\qquad\forall n\geq 1:\; h_{4}^{n}=\min\{\,2h_{4}^{k}+h_{3}^{n-k}\mid 0\le k \le n-1\,\}.
\end{equation}
For five or more pegs, the Frame--Stewart algorithm \cite{Frame1941,Stewart1941} is still conjectured to be optimal, but a complete proof remains open; see \cite[Chapter~5]{HinzMythsMaths2018} for a detailed account.

In this paper, we introduce a new variant of the four-peg Tower of Hanoi in which disc movements are subject to \emph{parity constraints}. Among the four pegs, two are \emph{neutral} and may hold discs of any label, whereas the remaining two are restricted: one may hold only even-labelled discs and the other only odd-labelled discs. This restriction creates a new interaction between the classical Hanoi
ordering rule and a parity-induced separation rule.

We investigate four natural optimization objectives associated with this parity-constrained model. The main contribution of the paper is the construction of four coupled recursive
algorithms for these objectives. The construction is based on repeatedly separating the smaller discs by parity,
moving one or two large discs, and then completing the required transfer by
ordinary three-peg Hanoi moves on one parity class. This leads to a coupled
system of recurrences for the four candidate sequences. These recurrences can be
simplified, decoupled, and solved explicitly.

The central conjecture of the paper is that these recursive algorithms are
optimal.

We call this the main optimality conjecture for the parity-constrained four-peg
Tower of Hanoi problem.

The difficulty in proving this conjecture is structural. The proposed recursive
algorithms have the largest disc move exactly once, and they require certain
canonical arrangements of the smaller discs before this move. To prove
optimality, one would have to show that such arrangements are unavoidable in
every shortest solution. This is precisely the type of issue that appears in the
classical Reve's puzzle and in the Frame--Stewart framework: an intuitively
natural recursive decomposition does not automatically imply a shortest
solution. For this reason, we present the recurrences as conjectural optimal
recurrences rather than as proved optimal formulas.

Besides the main conjecture, we derive several rigorous consequences concerning
the proposed algorithms. In particular, we obtain higher-order recurrences and
closed formulas for the candidate sequences. 
 
These formulas reveal a periodic structure modulo \(6\) and show that the four
candidate sequences have the same exponential order of growth. We also examine
the number of shortest solutions suggested by computational experiments, discuss
a linear variant in which only adjacent moves are allowed, and introduce the
state graph associated with the parity-constrained problem.

The paper is organized as follows. Section~\ref{sec:ProblemDescription} defines the parity-constrained four-peg Tower of Hanoi problem and introduces the four transfer objectives considered throughout the paper. Section~\ref{sec:algorithms} presents the four recursive algorithms, derives the coupled recurrence system for their candidate move counts, and records the corresponding canonical decompositions. Section~\ref{sec:conjecture} formulates the main optimality conjecture and discusses the structural obstacle to proving it, namely the difficulty of showing that the canonical intermediate configurations are unavoidable in shortest solutions. Section~\ref{sec:DerivedRecurrences} derives higher-order recurrences and explicit closed formulas for the candidate sequences, including an alternative decoupled formulation. Section~\ref{sec:AsymptoticGrowth} studies the asymptotic growth of these sequences and compares them with the classical three-peg and four-peg Hanoi sequences. Section~\ref{sec:NumberofOptimalSolutions} records computational observations and conjectural formulas for the number of shortest solutions. Section~\ref{sec:LinearVariant} studies a linear variant of the problem in which only adjacent moves are allowed. Section~\ref{sec:AssociatedGraph} introduces the associated state graph and investigates several of its structural and combinatorial properties. Finally, Section~\ref{sec:Conclusion} summarizes the results and discusses possible directions for future research.

\section{Problem Description}
\label{sec:ProblemDescription}

Fix an integer $n \geq 1$. Consider four pegs partitioned as follows: $P = \{N_{1}, N_{2}\} \cup \{O, E\}$,  where $N_{1}$ and $N_{2}$ are \emph{neutral} pegs that may hold discs of any label, $O$ is the \emph{odd} peg that may only hold discs with odd labels, and $E$ is the \emph{even} peg that may only hold discs with even labels. The discs are labeled from top to bottom as $[n] = \{1,2,\dots,n\}$,  with disc $1$ being the smallest.

A \emph{state} is a placement of all $n$ discs on the four pegs such that:
\begin{enumerate}
    \item[\textbf{(i)}] the discs on every peg form a strictly increasing sequence from top to bottom (no larger disc lies on a smaller one);
    \item[\textbf{(ii)}] every disc occupies a peg compatible with its parity.
\end{enumerate}

For $k \geq 1$, define $[k] = \{1,2,\dots,k\}$, $[k]_{0} = \{d \in [k] \mid d \text{ even}\}$, and $[k]_{1} = \{d \in [k] \mid d \text{ odd}\}$.
 
Any feasible state can then be expressed as an ordered $4$-partition of $[n]$: $S = (N_{1}, E, O, N_{2})$, where each component denotes the (stacked) set of discs assigned to the corresponding peg. This representation is a parity-restricted adaptation of the classical encoding used in~\cite{Rittaud01082023}.

We denote the \emph{allowed parity peg}  for disc $d\in[n]$ by     $p(d)\in\{E,O\}$, that is,
\[
p(d)=
\begin{cases}
E,&  d \text{ even},\\
O,&   d \text{ odd},
\end{cases}
\]
and by $\overline{p}(d)$ the \emph{forbidden parity peg} of disc $d$, that is, $\overline{p}(d)\in \{E,O\}\setminus\{p(d)\}$.

Given a state $S$ and a disc $d \in [n]$ currently on peg $p \in P$, a move $p \to q$, with $q \in P$,  is \emph{legal} if and only if:
\begin{enumerate}
    \item  (\emph{classical rule}): disc $d$ is the topmost disc on $p$, and $q$ is either empty or its top disc is larger than $d$;
    \item  (\emph{parity rule}): both $p$ and $q$ permit the parity of $d$, that is,
    \[
    \text{$d$ odd} \Rightarrow p,q \in \{N_{1},N_{2},O\}, \qquad
    \text{$d$ even} \Rightarrow p,q \in \{N_{1},N_{2},E\}.
    \]
\end{enumerate}

Starting from the \emph{initial state} $S_{0} = \bigl([n], \varnothing, \varnothing, \varnothing\bigr)$,  we consider four canonical target configurations, each defining an optimal move-count function.

\begin{enumerate}
\item[\textup{\texttt{(a)}}] \textbf{Full-tower transfer.}
Move all discs from peg $N_{1}$ to peg $N_{2}$.  
The optimum number of moves is denoted by $A_{n}$:
\[
([n], \varnothing, \varnothing, \varnothing)
\;\longrightarrow\;
(\varnothing, \varnothing, \varnothing, [n]).
\]

\item[\textup{\texttt{(b)}}]  \textbf{Parity separation.}
Move all even-labelled discs to peg $E$ and all odd-labelled discs to peg $O$.  
Denote the optimum number by $B_{n}$:
\[
([n], \varnothing, \varnothing, \varnothing)
\;\longrightarrow\;
(\varnothing, [n]_{0}, [n]_{1}, \varnothing).
\]

\item[\textup{\texttt{(c)}}]  \textbf{Odd to $O$, even to $N_{2}$.}
Transfer all odd discs to peg $O$ and all even discs to peg $N_{2}$.  
Denote the optimum number by $C_{n}$:
\[
([n], \varnothing, \varnothing, \varnothing)
\;\longrightarrow\;
(\varnothing, \varnothing, [n]_{1}, [n]_{0}).
\]

\item[\textup{\texttt{(d)}}]  \textbf{Odd to $N_{2}$, even to $E$.}
Transfer all odd discs to peg $N_{2}$ and all even discs to peg $E$.  
Denote the optimum number by $D_{n}$:
\[
([n], \varnothing, \varnothing, \varnothing)
\;\longrightarrow\;
(\varnothing, [n]_{0}, \varnothing, [n]_{1}).
\]
\end{enumerate}

\begin{figure}[ht]
    \centering

\tikzset{every picture/.style={line width=0.75pt}} 
\resizebox{\columnwidth}{!}{
\begin{tikzpicture}[x=0.75pt,y=0.75pt,yscale=-1,xscale=1]

\draw [line width=1.5]    (0,290) -- (160,290) ;
\draw [line width=1.5]    (20,290) -- (20,240) ;
\draw [line width=1.5]    (60,290) -- (60,240) ;
\draw [line width=1.5]    (100,290) -- (100,240) ;
\draw [line width=1.5]    (140,290) -- (140,240) ;
\draw [line width=1.5]    (270,92) -- (430,92) ;
\draw [line width=1.5]    (290,92) -- (290,42) ;
\draw [line width=1.5]    (330,92) -- (330,42) ;
\draw [line width=1.5]    (370,92) -- (370,42) ;
\draw [line width=1.5]    (410,92) -- (410,42) ;
\draw  [fill={rgb, 255:red, 126; green, 211; blue, 33 }  ,fill opacity=1 ] (290,47) -- (307.5,92) -- (272.5,92) -- cycle ;
\draw  [fill={rgb, 255:red, 126; green, 211; blue, 33 }  ,fill opacity=1 ] (140,245) -- (157.5,290) -- (122.5,290) -- cycle ;
\draw [line width=1.5]    (180,290) -- (340,290) ;
\draw [line width=1.5]    (200,290) -- (200,240) ;
\draw [line width=1.5]    (240,290) -- (240,240) ;
\draw [line width=1.5]    (280,290) -- (280,240) ;
\draw [line width=1.5]    (320,290) -- (320,240) ;
\draw [line width=1.5]    (360,290) -- (520,290) ;
\draw [line width=1.5]    (380,290) -- (380,240) ;
\draw [line width=1.5]    (420,290) -- (420,240) ;
\draw [line width=1.5]    (460,290) -- (460,240) ;
\draw [line width=1.5]    (500,290) -- (500,240) ;
\draw [line width=1.5]    (540,290) -- (700,290) ;
\draw [line width=1.5]    (560,290) -- (560,240) ;
\draw [line width=1.5]    (600,290) -- (600,240) ;
\draw [line width=1.5]    (640,290) -- (640,240) ;
\draw [line width=1.5]    (680,290) -- (680,240) ;
\draw  [fill={rgb, 255:red, 80; green, 227; blue, 194 }  ,fill opacity=1 ] (240,255) -- (252.5,290) -- (227.5,290) -- cycle ;
\draw  [fill={rgb, 255:red, 248; green, 231; blue, 28 }  ,fill opacity=1 ] (280,255) -- (292.5,290) -- (267.5,290) -- cycle ;
\draw  [fill={rgb, 255:red, 80; green, 227; blue, 194 }  ,fill opacity=1 ] (500,255) -- (512.5,290) -- (487.5,290) -- cycle ;
\draw  [fill={rgb, 255:red, 248; green, 231; blue, 28 }  ,fill opacity=1 ] (460,255) -- (472.5,290) -- (447.5,290) -- cycle ;
\draw  [fill={rgb, 255:red, 80; green, 227; blue, 194 }  ,fill opacity=1 ] (600,255) -- (612.5,290) -- (587.5,290) -- cycle ;
\draw  [fill={rgb, 255:red, 248; green, 231; blue, 28 }  ,fill opacity=1 ] (680,255) -- (692.5,290) -- (667.5,290) -- cycle ;

\draw (13,299.4) node [anchor=north west][inner sep=0.75pt]  [font=\scriptsize,]  {$N_{1}$};
\draw (133,299.4) node [anchor=north west][inner sep=0.75pt]  [font=\scriptsize,]  {$N_{2}$};
\draw (93,299.4) node [anchor=north west][inner sep=0.75pt]  [font=\scriptsize,]  {$O$};
\draw (53,299.4) node [anchor=north west][inner sep=0.75pt]  [font=\scriptsize,]  {$E$};
\draw (283,101.4) node [anchor=north west][inner sep=0.75pt]  [font=\scriptsize,]  {$N_{1}$};
\draw (403,101.4) node [anchor=north west][inner sep=0.75pt]  [font=\scriptsize,]  {$N_{2}$};
\draw (363,101.4) node [anchor=north west][inner sep=0.75pt]  [font=\scriptsize,]  {$O$};
\draw (323,101.4) node [anchor=north west][inner sep=0.75pt]  [font=\scriptsize,]  {$E$};
\draw (193,299.4) node [anchor=north west][inner sep=0.75pt]  [font=\scriptsize,]  {$N_{1}$};
\draw (313,299.4) node [anchor=north west][inner sep=0.75pt]  [font=\scriptsize,]  {$N_{2}$};
\draw (273,299.4) node [anchor=north west][inner sep=0.75pt]  [font=\scriptsize,]  {$O$};
\draw (233,299.4) node [anchor=north west][inner sep=0.75pt]  [font=\scriptsize,]  {$E$};
\draw (373,299.4) node [anchor=north west][inner sep=0.75pt]  [font=\scriptsize,]  {$N_{1}$};
\draw (493,299.4) node [anchor=north west][inner sep=0.75pt]  [font=\scriptsize,]  {$N_{2}$};
\draw (453,299.4) node [anchor=north west][inner sep=0.75pt]  [font=\scriptsize,]  {$O$};
\draw (413,299.4) node [anchor=north west][inner sep=0.75pt]  [font=\scriptsize,]  {$E$};
\draw (553,299.4) node [anchor=north west][inner sep=0.75pt]  [font=\scriptsize,]  {$N_{1}$};
\draw (673,299.4) node [anchor=north west][inner sep=0.75pt]  [font=\scriptsize,]  {$N_{2}$};
\draw (633,299.4) node [anchor=north west][inner sep=0.75pt]  [font=\scriptsize,]  {$O$};
\draw (593,299.4) node [anchor=north west][inner sep=0.75pt]  [font=\scriptsize,]  {$E$};
\draw (22,345.91) node [anchor=north west][inner sep=0.75pt]  [font=\footnotesize,] [align=left] {\textbf{\textup{\texttt{(a)}}} - Full-tower transfer};
\draw (202,345.91) node [anchor=north west][inner sep=0.75pt]  [font=\footnotesize,] [align=left] {\textbf{\textup{\texttt{(b)}}} - Parity separation};
\draw (372,345.91) node [anchor=north west][inner sep=0.75pt]  [font=\footnotesize,] [align=left] {\textbf{\textup{\texttt{(c)}}} - Odd to $\displaystyle O$, even to $\displaystyle N_{2}$};
\draw (552,345.91) node [anchor=north west][inner sep=0.75pt]  [font=\footnotesize,] [align=left] {\textbf{\textup{\texttt{(d)}}} - Odd to $\displaystyle N_{2}$ even to $\displaystyle E$};
\draw (312,2) node [anchor=north west][inner sep=0.75pt]  [] [align=left] {Initial state};
\draw (312,202) node [anchor=north west][inner sep=0.75pt]  [] [align=left] {Final state};

\end{tikzpicture}

    }
    \caption{Illustration of the four main objectives \textup{\texttt{(a)}}--\textup{\texttt{(d)}}.  
    Green: entire tower; blue: even sub-tower; yellow: odd sub-tower.}
    \label{fig:objectives}
\end{figure}
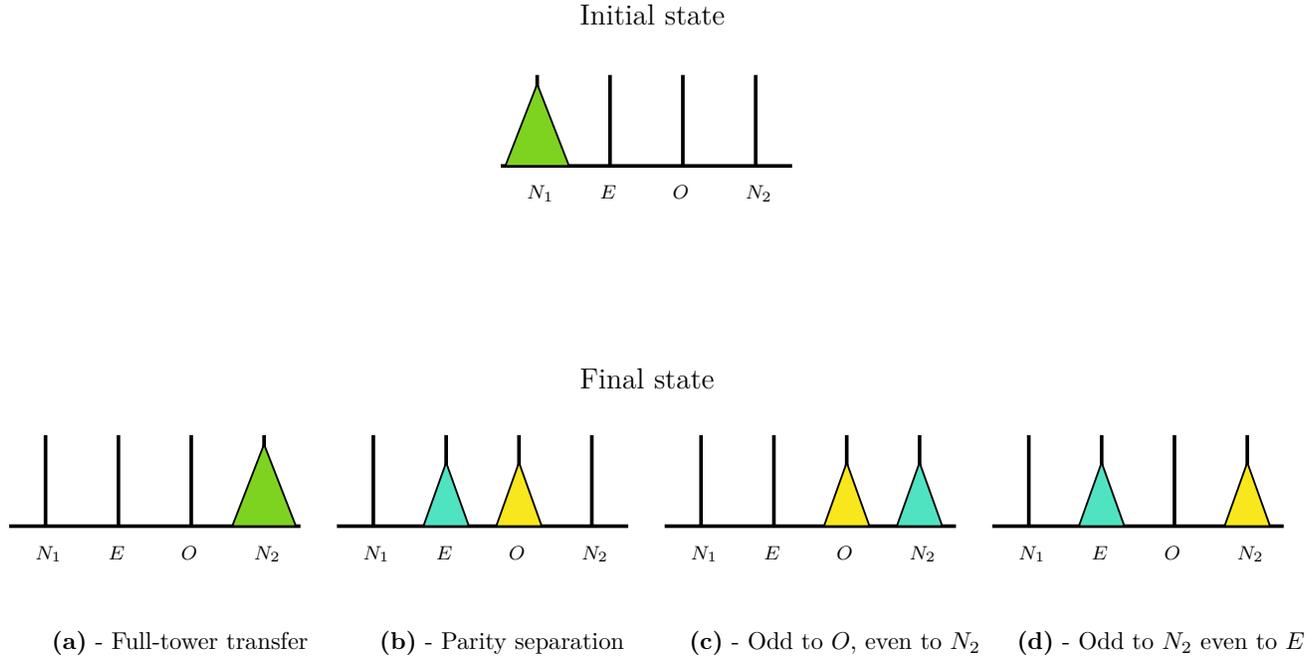

A state is called \emph{perfect} if all discs are stacked on the same peg in increasing order.
Under the given parity constraints, only two such perfect states exist for $n\geq 2$: $([n], \varnothing, \varnothing, \varnothing)$ and $(\varnothing, \varnothing, \varnothing, [n])$,  which correspond to the initial and final configurations of objective~\textup{\texttt{(a)}}.

\section{Recursive Algorithms and Candidate Move Counts}\label{sec:algorithms}

In this section, we describe four recursive algorithms for the objectives
introduced in Section~\ref{sec:ProblemDescription}. We emphasize that, at this stage, these algorithms are
not claimed to be optimal. Their optimality will be formulated as a conjecture
in the next section.

We denote by $a_n$, $b_n$, $c_n$, and $d_n$, the move counts of the recursive algorithms for objectives \texttt{(a)}, \texttt{(b)},
\texttt{(c)}, and \texttt{(d)}, respectively. Thus \(a_n,b_n,c_n,d_n\) are candidate
values for the true optima \(A_n,B_n,C_n,D_n\).

 Whenever all discs involved have the same parity, the parity-constrained
problem reduces to an ordinary three-peg Hanoi transfer on the three pegs
available to that parity. We denote such a transfer by $\mathsf{H}_3(T;X,Y,Z)$,  where \(T\) is the tower to be moved from peg \(X\) to peg \(Y\), using peg \(Z\)
as auxiliary. Since every legal move is reversible, we shall also use the reverse of any of
the recursive procedures below. For example, the reverse of the parity-separation
algorithm collects a parity-separated configuration back onto a prescribed
neutral peg.

\begin{algorithm}[ht]
\caption{Algorithm \(\mathcal A(n)\): full-tower transfer}\label{alg:A}
\begin{algorithmic}[1]
\Require All discs \(1,\ldots,n\) are stacked on \(N_1\).
\Ensure All discs \(1,\ldots,n\) are stacked on \(N_2\).
\If{\(n=1\)}
    \State Move disc \(1\) from \(N_1\) to \(N_2\).
    \State \Return
\EndIf
\State Apply \(\mathcal B(n-1)\) to the smaller discs \(1,\ldots,n-1\).
\State Move disc \(n\) from \(N_1\) to \(N_2\).
\State Apply the reverse of \(\mathcal B(n-1)\) to collect the smaller discs
       from \(E\) and \(O\) onto \(N_2\).
\end{algorithmic}
\end{algorithm}

\begin{algorithm}[ht]
\caption{Algorithm \(\mathcal B(n)\): parity separation}\label{alg:B}
\begin{algorithmic}[1]
\Require All discs \(1,\ldots,n\) are stacked on \(N_1\).
\Ensure Even discs are stacked on \(E\), and odd discs are stacked on \(O\).
 
\If{\(n=1\)}
    \State Move disc \(1\) from \(N_1\) to \(O\).
    \State \Return
\EndIf
\If{\(n\) is even}
    \State Apply \(\mathcal C(n-1)\) to the smaller discs \(1,\ldots,n-1\).
    \State Move disc \(n\) from \(N_1\) to \(E\).
    \State Apply \(\mathsf H_3([n-1]_0;N_2,E,N_1)\) to move the smaller even tower from \(N_2\) to \(E\).
\Else
    \State Apply \(\mathcal D(n-1)\) to the smaller discs \(1,\ldots,n-1\).
    \State Move disc \(n\) from \(N_1\) to \(O\).
    \State Apply \(\mathsf H_3([n-1]_1;N_2,O,N_1)\) to move the smaller odd tower from \(N_2\) to \(O\).
\EndIf
\end{algorithmic}
\end{algorithm}

\begin{algorithm}[ht]
\caption{Algorithm \(\mathcal C(n)\): odd discs to \(O\), even discs to \(N_2\)}\label{alg:C}
\begin{algorithmic}[1]
\Require All discs \(1,\ldots,n\) are stacked on \(N_1\).
\Ensure Odd discs are stacked on \(O\), and even discs are stacked on \(N_2\).
 
\If{\(n=1\)}
    \State Move disc \(1\) from \(N_1\) to \(O\).
    \State \Return
\EndIf
\If{\(n\) is even}
    \State Apply \(\mathcal B(n-1)\) to the smaller discs \(1,\ldots,n-1\).
    \State Move disc \(n\) from \(N_1\) to \(N_2\).
    \State Apply \(\mathsf H_3([n-1]_0;E,N_2,N_1)\) to move the smaller even tower from \(E\) to \(N_2\).
\Else
    \State Apply \(\mathcal B(n-2)\) to the discs \(1,\ldots,n-2\), keeping discs \(n-1\) and \(n\) fixed on \(N_1\).
    \State Move disc \(n-1\) from \(N_1\) to \(N_2\).
    \State Apply \(\mathsf H_3([n-2]_1;O,N_2,N_1)\) to move the smaller odd tower from \(O\) to \(N_2\).
    \State Move disc \(n\) from \(N_1\) to \(O\).
    \State Apply \(\mathsf H_3([n-2]_1;N_2,O,N_1)\) to move the smaller odd tower back from \(N_2\) to \(O\).
    \State Apply \(\mathsf H_3([n-2]_0;E,N_2,N_1)\) to move the smaller even tower from \(E\) to \(N_2\).
\EndIf
\end{algorithmic}
\end{algorithm}

\begin{algorithm}[ht]
\caption{Algorithm \(\mathcal D(n)\): even discs to \(E\), odd discs to \(N_2\)}\label{alg:D}
\begin{algorithmic}[1]
\Require All discs \(1,\ldots,n\) are stacked on \(N_1\).
\Ensure Even discs are stacked on \(E\), and odd discs are stacked on \(N_2\).
 
\If{\(n=1\)}
    \State Move disc \(1\) from \(N_1\) to \(N_2\).
    \State \Return
\EndIf
\If{\(n\) is odd}
    \State Apply \(\mathcal B(n-1)\) to the smaller discs \(1,\ldots,n-1\).
    \State Move disc \(n\) from \(N_1\) to \(N_2\).
    \State Apply \(\mathsf H_3([n-1]_1;O,N_2,N_1)\) to move the smaller odd tower from \(O\) to \(N_2\).
\Else
    \State Apply \(\mathcal B(n-2)\) to the discs \(1,\ldots,n-2\), keeping discs \(n-1\) and \(n\) fixed on \(N_1\).
    \State Move disc \(n-1\) from \(N_1\) to \(N_2\).
    \State Apply \(\mathsf H_3([n-2]_0;E,N_2,N_1)\) to move the smaller even tower from \(E\) to \(N_2\).
    \State Move disc \(n\) from \(N_1\) to \(E\).
    \State Apply \(\mathsf H_3([n-2]_0;N_2,E,N_1)\) to move the smaller even tower back from \(N_2\) to \(E\).
    \State Apply \(\mathsf H_3([n-2]_1;O,N_2,N_1)\) to move the smaller odd tower from \(O\) to \(N_2\).
\EndIf
\end{algorithmic}
\end{algorithm}

Algorithms~\ref{alg:A}--\ref{alg:D} immediately give the following recurrence system.

\begin{proposition}[Move counts of the recursive algorithms]\label{prop:mainrec}
The candidate move counts \(a_n,b_n,c_n,d_n\) satisfy
\begin{align}
a_n &= 2b_{n-1}+1, \label{eq:a}\\[4pt]
b_n &=
\begin{cases}
c_{n-1}+h_{3}^{\lfloor \frac{n-1}{2} \rfloor}+1, & \text{$n$ even},\\[2pt]
d_{n-1}+h_{3}^{\lfloor \frac{n-1}{2} \rfloor}+1, & \text{$n$ odd},
\end{cases}\label{eq:b}\\[4pt]
c_n &=
\begin{cases}
b_{n-1}+h_{3}^{\lfloor \frac{n-1}{2} \rfloor}+1, & \text{$n$ even},\\[2pt]
b_{n-2}+2h_{3}^{\lfloor \frac{n-1}{2} \rfloor}+h_{3}^{\lfloor \frac{n-2}{2} \rfloor}+2, & \text{$n$ odd},
\end{cases}\label{eq:c}\\[4pt]
d_n &=
\begin{cases}
b_{n-2}+3h_{3}^{\lfloor \frac{n-2}{2} \rfloor}+2, & \text{$n$ even},\\[2pt]
b_{n-1}+h_{3}^{\lfloor \frac{n-1}{2} \rfloor}+1, & \text{$n$ odd},
\end{cases}\label{eq:d}
\end{align}
with initial conditions $a_0 = b_0 = c_0 = d_0 = 0$, and $a_1 = b_1 = c_1 = d_1 = 1$.

Consequently, since these algorithms are feasible,
\[
    A_n\leq a_n,\qquad
    B_n\leq b_n,\qquad
    C_n\leq c_n,\qquad
    D_n\leq d_n .
\]
\end{proposition}

\begin{proof}
The formula for \(a_n\) follows from Algorithm~\(\mathcal A\): first the
\(n-1\) smaller discs are separated by parity in \(b_{n-1}\) moves, then disc
\(n\) is moved once from \(N_1\) to \(N_2\), and finally the smaller discs are
collected onto \(N_2\) by reversing the same parity-separation algorithm.

For \(b_n\), suppose first that \(n\) is even. Algorithm~\(\mathcal B\) first
uses \(\mathcal C(n-1)\), then moves disc \(n\) once to \(E\), and finally
moves the smaller even tower from \(N_2\) to \(E\). This last step is an
ordinary three-peg Hanoi transfer on \(\lfloor(n-1)/2\rfloor\) even discs.
Thus
\[
    b_n=c_{n-1}+h_3^{\lfloor (n-1)/2\rfloor}+1.
\]
The odd case is symmetric, using \(\mathcal D(n-1)\) and then moving the smaller
odd tower from \(N_2\) to \(O\).

For \(c_n\), if \(n\) is even, Algorithm~\(\mathcal C\) first separates the
smaller discs by parity, moves disc \(n\) to \(N_2\), and then moves the smaller
even tower from \(E\) to \(N_2\). This gives
\[
    c_n=b_{n-1}+h_3^{\lfloor (n-1)/2\rfloor}+1.
\]
If \(n\) is odd, the algorithm first separates the discs \(1,\ldots,n-2\),
moves disc \(n-1\) to \(N_2\), transfers the smaller odd tower from \(O\) to
\(N_2\), moves disc \(n\) to \(O\), transfers the smaller odd tower back to
\(O\), and finally transfers the smaller even tower from \(E\) to \(N_2\). This
gives
\[
    c_n
    =
    b_{n-2}
    +2h_3^{\lfloor (n-1)/2\rfloor}
    +h_3^{\lfloor (n-2)/2\rfloor}
    +2.
\]

The formulas for \(d_n\) are obtained analogously. If \(n\) is odd, the
algorithm separates the smaller discs, moves disc \(n\) to \(N_2\), and then
moves the smaller odd tower from \(O\) to \(N_2\). If \(n\) is even, the
algorithm separates the first \(n-2\) discs, moves disc \(n-1\) to \(N_2\),
moves the smaller even tower from \(E\) to \(N_2\), moves disc \(n\) to \(E\),
moves the smaller even tower back to \(E\), and finally moves the smaller odd
tower from \(O\) to \(N_2\). This yields the stated recurrence for \(d_n\).

The final inequalities follow because \(A_n,B_n,C_n,D_n\) are true minimum
values, while \(a_n,b_n,c_n,d_n\) are lengths of explicit feasible algorithms.
\end{proof}

Equivalently, the recursive algorithms can be described by the following
canonical decompositions.

For objective \((a)\),
\[
([n],\varnothing,\varnothing,\varnothing)
\xrightarrow{b_{n-1}}
(\{n\},[n-1]_0,[n-1]_1,\varnothing)
\xrightarrow{1}
(\varnothing,[n-1]_0,[n-1]_1,\{n\})
\xrightarrow{b_{n-1}}
(\varnothing,\varnothing,\varnothing,[n]).
\]

For objective \((b)\), if \(n\) is even, then
\[
([n],\varnothing,\varnothing,\varnothing)
\xrightarrow{c_{n-1}}
(\{n\},\varnothing,[n-1]_1,[n-1]_0)
\xrightarrow{1}
(\varnothing,\{n\},[n-1]_1,[n-1]_0)
\xrightarrow{h_3^{\lfloor (n-1)/2\rfloor}}
(\varnothing,[n]_0,[n]_1,\varnothing).
\]
If \(n\) is odd, then
\[
([n],\varnothing,\varnothing,\varnothing)
\xrightarrow{d_{n-1}}
(\{n\},[n-1]_0,\varnothing,[n-1]_1)
\xrightarrow{1}
(\varnothing,[n-1]_0,\{n\},[n-1]_1)
\xrightarrow{h_3^{\lfloor (n-1)/2\rfloor}}
(\varnothing,[n]_0,[n]_1,\varnothing).
\]

For objective \((c)\), if \(n\) is even, then
\[
([n],\varnothing,\varnothing,\varnothing)
\xrightarrow{b_{n-1}}
(\{n\},[n-1]_0,[n-1]_1,\varnothing)
\xrightarrow{1}
(\varnothing,[n-1]_0,[n-1]_1,\{n\})
\xrightarrow{h_3^{\lfloor (n-1)/2\rfloor}}
(\varnothing,\varnothing,[n]_1,[n]_0).
\]
If \(n\) is odd, then
\[
([n],\varnothing,\varnothing,\varnothing)
\xrightarrow{b_{n-2}}
(\{n,n-1\},[n-2]_0,[n-2]_1,\varnothing)
\xrightarrow{1}
(\{n\},[n-2]_0,[n-2]_1,\{n-1\})
\]
\[
\xrightarrow{h_3^{\lfloor (n-1)/2\rfloor}}
(\{n\},[n-2]_0,\varnothing,[n-2]_1\cup\{n-1\})
\xrightarrow{1}
(\varnothing,[n-2]_0,\{n\},[n-2]_1\cup\{n-1\})
\]
\[
\xrightarrow{h_3^{\lfloor (n-1)/2\rfloor}}
(\varnothing,[n-2]_0,[n]_1,\{n-1\})
\xrightarrow{h_3^{\lfloor (n-2)/2\rfloor}}
(\varnothing,\varnothing,[n]_1,[n]_0).
\]

For objective \((d)\), if \(n\) is odd, then
\[
([n],\varnothing,\varnothing,\varnothing)
\xrightarrow{b_{n-1}}
(\{n\},[n-1]_0,[n-1]_1,\varnothing)
\xrightarrow{1}
(\varnothing,[n-1]_0,[n-1]_1,\{n\})
\xrightarrow{h_3^{\lfloor (n-1)/2\rfloor}}
(\varnothing,[n]_0,\varnothing,[n]_1).
\]
If \(n\) is even, then
\[
([n],\varnothing,\varnothing,\varnothing)
\xrightarrow{b_{n-2}}
(\{n,n-1\},[n-2]_0,[n-2]_1,\varnothing)
\xrightarrow{1}
(\{n\},[n-2]_0,[n-2]_1,\{n-1\})
\]
\[
\xrightarrow{h_3^{\lfloor (n-2)/2\rfloor}}
(\{n\},\varnothing,[n-2]_1,[n-2]_0\cup\{n-1\})
\xrightarrow{1}
(\varnothing,\{n\},[n-2]_1,[n-2]_0\cup\{n-1\})
\]
\[
\xrightarrow{h_3^{\lfloor (n-2)/2\rfloor}}
(\varnothing,[n]_0,[n-2]_1,\{n-1\})
\xrightarrow{h_3^{\lfloor (n-2)/2\rfloor}}
(\varnothing,[n]_0,\varnothing,[n]_1).
\]

\begin{corollary}[Simplified coupled system]
\label{coro:simplified_coupled}
The candidate move counts \(a_n,b_n,c_n,d_n\) satisfy the following
parity-dependent relations:

\begin{align}
b_n &=
\begin{cases}
c_{n-1}+2^{\frac{n-2}{2}}, & n\ \text{even},\\[2pt]
d_{n-1}+2^{\frac{n-1}{2}}, & n\ \text{odd},
\end{cases}\label{eq:bn_simple}\\[4pt]
c_n &=
\begin{cases}
b_{n-1}+2^{\frac{n-2}{2}}, & n\ \text{even},\\[2pt]
b_{n-2}+5\cdot2^{\frac{n-3}{2}}-1, & n\ \text{odd},
\end{cases}\label{eq:cn_simple}\\[4pt]
d_n &=
\begin{cases}
b_{n-2}+3\cdot2^{\frac{n-2}{2}}-1, & n\ \text{even},\\[2pt]
b_{n-1}+2^{\frac{n-1}{2}}, & n\ \text{odd}.
\end{cases}\label{eq:dn_simple}
\end{align}
\end{corollary}

 \begin{proof}
These identities follow from substituting the classical three-peg solution $h_{3}^{m} = 2^{m} - 1$ 
into the recurrences of Proposition~\ref{prop:mainrec} and simplifying. The simplifications rely on the fact that for all integers $n, t \geq 0$, $\left\lfloor \frac{n - 2t}{2} \right\rfloor
=
\left\lfloor \frac{n - (2t + 1)}{2} \right\rfloor$ if $n$ is odd, and $\left\lfloor \frac{n - (2t + 1)}{2} \right\rfloor
=
\left\lfloor \frac{n - (2t + 2)}{2} \right\rfloor$ otherwise. 
\end{proof}
\section{The Main Optimality Conjecture}\label{sec:conjecture}

 The largest disc is known to move exactly once in any optimal solution of the
classical Tower of Hanoi with \(p\) pegs and \(n\) discs; see
\cite[Theorem~5.18]{HinzMythsMaths2018}. More results on the number of largest
disc moves can be found in \cite{Aumann2014}. More recently, Hinz and Parisse
gave a simple and elegant proof of this fact in \cite[Lemma~2]{HinzParisse}.
For the parity-constrained problem considered here, we expect an analogous
property to hold, although the parity restrictions make the usual exchange
arguments more delicate.

\begin{conjecture}[Largest-disc one-move property]
\label{conj:largest-disc-once}
In any optimal solution of objectives
\emph{\textup{\texttt{(a)}}}--\emph{\textup{\texttt{(d)}}}, disc \(n\) moves
exactly once. More precisely:
\begin{itemize}
    \item for objective \emph{\textup{\texttt{(a)}}}, the unique move of disc
    \(n\) is from \(N_1\) to \(N_2\);
    
    \item for objective \emph{\textup{\texttt{(b)}}}, the unique move of disc
    \(n\) is from \(N_1\) to \(p(n)\);
    
    \item for objective \emph{\textup{\texttt{(c)}}}, the unique move of disc
    \(n\) is from \(N_1\) to \(N_2\) if \(n\) is even, and from \(N_1\) to \(O\)
    if \(n\) is odd;
    
    \item for objective \emph{\textup{\texttt{(d)}}}, the unique move of disc
    \(n\) is from \(N_1\) to \(N_2\) if \(n\) is odd, and from \(N_1\) to \(E\)
    if \(n\) is even.
\end{itemize}
\end{conjecture}
 
 Conjecture~\ref{conj:largest-disc-once} is a necessary structural step toward
the main optimality conjecture below. However, it is not sufficient by itself:
one must also prove that, before the unique move of disc \(n\), the smaller
discs occupy the canonical intermediate configurations prescribed by the
recursive algorithms.

\begin{conjecture}[Main optimality conjecture]\label{conj:mainconj}
For every \(n\geq 0\), the recursive algorithms described in Section~\ref{sec:algorithms} are
optimal. Equivalently, $A_n=a_n$, $B_n=b_n$, $C_n=c_n$, and  $D_n=d_n$.
\end{conjecture}

The main optimality conjecture is supported by exhaustive shortest-path
computations in the associated state graph for small values of \(n\). In all
computed cases, the true distances \(A_n,B_n,C_n,D_n\) agree with the candidate
values \(a_n,b_n,c_n,d_n\) obtained from the recursive algorithms. This
computational evidence suggests that the proposed recursive decompositions
capture the correct metric structure of the parity-constrained problem.

The main difficulty is to prove that the canonical intermediate configurations
used in Algorithms~\ref{alg:A}--\ref{alg:D} are unavoidable in every shortest solution. For example,
in the full-tower transfer, Algorithm~\ref{alg:A} first separates the \(n-1\) smaller
discs by parity, moves disc \(n\) once from \(N_1\) to \(N_2\), and then reverses
the separation procedure. To prove optimality, one would need to show that every
shortest solution can be transformed into one having this structure, or at least
that no shorter solution can avoid such a separation stage.

More generally, a proof of Conjecture~\ref{conj:mainconj} would likely
require establishing the following structural facts. First, one needs to establish Conjecture~\ref{conj:largest-disc-once}, namely
that the largest disc moves exactly once in a shortest solution for each of the
four objectives. 
Second, immediately before the move of the largest disc, the smaller discs
should occupy the canonical parity-separated configuration prescribed by the
corresponding recursive algorithm. Third, in the mixed-parity cases of
Algorithms~\ref{alg:C} and~\ref{alg:D}, the interaction between discs \(n\) and \(n-1\) should force
the additional intermediate transfers appearing in the algorithm.

These requirements are nontrivial. A local exchange argument is not enough,
because moving the largest disc may change which pegs are available as bases for
some smaller discs, while other smaller discs are forbidden from using that peg
by parity. Thus, the question is not only whether a certain move can be delayed
or advanced, but whether the whole configuration of smaller discs can be
rearranged without increasing the length of the solution. This is reminiscent of
the difficulty in the Reve's puzzle and the Frame--Stewart algorithm, where a
natural recursive strategy does not by itself imply optimality.

\begin{remark}[Non-uniqueness of shortest paths]
The shortest path between two prescribed states is not necessarily unique.
For instance, for objective \texttt{(c)} with \(n=3\), exhaustive computation in the
state graph gives $ C_3=5$. There are at least two shortest paths of length \(5\), namely
\[
1:N_1\to O,\quad
2:N_1\to N_2,\quad
1:O\to N_2,\quad
3:N_1\to O,\quad
1:N_2\to O,
\]
and
\[
1:N_1\to N_2,\quad
2:N_1\to E,\quad
3:N_1\to O,\quad
1:N_2\to O,\quad
2:E\to N_2.
\]
This shows that even when the candidate value agrees with the true optimum,
the corresponding shortest path need not be unique. More about the number of optimal solution is given in Section~\ref{sec:NumberofOptimalSolutions}.

\end{remark}

 \section{Derived Recurrences and Explicit Formulas}
\label{sec:DerivedRecurrences}

The coupled system of recursive relations in Proposition~\ref{prop:mainrec} can be algebraically
combined to yield single higher-order recurrences for each of the four candidate move sequences 
$a_{n}$, $b_{n}$, $c_{n}$, and $d_{n}$. In this section, we derive these recurrences and list their
initial values. We then provide explicit closed-form expressions for each sequence.

\begin{proposition}[Higher-order recurrences]
\label{prop:higher}
The candidate move sequences $a_n, b_n, c_n, d_n$ defined in 
Proposition~\ref{prop:mainrec} satisfy the following higher-order recurrences:
\begin{align}
&\begin{split}
&a_0=0,\; a_1=1,\;a_2=3,\\ &a_3=5,
\end{split} &\qquad\forall n\geq 4:
a_n=
\begin{cases}
a_{n-3}+5\cdot2^{\frac{n-2}{2}}-2, & n\ \text{even},\\[2pt]
a_{n-3}+7\cdot2^{\frac{n-3}{2}}-2, & n\ \text{odd};
\end{cases}\label{eq:an_high}\\
&\begin{split}
&b_0=0,\; b_1=1,\; b_2=2,
\end{split}&\qquad\forall n\geq 3:
b_n=
\begin{cases}
b_{n-3}+7\cdot2^{\frac{n-4}{2}}-1, & n\ \text{even},\\[2pt]
b_{n-3}+5\cdot2^{\frac{n-3}{2}}-1, & n\ \text{odd};
\end{cases}\label{eq:bn_high}\\
&\begin{split}
&c_0=0,\; c_1=1,\; c_2=2,\\ &c_3=5,\; c_4=6,\; c_5=13,
\end{split}&\qquad\forall n\geq 6:
c_n=
\begin{cases}
c_{n-5}+15\cdot2^{\frac{n-6}{2}}-1, & n\ \text{even},\\[2pt]
c_{n-6}+31\cdot2^{\frac{n-7}{2}}-2, & n\ \text{odd};
\end{cases}\label{eq:cn_high}\\
&\begin{split}
&d_0=0,\; d_1=1,\; d_2=2,\\ &d_3=4,\; d_4=7,\; d_5=11,
\end{split}&\qquad\forall n\geq 6:
d_n=
\begin{cases}
d_{n-6}+5\cdot2^{\frac{n-2}{2}}-2, & n\ \text{even},\\[2pt]
d_{n-5}+3\cdot2^{\frac{n-1}{2}}-1, & n\ \text{odd}.
\end{cases}\label{eq:dn_high}
\end{align}

\end{proposition}

\begin{proof}
We first derive the recurrence for $b_n$ using the simplified coupled system of
Corollary~\ref{coro:simplified_coupled}:
\[
b_{n}=
\begin{cases}
c_{n-1}+2^{\frac{n-2}{2}}, & \text{$n$ even},\\[6pt]
d_{n-1}+2^{\frac{n-1}{2}}, & \text{$n$ odd}.
\end{cases}
\]
Substituting the expressions of $c_{n-1}$ and $d_{n-1}$ from the same corollary and applying the parity-dependent identities for floor functions from its proof yield
\[
b_n=
\begin{cases}
b_{n-3}+7 \cdot 2^{\frac{n-4}{2}} - 1, & \text{$n$ even},\\[6pt]
b_{n-3}+5 \cdot 2^{\frac{n-3}{2}} - 1, & \text{$n$ odd},
\end{cases}
\]
which establishes~\eqref{eq:bn_high}. The initial conditions follow directly from Corollary~\ref{coro:simplified_coupled}.

Having determined $b_n$, we derive the recurrence for $a_n$ from
\[
a_n = 2b_{n-1} + 1.
\]
By applying~\eqref{eq:bn_high} to $b_{n-1}$ and noting that $n-1$ has opposite parity to $n$, we obtain
\[
a_n=
\begin{cases}
a_{n-3}+5 \cdot 2^{\frac{n-2}{2}} - 2, & \text{$n$ even},\\[6pt]
a_{n-3}+7 \cdot 2^{\frac{n-3}{2}} - 2, & \text{$n$ odd},
\end{cases}
\]
using $a_{n-3} = 2b_{n-4} + 1$ to eliminate $b_{n-4}$. This proves~\eqref{eq:an_high}.

The recurrences for $c_n$ and $d_n$ are obtained analogously by substituting from
Corollary~\ref{coro:simplified_coupled} and then eliminating shifted $b$-terms via~\eqref{eq:bn_high}.
Their initial values follow from explicit computation for $n \le 5$.

Hence, all four higher-order recurrences are established.\qedhere
\end{proof}

The first fifteen terms of each of the six sequences $h_{3}^{n}$, $h_{4}^{n}$, $a_{n}$, $b_{n}$, $c_{n}$, and $d_{n}$ are listed in Table~\ref{tab:firstvalues}. These values confirm the consistency of the derived recurrences with known results for the classical three-peg and four-peg Tower of Hanoi problems.

\begin{table}
\caption{First fifteen terms of the sequences $h_{3}^{n}$, $h_{4}^{n}$, $a_{n}$, $b_{n}$, $c_{n}$, and $d_{n}$.}
\label{tab:firstvalues}
 
\begin{tabular*}{\linewidth}{@{\extracolsep{\fill}}l|lllllllllllllll@{}}
\toprule
$n$ & 0 & 1 & 2 & 3 & 4 & 5 & 6 & 7 & 8 & 9 & 10 & 11 & 12 & 13 & 14\\ \midrule
$h_{3}^{n}$ & 0 & 1 & 3 & 7 & 15 & 31 & 63 & 127 & 255 & 511 & 1023 & 2047 & 4095 & 8191 & 16383\\
$a_n$      & 0 & 1 & 3 & 5 & 9 & 15 & 23 & 35 & 53 & 77 & 113 & 163 & 235 & 335 & 481\\
$b_n$      & 0 & 1 & 2 & 4 & 7 & 11 & 17 & 26 & 38 & 56 & 81 & 117 & 167 & 240 & 340\\
$c_n$      & 0 & 1 & 2 & 5 & 6 & 13 & 15 & 30 & 34 & 65 & 72 & 135 & 149 & 276 & 304\\
$d_n$      & 0 & 1 & 2 & 4 & 7 & 11 & 18 & 25 & 40 & 54 & 85 & 113 & 176 & 231 & 358\\
$h_{4}^{n}$ & 0 & 1 & 3 & 5 & 9 & 13 & 17 & 25 & 33 & 41 & 49 & 65 & 81 & 97 & 113\\ \bottomrule
\end{tabular*}
\end{table}

\begin{proposition}[Closed-form expressions]
\label{prop:closedform}
For all integers $n \geq 0$, let
\[
\rho(n) \coloneqq n \bmod 3,
\qquad
\theta(n) \coloneqq \frac{n - \rho(n)}{3}.
\]
Then the four sequences admit the following closed forms:
{\footnotesize

\begin{align}
\label{eq:a-closed}
a_{n}
&=
\begin{cases}
a_{\rho(n)} - 2\,\theta(n)
\;+\; 2^{\frac{n}{2}}
\!\left(
\Bigl(1 - 2^{-3\left\lfloor \frac{\theta(n)}{2}\right\rfloor}\Bigr)
\;+\;
\dfrac{20}{7}
\Bigl(1 - 2^{-3\left\lfloor \frac{\theta(n)-1}{2}\right\rfloor + 1}\Bigr)
\right),
& \text{$n$ even},\\[8pt]
a_{\rho(n)} - 2\,\theta(n)
\;+\; 2^{\frac{n+1}{2}}
\!\left(
\dfrac{5}{7}\Bigl(1 - 2^{-3\left\lfloor \frac{\theta(n)}{2}\right\rfloor}\Bigr)
\;+\;
2\Bigl(1 - 2^{-3\left\lfloor \frac{\theta(n)-1}{2}\right\rfloor + 1}\Bigr)
\right),
& \text{$n$ odd},
\end{cases}
\\[10pt]
\label{eq:b-closed}
b_{n}
&=
\begin{cases}
\dfrac{a_{\rho(n+1)} - 1}{2} - \theta(n+1)
\;+\; 2^{\frac{n}{2}}
\!\left(
\dfrac{5}{7}\Bigl(1 - 2^{-3\left\lfloor \frac{\theta(n+1)}{2}\right\rfloor}\Bigr)
\;+\;
2\Bigl(1 - 2^{-3\left\lfloor \frac{\theta(n+1)-1}{2}\right\rfloor + 1}\Bigr)
\right),
& \text{$n$ even},\\[10pt]
\dfrac{a_{\rho(n+1)} - 1}{2} - \theta(n+1)
\;+\; 2^{\frac{n-1}{2}}
\!\left(
\Bigl(1 - 2^{-3\left\lfloor \frac{\theta(n+1)}{2}\right\rfloor}\Bigr)
\;+\;
\dfrac{20}{7}\Bigl(1 - 2^{-3\left\lfloor \frac{\theta(n+1)-1}{2}\right\rfloor + 1}\Bigr)
\right),
& \text{$n$ odd},
\end{cases}
\end{align}
}

{\footnotesize

\begin{align}
\label{eq:c-closed}
c_{n}
&=
\begin{cases}
\dfrac{a_{\rho(n)} - 1}{2} - \theta(n)
\;+\; 2^{\frac{n-2}{2}}
\!\left(
2 - 2^{-3\left\lfloor \frac{\theta(n)}{2}\right\rfloor}
\;+\;
\dfrac{20}{7}
\Bigl(1 - 2^{-3\left\lfloor \frac{\theta(n)-1}{2}\right\rfloor + 1}\Bigr)
\right),
& \text{$n$ even},\\[10pt]
\dfrac{a_{\rho(n-1)} - 1}{2} - \theta(n-1)
\;+\; 2^{\frac{n-3}{2}}
\!\left(
6 - 2^{-3\left\lfloor \frac{\theta(n-1)}{2}\right\rfloor}
\;+\;
\dfrac{20}{7}
\Bigl(1 - 2^{-3\left\lfloor \frac{\theta(n-1)-1}{2}\right\rfloor + 1}\Bigr)
\right)
- 1,
& \text{$n$ odd},
\end{cases}
\\[10pt]
\label{eq:d-closed}
d_{n}
&=
\begin{cases}
\dfrac{a_{\rho(n-1)} - 1}{2} - \theta(n-1)
\;+\; 2^{\frac{n-2}{2}}
\!\left(
\dfrac{5}{7}\Bigl(1 - 2^{-3\left\lfloor \frac{\theta(n-1)}{2}\right\rfloor}\Bigr)
\;+\;
2\Bigl(1 - 2^{-3\left\lfloor \frac{\theta(n-1)-1}{2}\right\rfloor + 1}\Bigr)
+ 3
\right)
- 1,
& \text{$n$ even},\\[10pt]
\dfrac{a_{\rho(n)} - 1}{2} - \theta(n)
\;+\; 2^{\frac{n-1}{2}}
\!\left(
\dfrac{5}{7}\Bigl(1 - 2^{-3\left\lfloor \frac{\theta(n)}{2}\right\rfloor}\Bigr)
\;+\;
2\Bigl(1 - 2^{-3\left\lfloor \frac{\theta(n)-1}{2}\right\rfloor + 1}\Bigr)
+ 1
\right),
& \text{$n$ odd}.
\end{cases}
\end{align}
}
\end{proposition}

\begin{proof}
The result can be proved by induction on $n$, using the higher-order recurrences in Proposition~\ref{prop:higher} and the definitions of $\rho(n)$ and $\theta(n)$.\qedhere
\end{proof}

\subsection{An alternative formulation of the recurrences and explicit expressions}\label{subsec:alternative_formulation}

The closed-form expressions for the sequences \((a_n)\), \((b_n)\), \((c_n)\), and \((d_n)\) given in Propositions~\ref{prop:higher} and \ref{prop:closedform} may appear rather intricate and difficult to interpret. We therefore present here an alternative formulation of the recurrence system of Corollary~\ref{coro:simplified_coupled}, which makes the underlying structure of the four sequences more transparent.

\begin{lemma}[Decoupled parity formulation]\label{lem:decoupled-parity-form}
We have
\begin{align}
&a_{2n}   = 2b_{2n-1}+1,  
&&a_{2n+1} = 2b_{2n}+1, &(n\geq 1)\label{eq:a-decoupled}\\
&b_{2n}    = b_{2(n-3)}+19\cdot 2^{n-3}-2, &&b_{2n+1}  = b_{2(n-3)+1}+27\cdot 2^{n-3}-2, &(n\geq 3) \label{eq:b-decoupled}\\
&c_{2n}    = b_{2n-1}+2^{n-1}, &&c_{2n+1}  = b_{2n-1}+5\cdot 2^{n-1}-1, &(n\geq 1)\label{eq:c-decoupled}\\
&d_{2n}    = b_{2n-2}+3\cdot 2^{n-1}-1, &&d_{2n+1}  = b_{2n}+2^n. &(n\geq 1)\label{eq:d-decoupled}
\end{align}
\end{lemma}
\begin{proof}
Substituting  \(2n\) and \(2n+1\) for $n$ into \eqref{eq:a}, \eqref{eq:cn_simple}, and \eqref{eq:dn_simple} yields \eqref{eq:a-decoupled},  \eqref{eq:c-decoupled}, and \eqref{eq:d-decoupled}, respectively.  Now, we focus on \eqref{eq:b-decoupled}. Substituting   \(2n\) and \(2n+1\) for $n$ in \eqref{eq:bn_simple} yields
\begin{equation}
 b_{2n}  = c_{2n-1}+2^{n-1}, \quad b_{2n+1}  = d_{2n}+2^n = b_{2n-2}+5\cdot 2^{n-1}-1. \label{eq:b-first-step}
\end{equation}
The final identity of the above expression for \(b_{2n+1}\), is obtained   using   the even-case formula for \(d_{2n}\) in \eqref{eq:dn_simple}.

Substituting the second equation in \eqref{eq:c-decoupled} for $n - 1$ instead of $n$ in the first equation of \eqref{eq:b-first-step}, it
follows that
\begin{equation}
    b_{2n}=c_{2n-1}+2^{n-1}
      =\bigl(b_{2n-3}+5\cdot 2^{n-2}-1\bigr)+2^{n-1}
      =b_{2n-3}+7\cdot 2^{n-2}-1.\label{eq:b-even-second-step}
\end{equation}
Inserting the second equation of \eqref{eq:b-first-step} for $n - 2$ instead of $n$ in \eqref{eq:b-even-second-step}, we get
\[
b_{2n}=(b_{2n-6}+5\cdot 2^{n-3}-1)+7\cdot 2^{n-2}-1= b_{2n-6}+19\cdot 2^{n-3}-2,
\]
while inserting \eqref{eq:b-even-second-step} for $n - 1$ instead of $n$ in the second equation of \eqref{eq:b-first-step} yields
\[b_{2n+1}=b_{2n-2}+5\cdot 2^{n-1}-1
        =\bigl(b_{2n-5}+7\cdot 2^{n-3}-1\bigr)+5\cdot 2^{n-1}-1
        =b_{2n-5}+27\cdot 2^{n-3}-2.\qedhere\]
\end{proof}

Lemma \ref{lem:decoupled-parity-form} shows that once the sequence \((b_n)\) is known, the three other sequences follow immediately. In this sense, the second objective is the fundamental one. Ultimately, this analysis shows that it is sufficient to consider only the second objective \textup{\texttt{(b)}}.

By setting $\alpha_n:=b_{2n}$ and $\beta_n:=b_{2n+1}$, $n\ge 0$, we obtain the third-order linear inhomogeneous recurrences from the two equations in \eqref{eq:b-decoupled}, 
\begin{align}
\alpha_0&=0,\quad \alpha_1=2,\quad \alpha_2=7,  &\forall n\ge 3:
 \alpha_n=\alpha_{n-3}+19\cdot 2^{n-3}-2, \label{eq:alpha-third-order}\\
\beta_0&=1,\quad \beta_1=4,\quad \beta_2=11, &\forall n\ge 3:
 \beta_n=\beta_{n-3}+27\cdot 2^{n-3}-2. \label{eq:beta-third-order}
\end{align}

\begin{theorem}[Explicit formulas for \( (b_n) \)]\label{thm:b-mod6}
For every integer \(m\ge 0\), the sequence \( (b_n) \) satisfies
\begin{align}
b_{6m}   &= \frac{19}{7}(8^m-1)-2m, \label{eq:b6m}\\
b_{6m+1} &= \frac{27}{7}(8^m-1)+1-2m, \label{eq:b6m1}\\
b_{6m+2} &= \frac{38}{7}(8^m-1)+2-2m, \label{eq:b6m2}\\
b_{6m+3} &= \frac{54}{7}(8^m-1)+4-2m, \label{eq:b6m3}\\
b_{6m+4} &= \frac{76}{7}(8^m-1)+7-2m, \label{eq:b6m4}\\
b_{6m+5} &= \frac{108}{7}(8^m-1)+11-2m. \label{eq:b6m5}
\end{align}
\end{theorem}

\begin{proof}
We solve separately the three residue classes of \((\alpha_n)\) modulo \(3\), and then those of \((\beta_n)\). From \eqref{eq:alpha-third-order},
\[
\alpha_{3m}=\alpha_{3m-3}+19\cdot 2^{3m-3}-2
           =\alpha_{3(m-1)}+19\cdot 8^{m-1}-2
\qquad (m\ge 1),
\]
with initial value \(\alpha_0=0\). Iterating, we obtain
\[
\alpha_{3m}
= \sum_{j=0}^{m-1}\bigl(19\cdot 8^j-2\bigr)
= 19\sum_{j=0}^{m-1}8^j-2m
= 19\cdot \frac{8^m-1}{7}-2m.
\]
Since \(\alpha_{3m}=b_{6m}\), this gives \eqref{eq:b6m}. Again by \eqref{eq:alpha-third-order},
\[
\alpha_{3m+1}=\alpha_{3m-2}+19\cdot 2^{3m-2}-2
             =\alpha_{3(m-1)+1}+38\cdot 8^{m-1}-2
\qquad (m\ge 1),
\]
with initial value \(\alpha_1=2\). Therefore,
\[
\alpha_{3m+1}
=2+\sum_{j=0}^{m-1}\bigl(38\cdot 8^j-2\bigr)
=2+38\cdot \frac{8^m-1}{7}-2m.
\]
Since \(\alpha_{3m+1}=b_{6m+2}\), we obtain \eqref{eq:b6m2}. Similarly,
\[
\alpha_{3m+2}=\alpha_{3(m-1)+2}+76\cdot 8^{m-1}-2
\qquad (m\ge 1),
\]
with initial value \(\alpha_2=7\). Hence
\[
\alpha_{3m+2}
=7+\sum_{j=0}^{m-1}\bigl(76\cdot 8^j-2\bigr)
=7+76\cdot \frac{8^m-1}{7}-2m.
\]
Since \(\alpha_{3m+2}=b_{6m+4}\), this proves \eqref{eq:b6m4}. From \eqref{eq:beta-third-order},
\[
\beta_{3m}=\beta_{3(m-1)}+27\cdot 8^{m-1}-2
\qquad (m\ge 1),
\]
with \(\beta_0=1\). Therefore,
\[
\beta_{3m}
=1+\sum_{j=0}^{m-1}\bigl(27\cdot 8^j-2\bigr)
=1+27\cdot \frac{8^m-1}{7}-2m.
\]
Since \(\beta_{3m}=b_{6m+1}\), this proves \eqref{eq:b6m1}. Likewise,
\[
\beta_{3m+1}=\beta_{3(m-1)+1}+54\cdot 8^{m-1}-2
\qquad (m\ge 1),
\]
with \(\beta_1=4\). Thus
\[
\beta_{3m+1}
=4+\sum_{j=0}^{m-1}\bigl(54\cdot 8^j-2\bigr)
=4+54\cdot \frac{8^m-1}{7}-2m.
\]
Since \(\beta_{3m+1}=b_{6m+3}\), this proves \eqref{eq:b6m3}. Finally,
\[
\beta_{3m+2}=\beta_{3(m-1)+2}+108\cdot 8^{m-1}-2
\qquad (m\ge 1),
\]
with \(\beta_2=11\). Hence
\[
\beta_{3m+2}
=11+\sum_{j=0}^{m-1}\bigl(108\cdot 8^j-2\bigr)
=11+108\cdot \frac{8^m-1}{7}-2m.
\]
Since \(\beta_{3m+2}=b_{6m+5}\), this proves \eqref{eq:b6m5}.\qedhere
\end{proof}

By substituting \eqref{eq:b6m}--\eqref{eq:b6m5}  into \eqref{eq:a-decoupled},  \eqref{eq:c-decoupled}, and  \eqref{eq:d-decoupled} we obtain
similar solutions for the subsequences of $(a_n)$, $(c_n)$, and $(d_n)$. These formulas all have the same structure, namely 
$A\cdot(8^n-1)+B-C\cdot n$,  
where the coefficients $A$, $B$, and $C$ are different for each subsequence.

\begin{corollary}[Explicit formulas for \( (c_n) \), \( (d_n) \), and \( (a_n) \)]\label{cor:acd-mod6}
For every integer \(m\ge 0\), the remaining three sequences are given by
\begin{align}
a_{6m}   &= \frac{27}{7}(8^m-1)-4m, \\
a_{6m+1} &= \frac{38}{7}(8^m-1)+1-4m, \\
a_{6m+2} &= \frac{54}{7}(8^m-1)+3-4m, \\
a_{6m+3} &= \frac{76}{7}(8^m-1)+5-4m, \\
a_{6m+4} &= \frac{108}{7}(8^m-1)+9-4m, \\
a_{6m+5} &= \frac{152}{7}(8^m-1)+15-4m,
\end{align}
\begin{align}
c_{6m}   &= \frac{17}{7}(8^m-1)-2m, \\
c_{6m+1} &= \frac{31}{7}(8^m-1)+1-2m, \\
c_{6m+2} &= \frac{34}{7}(8^m-1)+2-2m, \\
c_{6m+3} &= \frac{62}{7}(8^m-1)+5-2m, \\
c_{6m+4} &= \frac{68}{7}(8^m-1)+6-2m, \\
c_{6m+5} &= \frac{124}{7}(8^m-1)+13-2m,
\end{align}
\begin{align}
d_{6m}   &= \frac{20}{7}(8^m-1)-2m, \\
d_{6m+1} &= \frac{26}{7}(8^m-1)+1-2m, \\
d_{6m+2} &= \frac{40}{7}(8^m-1)+2-2m, \\
d_{6m+3} &= \frac{52}{7}(8^m-1)+4-2m, \\
d_{6m+4} &= \frac{80}{7}(8^m-1)+7-2m, \\
d_{6m+5} &= \frac{104}{7}(8^m-1)+11-2m.
\end{align}
\end{corollary}

\begin{proof}
These formulas follow directly from Lemma~\ref{lem:decoupled-parity-form} and Theorem~\ref{thm:b-mod6}.

For instance,
\[
a_{6m}=2b_{6m-1}+1=2b_{6(m-1)+5}+1
\]
for \(m\ge 1\). Using \eqref{eq:b6m5}, we obtain
\[
a_{6m}
=2\left(\frac{108}{7}(8^{m-1}-1)+11-2(m-1)\right)+1
=\frac{27}{7}(8^m-1)-4m.
\]
The case \(m=0\) gives \(a_0=0\), so the formula remains valid for all \(m\ge 0\).

Likewise,
\[
c_{6m+3}=b_{6m+1}+5\cdot 2^{3m}-1
\]
by the second equation in \eqref{eq:c-decoupled}, whence
\[
c_{6m+3}
=\left(\frac{27}{7}(8^m-1)+1-2m\right)+5\cdot 8^m-1
=\frac{62}{7}(8^m-1)+5-2m.
\]
Similarly,
\[
d_{6m+4}=b_{6m+2}+3\cdot 2^{3m+1}-1
\]
by the first equation in \eqref{eq:d-decoupled}, and therefore
\[
d_{6m+4}
=\left(\frac{38}{7}(8^m-1)+2-2m\right)+6\cdot 8^m-1
=\frac{80}{7}(8^m-1)+7-2m.
\]

All remaining identities are obtained in the same way.\qedhere
\end{proof}

\section{Asymptotic Growth Analysis}
\label{sec:AsymptoticGrowth}

We now study the relative growth behavior of the four candidate move sequences 
$a_n$, $b_n$, $c_n$, and $d_n$, and compare them with the classical three-peg and four-peg Tower of Hanoi sequences $h_3^n$ and $h_4^n$, respectively.

\begin{proposition}[Comparative bounds]
\label{prop:bounds}
For all $n \geq 0$, the following hold:
\begin{itemize}
    \item For all $n\geq 0$, 
    \begin{equation}\label{eq:bounds1}
h_4^n \le A_n \le a_n \le h_3^n.
\end{equation}
\item For all $n\geq 0$, 
\begin{equation}\label{eq:bounds}
 \max\{b_n,c_n,d_n\} \le  a_n .
\end{equation}

\item Moreover, the sequences \(b_n,c_n,d_n\) remain close to one another,
and none dominates the others for all~\(n\geq 0\).
\end{itemize}
\end{proposition}

\begin{proof}
The inequality $a_n \le h^{n}_3$ can be proved by induction on $n$ using the recurrences in Proposition~\ref{prop:higher}. Moreover, since the parity constraints can only reduce the set of feasible moves compared with the unconstrained four-peg problem, we also have $a_n \ge h^{n}_4$.

The inequality $\max\{b_n,c_n,d_n\} \le a_n$ can be proved by checking the sign of the difference $a_n - x_n$, for each $x \in \{b,c,d\}$, using the explicit formulas in Proposition~\ref{prop:closedform}.

The non-domination follows from inspection of Table~\ref{tab:firstvalues}; e.g., for $n=6$ we have $d_{6} > b_{6} > c_{6}$, while for $n=7$ we have $c_{7} > b_{7} > d_{7}$.\qedhere
\end{proof}

\begin{conjecture}\label{conj:bounds}
Based on the numerical evidence reported in Table~\ref{tab:firstvalues}, we conjecture that
\[
\min\{b_n, c_n, d_n\} \geq h_4^{n} \qquad \text{for all } n \geq 7.
\]
\end{conjecture}



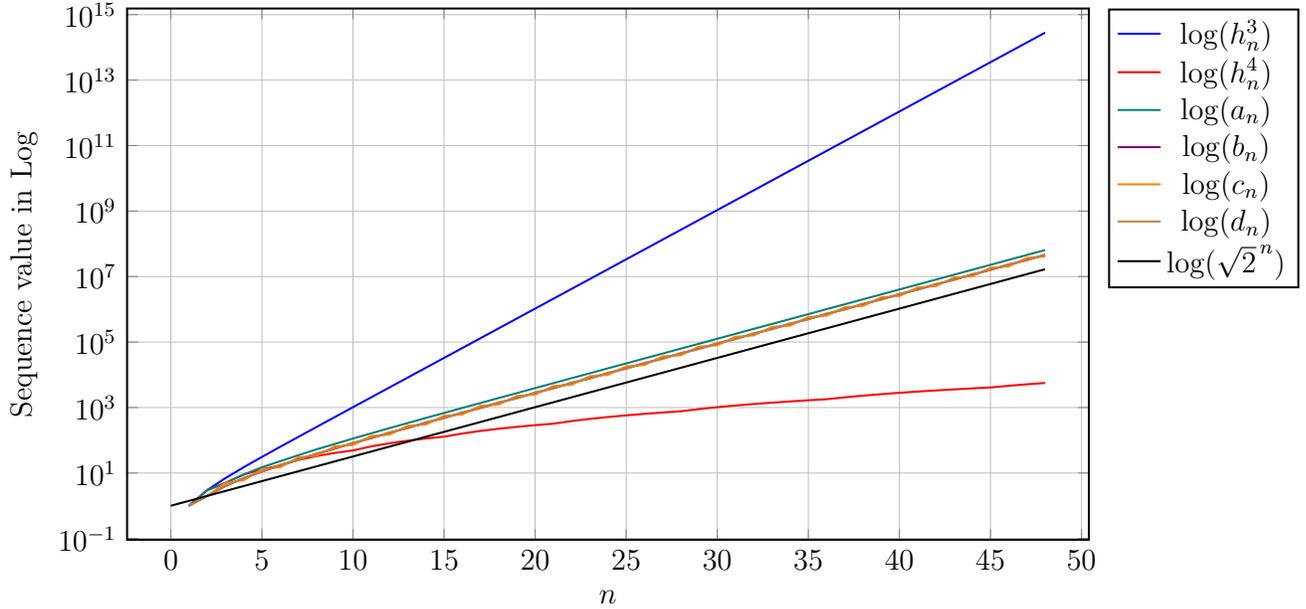
\begin{figure}[ht]
    \centering

\resizebox{\columnwidth}{!}{

\begin{tikzpicture}
\begin{axis}[
    width=15cm,
    height=9cm,
    xlabel={$n$},
    ylabel={Sequence value in Log},
    legend style={at={(1.02,1)},anchor=north west},
    grid=both,
    ymode=log,
    log basis y=10,
    ymin=0.5,
    enlargelimits=0.05,
    thick,
]

\addplot+[mark=, mark size=1pt, blue]
table[row sep=\\] {
n val\\
0 0\\ 1 1\\ 2 3\\ 3 7\\ 4 15\\ 5 31\\ 6 63\\ 7 127\\ 8 255\\ 9 511\\
10 1023\\ 11 2047\\ 12 4095\\ 13 8191\\ 14 16383\\ 15 32767\\ 16 65535\\
17 131071\\ 18 262143\\ 19 524287\\ 20 1048575\\ 21 2097151\\ 22 4194303\\
23 8388607\\ 24 16777215\\ 25 33554431\\ 26 67108863\\ 27 134217727\\
28 268435455\\ 29 536870911\\ 30 1073741823\\ 31 2147483647\\ 32 4294967295\\
33 8589934591\\ 34 17179869183\\ 35 34359738367\\ 36 68719476735\\
37 137438953471\\ 38 274877906943\\ 39 549755813887\\ 40 1099511627775\\
41 2199023255551\\ 42 4398046511103\\ 43 8796093022207\\ 44 17592186044415\\
45 35184372088831\\ 46 70368744177663\\ 47 140737488355327\\ 48 281474976710655\\
};
\addlegendentry{$\log(h_{n}^3)$}

\addplot+[mark=, mark size=1pt, red]
table[row sep=\\] {
n val\\
0 0\\ 1 1\\ 2 3\\ 3 5\\ 4 9\\ 5 13\\ 6 17\\ 7 25\\ 8 33\\ 9 41\\
10 49\\ 11 65\\ 12 81\\ 13 97\\ 14 113\\ 15 129\\ 16 161\\ 17 193\\
18 225\\ 19 257\\ 20 289\\ 21 321\\ 22 385\\ 23 449\\ 24 513\\ 25 577\\
26 641\\ 27 705\\ 28 769\\ 29 897\\ 30 1025\\ 31 1153\\ 32 1281\\ 33 1409\\
34 1537\\ 35 1665\\ 36 1793\\ 37 2049\\ 38 2305\\ 39 2561\\ 40 2817\\
41 3073\\ 42 3329\\ 43 3585\\ 44 3841\\ 45 4097\\ 46 4609\\ 47 5121\\ 48 5633\\
};
\addlegendentry{$\log(h_{n}^4)$}

\addplot+[mark=, mark size=1.5pt, teal]
table[row sep=\\] {
n val\\
0 0\\ 1 1\\ 2 3\\ 3 5\\ 4 9\\ 5 15\\ 6 23\\ 7 35\\ 8 53\\ 9 77\\ 10 113\\ 11 163\\
12 235\\ 13 335\\ 14 481\\ 15 681\\ 16 973\\ 17 1375\\ 18 1959\\ 19 2763\\
20 3933\\ 21 5541\\ 22 7881\\ 23 11099\\ 24 15779\\ 25 22215\\ 26 31577\\
27 44449\\ 28 63173\\ 29 88919\\ 30 126367\\ 31 177859\\ 32 252757\\ 33 355741\\
34 505537\\ 35 711507\\ 36 1011099\\ 37 1423039\\ 38 2022225\\ 39 2846105\\
40 4044477\\ 41 5692239\\ 42 8088983\\ 43 11384507\\ 44 16177997\\ 45 22769045\\
46 32356025\\ 47 45538123\\ 48 64712083\\
};
\addlegendentry{$\log(a_{n})$}

\addplot+[mark=, mark size=1.5pt, violet]
table[row sep=\\] {
n val\\
0 0\\ 1 1\\ 2 2\\ 3 4\\ 4 7\\ 5 11\\ 6 17\\ 7 26\\ 8 38\\ 9 56\\ 10 81\\ 11 117\\
12 167\\ 13 240\\ 14 340\\ 15 486\\ 16 687\\ 17 979\\ 18 1381\\ 19 1966\\ 20 2770\\
21 3940\\ 22 5549\\ 23 7889\\ 24 11107\\ 25 15788\\ 26 22224\\ 27 31586\\ 28 44459\\
29 63183\\ 30 88929\\ 31 126378\\ 32 177870\\ 33 252768\\ 34 355753\\ 35 505549\\
36 711519\\ 37 1011112\\ 38 1423052\\ 39 2022238\\ 40 2846119\\ 41 4044491\\
42 5692253\\ 43 8088998\\ 44 11384522\\ 45 16178012\\ 46 22769061\\ 47 32356041\\
48 45538139\\
};
\addlegendentry{$\log(b_{n})$}

\addplot+[mark=, mark size=1.5pt, orange]
table[row sep=\\] {
n val\\
0 0\\ 1 1\\ 2 2\\ 3 5\\ 4 6\\ 5 13\\ 6 15\\ 7 30\\ 8 34\\ 9 65\\ 10 72\\ 11 135\\
12 149\\ 13 276\\ 14 304\\ 15 559\\ 16 614\\ 17 1125\\ 18 1235\\ 19 2258\\
20 2478\\ 21 4525\\ 22 4964\\ 23 9059\\ 24 9937\\ 25 18128\\ 26 19884\\ 27 36267\\
28 39778\\ 29 72545\\ 30 79567\\ 31 145102\\ 32 159146\\ 33 290217\\ 34 318304\\
35 580447\\ 36 636621\\ 37 1160908\\ 38 1273256\\ 39 2321831\\ 40 2546526\\
41 4643677\\ 42 5093067\\ 43 9287370\\ 44 10186150\\ 45 18574757\\ 46 20372316\\
47 37149531\\ 48 40744649\\
};
\addlegendentry{$\log(c_{n})$}

\addplot+[mark=, mark size=1.5pt, brown, solid]
table[row sep=\\] {
n val\\
0 0\\ 1 1\\ 2 2\\ 3 4\\ 4 7\\ 5 11\\ 6 18\\ 7 25\\ 8 40\\ 9 54\\ 10 85\\ 11 113\\
12 176\\ 13 231\\ 14 358\\ 15 468\\ 16 723\\ 17 943\\ 18 1454\\ 19 1893\\ 20 2916\\
21 3794\\ 22 5841\\ 23 7597\\ 24 11692\\ 25 15203\\ 26 23394\\ 27 30416\\ 28 46799\\
29 60843\\ 30 93610\\ 31 121697\\ 32 187232\\ 33 243406\\ 34 374477\\ 35 486825\\
36 748968\\ 37 973663\\ 38 1497950\\ 39 1947340\\ 40 2995915\\ 41 3894695\\
42 5991846\\ 43 7789405\\ 44 11983708\\ 45 15578826\\ 46 23967433\\ 47 31157669\\
48 47934884\\
};
\addlegendentry{$\log(d_{n})$}


\addplot+[mark=, mark size=1.5pt, black, solid]
table[row sep=\\] {
n val\\
0 1.000000\\
1 1.414214\\
2 2.000000\\
3 2.828427\\
4 4.000000\\
5 5.656854\\
6 8.000000\\
7 11.313708\\
8 16.000000\\
9 22.627417\\
10 32.000000\\
11 45.254834\\
12 64.000000\\
13 90.509668\\
14 128.000000\\
15 181.019336\\
16 256.000000\\
17 362.038672\\
18 512.000000\\
19 724.077344\\
20 1024.000000\\
21 1448.154688\\
22 2048.000000\\
23 2896.309376\\
24 4096.000000\\
25 5792.618752\\
26 8192.000000\\
27 11585.237504\\
28 16384.000000\\
29 23170.475008\\
30 32768.000000\\
31 46340.950016\\
32 65536.000000\\
33 92681.900032\\
34 131072.000000\\
35 185363.800064\\
36 262144.000000\\
37 370727.600128\\
38 524288.000000\\
39 741455.200256\\
40 1048576.000000\\
41 1482910.400512\\
42 2097152.000000\\
43 2965820.801024\\
44 4194304.000000\\
45 5931641.602048\\
46 8388608.000000\\
47 11863283.204096\\
48 16777216.000000\\
};
\addlegendentry{$\log(\sqrt{2}^{\,n})$}

\end{axis}
\end{tikzpicture}

}
     \caption{Comparison of the parity-constrained sequences $a_n$, $b_n$, $c_n$, $d_n$
    with the classical three-peg and four-peg optima
    $h_{3}^{n}$ and $h_{4}^{n}$, and the sequence $(\sqrt{2}^{n})_{n\geq 0}$.}
    \label{fig:plot}
\end{figure}
Figure~\ref{fig:plot} displays the logarithmic growth of $a_n$, $b_n$, $c_n$, and $d_n$ 
for $n \geq 7$, alongside the classical optima $h_{3}^{n}$ and $h_{4}^{n}$. 
The four parity-constrained sequences lie strictly between $h_{3}^{n}$ and $h_{4}^{n}$,
in accordance with Proposition~\ref{prop:bounds} and Conjecture~\ref{conj:bounds}.

Furthermore, the slopes of the logarithmic curves of $a_n$, $b_n$, $c_n$, and $d_n$ are nearly parallel, 
indicating that all four sequences exhibit the same exponential order of growth.
This visual trend is reinforced by plotting the auxiliary function  $y = (\sqrt{2})^{\,n}$, whose logarithmic slope closely matches those of the four sequences. 
This suggests a half-exponential asymptotic behavior of the form
\[
a_n,\, b_n,\, c_n,\, d_n = \Theta\bigl((\sqrt{2})^{\,n}\bigr).
\]

To make this observation precise, we next analyze subsequence-dependent growth ratios 
and show that all four sequences double asymptotically over two-step increments.

\begin{theorem}[Asymptotic subsequence ratios]
\label{theo:asymptoticssubsequence}
Each of the four sequences $(a_n)$, $(b_n)$, $(c_n)$, and $(d_n)$ admits two asymptotically distinct subsequences whose ratios converge to fixed constants. 

\begin{align}
\lim_{n\to \infty}\frac{a_{2n}}{a_{2n-1}} = \frac{27}{19},
&\qquad
\lim_{n\to \infty}\frac{a_{2n+1}}{a_{2n}} = \frac{38}{27},\\[6pt]
\lim_{n\to \infty}\frac{b_{2n}}{b_{2n-1}} = \frac{38}{27},
&\qquad
\lim_{n\to \infty}\frac{b_{2n+1}}{b_{2n}} = \frac{27}{19},\\[6pt]
\lim_{n\to \infty}\frac{c_{2n}}{c_{2n-1}} = \frac{34}{31},
&\qquad
\lim_{n\to \infty}\frac{c_{2n+1}}{c_{2n}} = \frac{62}{34},\\[6pt]
\lim_{n\to \infty}\frac{d_{2n}}{d_{2n-1}} = \frac{20}{13},
&\qquad
\lim_{n\to \infty}\frac{d_{2n+1}}{d_{2n}} = \frac{26}{20}.
\end{align}
\end{theorem}

\begin{proof}
The limits follow by expressing each sequence using the closed-form expressions in Proposition~\ref{prop:closedform} and applying dominant-term asymptotics as $k \to \infty$. For instance, for $a_n$ the leading exponential behavior alternates with parity, giving the stated ratios. The remaining limits are obtained similarly.

These limits can also be obtained from   the mod-\(6\) explicit formulas of Theorem~\ref{thm:b-mod6} and Corollary~\ref{cor:acd-mod6}. For instance, 
\[
\lim_{n\to\infty}\frac{b_{6n}}{b_{6n-1}}
=
\lim_{n\to\infty}
\frac{\frac{19}{7}(8^n-1)-2n}
{\frac{108}{7}(8^{n-1}-1)+11-2(n-1)}
=
\frac{19\cdot 8}{108}
=
\frac{38}{27}.
\]
Similarly,
\[
\lim_{n\to\infty}\frac{b_{6n+2}}{b_{6n+1}}
=
\lim_{n\to\infty}
\frac{\frac{38}{7}(8^n-1)+2-2n}
{\frac{27}{7}(8^n-1)+1-2n}
=
\frac{38}{27},
\]
and
\[
\lim_{n\to\infty}\frac{b_{6n+4}}{b_{6n+3}}
=
\lim_{n\to\infty}
\frac{\frac{76}{7}(8^n-1)+7-2n}
{\frac{54}{7}(8^n-1)+4-2n}
=
\frac{76}{54}
=
\frac{38}{27}.
\]
Therefore,
\[
\lim_{n\to\infty}\frac{b_{2n}}{b_{2n-1}}=\frac{38}{27}.
\]

Likewise,
\[
\lim_{n\to\infty}\frac{b_{6n+1}}{b_{6n}}
=
\lim_{n\to\infty}
\frac{\frac{27}{7}(8^n-1)+1-2n}
{\frac{19}{7}(8^n-1)-2n}
=
\frac{27}{19},
\]
\[
\lim_{n\to\infty}\frac{b_{6n+3}}{b_{6n+2}}
=
\lim_{n\to\infty}
\frac{\frac{54}{7}(8^n-1)+4-2n}
{\frac{38}{7}(8^n-1)+2-2n}
=
\frac{54}{38}
=
\frac{27}{19},
\]
and
\[
\lim_{n\to\infty}\frac{b_{6n+5}}{b_{6n+4}}
=
\lim_{n\to\infty}
\frac{\frac{108}{7}(8^n-1)+11-2n}
{\frac{76}{7}(8^n-1)+7-2n}
=
\frac{108}{76}
=
\frac{27}{19}.
\]
Thus,
\[
\lim_{n\to\infty}\frac{b_{2n+1}}{b_{2n}}=\frac{27}{19}.
\]

The corresponding limits for \( (a_n) \), \( (c_n) \), and \( (d_n) \) are obtained in exactly the same way from their respective mod-\(6\) formulas.\qedhere

\end{proof}

\begin{proposition}[Two-step growth behavior]
\label{prop:ratios2}
For each sequence $x_n \in \{a_n, b_n, c_n, d_n\}$, the two-step growth ratio satisfies
\begin{equation}
    \lim_{n\to \infty}\frac{x_n}{x_{n-2}}= 2.
\end{equation}
\end{proposition}

\begin{proof}
From Theorem~\ref{theo:asymptoticssubsequence}, each sequence has two parity-dependent subsequences whose successive ratios converge to distinct constants $\alpha > 1$ and $\beta > 1$. Multiplying these constants shows
\[
\frac{x_{n}}{x_{n-2}} \sim \alpha \cdot \beta = 2.
\]
The same behavior follows directly from the leading exponential terms in the closed forms of Proposition~\ref{prop:closedform}.\qedhere
\end{proof}
 
\begin{lemma}[Asymptotic consequence]\label{lem:asymptotic-mod6}
For each sequence \(x_n\in\{a_n,b_n,c_n,d_n\}\) and each residue \(r\in\{0,1,\dots,5\}\), one has
\begin{equation}
    \lim_{m\to\infty}\frac{x_{6(m+1)+r}}{x_{6m+r}}=8.
\end{equation}
\end{lemma}

\begin{proof}
Each formula in Corollary~\ref{cor:acd-mod6} and Theorem~\ref{thm:b-mod6} has the form
\begin{equation}
    x_{6m+r}=A_r(8^m-1)+B_r-C_r m
\end{equation}
for suitable constants \(A_r>0\), \(B_r\), and \(C_r\). Therefore,
\[
\frac{x_{6(m+1)+r}}{x_{6m+r}}
=
\frac{A_r(8^{m+1}-1)+B_r-C_r(m+1)}
     {A_r(8^{m}-1)+B_r-C_r m}
\longrightarrow 8
\qquad (m\to\infty).\qedhere
\]
\end{proof}

\begin{theorem}[Half-exponential asymptotic growth]
\label{theo:asymptotic_growth}
\sloppy For each of the four sequences $x_n \in \{a_n, b_n, c_n, d_n\}$, we have

\begin{equation}
    x_n = \Theta\bigl((\sqrt{2})^{\,n}\bigr).
\end{equation}

\end{theorem}

\begin{proof}
From Proposition~\ref{prop:ratios2}, we have 
\[
\frac{x_{n}}{x_{n-2}} \to 2,
\]
which implies an asymptotic doubling every two steps. Hence, for sufficiently large $n$,
\[
x_n \sim C \cdot (\sqrt{2})^n
\]
for some constant $C > 0$, establishing the stated order of growth.

This result can also be obtained from Lemma~\ref{lem:asymptotic-mod6}. More precisely, for each residue class \(r\in\{0,1,\dots,5\}\), the subsequence \((x_{6m+r})_{m\ge0}\) satisfies
\[
x_{6m+r}\asymp 8^m.
\]
Hence,
\[
\frac{x_{6(m+1)+r}}{x_{6m+r}} \to 8
\qquad (m\to\infty),
\]
for every fixed \(r\). Since \(8^m=(\sqrt{2})^{6m}\), we conclude that
\[
x_n=\Theta\!\left((\sqrt{2})^n\right).\qedhere
\]
\end{proof}

\section{On the Number of Optimal Solutions}\label{sec:NumberofOptimalSolutions}

In this section, we record some observations on the number of optimal solutions for the four objectives introduced earlier and formulate a conjectural description of these numbers.

For each objective $x \in \{a,b,c,d\}$ and each $n \ge 0$, let $\nu_x(n)$ denote the number of optimal solutions for objective $x$ on $n$ discs. By convention, $\nu_a(0)=\nu_b(0)=\nu_c(0)=\nu_d(0)=1$, corresponding to the empty move sequence.

Using exact shortest-path counts in the state graph (see Section~\ref{sec:AssociatedGraph}), we computed the values of $\nu_x(n)$ for all $x \in \{a,b,c,d\}$ up to $n=16$. Thus, the first $17$ terms of each sequence listed in Table~\ref{tab:number-optimal-solutions} are exact. The remaining entries in the table are obtained from the conjectured recurrences stated in Conjecture~\ref{conj:number-optimal-solutions}; these recurrences agree with all computed values up to $n=16$.

\begin{table}
 
\caption{Number of optimal solutions for the four objectives up to $n=22$. The first $17$ terms are obtained by exact computation, while the remaining terms are conjectural.}
\label{tab:number-optimal-solutions}
 \footnotesize
 \setlength{\tabcolsep}{3pt}
\begin{tabular*}{\linewidth}{@{\extracolsep{\fill}}l|lllllllllllllllll|llllll@{}}
\toprule
$n$
&0& 1 & 2 & 3 & 4 & 5 & 6 & 7 & 8 & 9 & 10 & 11
& 12 & 13 & 14 & 15 & 16 & 17 & 18 & 19 & 20 & 21 & 22 \\
\midrule
$\nu_a(n)$
& 1&1 & 1 & 1 & 1 & 9 & 9 & 4 & 81 & 36 & 36 & 324
& 324 & 144 & 2916 & 1296 & 1296 & 11664 & 11664 & 5184 & 104976 & 46656 & 46656 \\

$\nu_b(n)$
& 1&1 & 1 & 1 & 3 & 3 & 2 & 9 & 6 & 6 & 18 & 18
& 12 & 54 & 36 & 36 & 108 & 108 & 72 & 324 & 216 & 216 & 648 \\

$\nu_c(n)$
& 1&1 & 1 & 2 & 1 & 2 & 3 & 6 & 9 & 18 & 6 & 12
& 18 & 36 & 54 & 108 & 36 & 72 & 108 & 216 & 324 & 648 & 216 \\

$\nu_d(n)$
& 1&1 & 1 & 1 & 1 & 3 & 9 & 2 & 6 & 6 & 18 & 18
& 54 & 12 & 36 & 36 & 108 & 108 & 324 & 72 & 216 & 216 & 648 \\
\bottomrule
\end{tabular*}

\end{table}

The values in Table~\ref{tab:number-optimal-solutions} suggest that the numbers of optimal solutions satisfy simple multiplicative recurrences, although we do not yet have a formal proof.

\begin{conjecture}\label{conj:number-optimal-solutions}
For all $n \ge 1$, the numbers $\nu_a(n)$, $\nu_b(n)$, $\nu_c(n)$, and $\nu_d(n)$ satisfy
\begin{align}
    \nu_a(n)&=\hspace{3mm}\nu_b(n-1)^2.\\
        \nu_b(n)&=
\begin{cases}
2\,\nu_b(n-3), & \text{ $n$ even},\\ 
3\,\nu_b(n-3), & \text{ $n$ odd},
\end{cases}\\
    \nu_c(n)&=
\begin{cases}
\nu_b(n-1), &   \text{ $n$ even},\\ 
2\,\nu_b(n-2), &  \text{ $n$ odd}.
\end{cases}\\
\nu_d(n)&=
\begin{cases}
3\,\nu_b(n-2), &\text{ $n$ even},\\ 
\nu_b(n-1), & \text{ $n$ odd}.
\end{cases}
\end{align}
 
with initial values $\nu_b(0)=\nu_b(1)=\nu_b(2)=\nu_b(3)=1$, and $\nu_b(4)=3$. 
\end{conjecture}

The recurrences of Conjecture~\ref{conj:number-optimal-solutions} were inferred experimentally from exact shortest-path counts in the state graph for small values of $n$, together with the canonical  decompositions suggested by Algorithms~\ref{alg:A}--\ref{alg:D}. More precisely, exhaustive shortest-path counting yields the initial values for $0 \le n \le 16$, and from these data one observes the striking pattern
\begin{equation}
    \frac{\nu_b(n)}{\nu_b(n-3)}=
\begin{cases}
2, & \text{if $n$ even},\\[1mm]
3, & \text{if $n$ odd},
\end{cases}
\qquad (n \ge 5).
\end{equation}

Next, the canonical decomposition of objective \textup{\texttt{(a)}} suggested by  Algorithm~\ref{alg:A} consists of a \textup{\texttt{(b)}}-block, followed by the unique move of disc $n$, and then a symmetric   \textup{\texttt{(b)}}-block. This strongly suggests that the two \textup{\texttt{(b)}}-blocks can be chosen independently, which leads to $\nu_a(n)=\nu_b(n-1)^2$.

For objective \textup{\texttt{(c)}} with $n$ even, and for objective \textup{\texttt{(d)}} with $n$ odd, Algorithm~\ref{alg:C} reduces the problem to one optimal \textup{\texttt{(b)}}-block followed by a unique classical three-peg transfer. Since that final transfer is unique, the multiplicity should come entirely from the initial \textup{\texttt{(b)}}-block. This leads to $\nu_c(n)=\nu_b(n-1)$  for $n$ even, and $\nu_d(n)=\nu_b(n-1)$ for $n$ odd.

Finally, in the mixed-parity cases---namely objective \textup{\texttt{(c)}} with $n$ odd and objective \textup{\texttt{(d)}} with $n$ even---the computed data suggest an additional finite branching in the middle part of an optimal solution. Empirically, this branching appears to have size $2$ in case \textup{\texttt{(c)}} and size $3$ in case \textup{\texttt{(d)}}. This leads to the conjectural formulas $\nu_c(n)=2\,\nu_b(n-2)$ for $n$ odd, and  $\nu_d(n)=3\,\nu_b(n-2)$ for $n$ even.

Substituting these formulas into the recurrence pattern for objective \textup{\texttt{(b)}} then yields the conjectured recurrence for $\nu_b(n)$.

Assuming Conjecture~\ref{conj:number-optimal-solutions}, the sequence
\((\nu_b(n))_{n\ge 0}\) is periodic modulo \(6\) up to a multiplicative factor \(6\).  More precisely,  for every integer $m \ge 1$,   
\begin{align}
    \nu_b(6m)&=2\cdot 6^{m-1}=\texttt{A167747}(m),\\
\nu_b(6m+1)&=9\cdot 6^{m-1}=\texttt{A199413}(m-1)-1,
\end{align}

while, for every \(m\ge 0\),
\begin{align}
    \nu_b(6m+2)&=\nu_b(6m+3)=6^m=\texttt{A000400}(m),\\
\nu_b(6m+4)&=\nu_b(6m+5)=3\cdot 6^m=\texttt{A081341}(m+1).
\end{align}

The formulas for the other three sequences are then obtained by substitution.

For \(\nu_a\),  we obtain $\nu_a(1)=\nu_a(2)=1$, and for every \(m\ge 1\),
\begin{align}
    \nu_a(6m)&=9\cdot 36^{m-1},\\
\nu_a(6m+1)&=4\cdot 36^{m-1},\\
\nu_a(6m+2)&=81\cdot 36^{m-1},
\end{align}
while, for every \(m\ge 0\),
\begin{align}
    \nu_a(6m+3)=\nu_a(6m+4)&=36^m=\texttt{A009980}(m),\\
\nu_a(6m+5)&=9\cdot 36^m.
\end{align}

For \(\nu_c\), we use the parity-dependent definition above in Conjecture~\ref{conj:number-optimal-solutions}. Since the odd case involves \(\nu_b(n-2)\), it is natural to state the formulas for \(n\ge 2\). We first have $\nu_c(2)=1$, and $\nu_c(3)=2$,  and then, for every \(m\ge 1\),
\begin{align}
    \nu_c(6m)&=3\cdot 6^{m-1}=\texttt{A081341}(m),\\
\nu_c(6m+1)&=6^m=\texttt{A000400}(m),\\
\nu_c(6m+2)&=9\cdot 6^{m-1}=\texttt{A199413}(m-1)-1,\\
\nu_c(6m+3)&=3\cdot 6^m=\texttt{A081341}(m+1),
\end{align}
while, for every \(m\ge 0\),
\begin{align}
    \nu_c(6m+4)&=6^m=\texttt{A000400}(m),\\
\nu_c(6m+5)&=2\cdot 6^m=\texttt{A167747}(m+1).
\end{align}

Finally, for \(\nu_d\), we obtain $\nu_d(1)=1$, and $\nu_d(2)=3$,  and  for every \(m\ge 1\),
\begin{align}
    \nu_d(6m)&=9\cdot 6^{m-1}=\texttt{A199413}(m-1)-1,\\
\nu_d(6m+1)&=2\cdot 6^{m-1}=\texttt{A167747}(m),\\
\nu_d(6m+2)&=6^m=\texttt{A000400}(m),
\end{align}
while, for every \(m\ge 0\),
\begin{align}
    \nu_d(6m+3)&=6^m=\texttt{A000400}(m),\\
\nu_d(6m+4)=\nu_d(6m+5)&=3\cdot 6^m=\texttt{A081341}(m+1).
\end{align}

Thus, each of the four sequences admits an explicit description according to the residue class of the index modulo \(6\), with \(\nu_a\) growing like powers of \(36\), and \(\nu_b\), \(\nu_c\), and \(\nu_d\) growing like powers of \(6\).

 \section{Linear Variant} \label{sec:LinearVariant}

In the section, we consider the  linear variant of the parity-constrained four-peg Tower of Hanoi. The set of pegs is still $P$ with the same parity restrictions as in Section~\ref{sec:ProblemDescription}, but the legal moves are now required to be \emph{adjacent}. More precisely, if \(d\) is even, then its allowed moves are only $N_1 \leftrightarrow E$ and $E \leftrightarrow N_2$,  whereas if \(d\) is odd, then its allowed moves are only $N_1 \leftrightarrow O$ and $O \leftrightarrow N_2$. In particular, no disc may move directly from \(N_1\) to \(N_2\).

We denote by $a_n^{\mathrm{lin}}$, $b_n^{\mathrm{lin}}$, $c_n^{\mathrm{lin}}$, and $d_n^{\mathrm{lin}}$, the optimal move counts for objectives \textup{\texttt{(a)}}--\textup{\texttt{(d)}} in this linear setting.

\subsection{First values}

The initial values were obtained by exhaustive shortest-path computation in the corresponding
state graph; they are listed in Table~\ref{tab:linear-parity-lengths}.

\begin{table}
\caption{Optimal lengths for the linear parity-constrained variant, together with their OEIS identifications when available.}
\label{tab:linear-parity-lengths}
\footnotesize
\setlength{\tabcolsep}{4pt}
\begin{tabular*}{\linewidth}{@{\extracolsep{\fill}}l|lllllllllllllllll|l@{}}
\toprule
$n$ & 0&1 & 2 & 3 & 4 & 5 & 6 & 7 & 8 & 9 & 10 & 11 & 12 & 13 & 14 & 15 & 16& OEIS \\
\midrule
$a_n^{\mathrm{lin}}$ &0& 2 & 4 & 10 & 16 & 34 & 52 & 106 & 160 & 322 & 484 & 970 & 1456 & 2914 & 4372 & 8746 & 13120& \texttt{A117862}$(n)$   \\
$b_n^{\mathrm{lin}}$ &0& 1 & 2 & 5 & 8 & 17 & 26 & 53 & 80 & 161 & 242 & 485 & 728 & 1457 & 2186 & 4373 & 6560& \texttt{A062318}$(n)$ \\
$c_n^{\mathrm{lin}}$ &0& 1 & 3 & 6 & 12 & 21 & 39 & 66 & 120 & 201 & 363 & 606 & 1092 & 1821 & 3279 & 5466 & 9840& \texttt{A087503}$(n-1)$ \\
$d_n^{\mathrm{lin}}$ &0& 2 & 3 & 9 & 12 & 30 & 39 & 93 & 120 & 282 & 363 & 849 & 1092 & 2550 & 3279 & 7653 & 9840& \texttt{-} \\
\bottomrule
\end{tabular*}

\end{table}

Interestingly, in contrast to the non-linear variant, the linear version leads to sequences with a
much more regular behavior. In particular, the sequences, $a_n^{\mathrm{lin}}$, $b_n^{\mathrm{lin}}$, and $c_n^{\mathrm{lin}}$, coincide with registered OEIS sequences, as indicated in Table~\ref{tab:linear-parity-lengths}. The
sequence $d_n^{\mathrm{lin}}$ does not appear to be registered as a whole; however, its even
subsequence
\[
\begin{aligned}
d_{2m-1}^{\mathrm{lin}}&= 2,  9,  30,  93,  282, 849,  2550,  7653 ,\dots\\
d_{2m}^{\mathrm{lin}}&=0,3,12,39,120,363,1092,3279,\dots
\end{aligned}
\]
matches the OEIS sequences \texttt{A237930}$(m-1)+1$ and \texttt{A029858}$(m)$, respectively. Morover, according to the OEIS corresponding webpage, the sequence \texttt{A087503} seems to have no explicit combinatorial interpretation yet since no comment is mentioned in the corresponding page, thus, this could be its first natural combinatorial interpretation: up to an index shift, it gives the optimal lengths for
objective~\textup{\texttt{(c)}} in the linear parity-constrained problem.

\subsection{Closed forms}

Before deriving the formulas for \(a_n^{\mathrm{lin}}, b_n^{\mathrm{lin}}, c_n^{\mathrm{lin}}, d_n^{\mathrm{lin}}\), we recall the classical linear three-peg Tower of Hanoi on a path of length two. Let \(l_n\) denote the minimum number of moves needed to transfer a tower of \(n\) discs between the left and the right pegs, and let \(\ell_n\) denote the minimum number of moves needed to transfer it between adjacent pegs. It is well known, and easy to prove by induction, that 
\begin{align}
    l_n&=3^n-1,\label{eq:l_m}\\
    \ell_n&=\frac{3^n-1}{2}\label{eq:ell_m}.
\end{align}

We now derive exact formulas for the four sequences of the linear parity-constrained problem.

\begin{theorem}\label{theo:LinearExplicit}
For every integer \(m\ge 1\), the optimal move counts of the linear variant are given by
\begin{align}
a_{2m-1}^{\mathrm{lin}}&=4\cdot 3^{m-1}-2,   &&a_{2m}^{\mathrm{lin}} =2\cdot 3^{m}-2,\label{eq:linear_thoerem_a}\\
    b_{2m-1}^{\mathrm{lin}}&=2\cdot 3^{m-1}-1, 
&&b_{2m}^{\mathrm{lin}} =3^{m}-1,\label{eq:linear_thoerem_b}\\
c_{2m-1}^{\mathrm{lin}}&=\frac{5\cdot 3^{m-1}-3}{2},&&
c_{2m}^{\mathrm{lin}}=\frac{3^{m+1}-3}{2},\label{eq:linear_thoerem_c}\\
d_{2m-1}^{\mathrm{lin}}&=\frac{7\cdot 3^{m-1}-3}{2},&&
d_{2m}^{\mathrm{lin}}=\frac{3^{m+1}-3}{2}.\label{eq:linear_thoerem_d}
\end{align}
In particular, for all $m\ge 1$, 
\begin{equation}
d_{2m}^{\mathrm{lin}}=c_{2m}^{\mathrm{lin}}.
\end{equation}
\end{theorem}

\begin{proof}
We proceed objective by objective.

 \textbf{Objective \textup{\texttt{(b)}}.}
To reach the parity-separated state on \(n\) discs, the two largest discs \(n-1\) and \(n\) must first be exposed. This requires transferring the \(n-2\) smallest discs from \(N_1\) to \(N_2\), which costs \(a_{n-2}^{\mathrm{lin}}\) moves. Then disc \(n-1\) is moved to its parity peg and disc \(n\) is moved to its parity peg, which costs \(2\) additional moves. Finally, the \(n-2\) smallest discs must be transferred from \(N_2\) to the parity-separated configuration; by symmetry, this costs \(b_{n-2}^{\mathrm{lin}}\) moves. Hence, for all $n\ge 3$, we have 
\begin{equation}\label{eq:linearProof1}
    b_n^{\mathrm{lin}}=a_{n-2}^{\mathrm{lin}}+b_{n-2}^{\mathrm{lin}}+2
\end{equation}
Now, once the parity-separated configuration has been reached, the passage from that state to the full tower on \(N_2\) is the mirror image of the passage from the initial state to the parity-separated configuration. Therefore, for all $n\ge1$, we have 
\begin{equation}\label{eq:linearProof2}
    a_n^{\mathrm{lin}}=2\,b_n^{\mathrm{lin}}
\end{equation}
Substituting \eqref{eq:linearProof2} this into \eqref{eq:linearProof1} yields, for all $n\ge 3$, 
\begin{equation}
    b_n^{\mathrm{lin}}=3b_{n-2}^{\mathrm{lin}}+2
\end{equation}
with initial conditions $b_1^{\mathrm{lin}}=1$, and $b_2^{\mathrm{lin}}=2$.

Separating odd and even indices gives
\begin{align}
 b_{2m-1}^{\mathrm{lin}}&=3b_{2m-3}^{\mathrm{lin}}+2,\\
 b_{2m}^{\mathrm{lin}}&=3b_{2m-2}^{\mathrm{lin}}+2.
\end{align}
Solving these first-order recurrences yields \eqref{eq:linear_thoerem_b}. Hence, \eqref{eq:linear_thoerem_a} follows immediately from \eqref{eq:linear_thoerem_b} and \eqref{eq:linearProof2}.

 \textbf{Objective \textup{\texttt{(c)}}, odd case.}
Let \(n=2m-1\). First transfer the \(2m-3\) smallest discs from \(N_1\) to \(N_2\), which costs \(a_{2m-3}^{\mathrm{lin}}\). Then move disc \(2m-2\) to \(E\) and disc \(2m-1\) to \(O\), which costs \(2\) moves. Next, transfer the \(2m-3\) smallest discs from \(N_2\) to the parity-separated configuration, which costs \(b_{2m-3}^{\mathrm{lin}}\). At that point, all odd discs are already on \(O\), while the even discs form a monochromatic tower on \(E\). To complete objective \textup{\texttt{(c)}}, this even tower of \(m\) discs must be transferred from \(E\) to \(N_2\), which is an adjacent transfer in the linear three-peg setting, and therefore costs \(\ell_m\). Thus
\[
c_{2m-1}^{\mathrm{lin}}
=
a_{2m-3}^{\mathrm{lin}}
+2
+b_{2m-3}^{\mathrm{lin}}
+\ell_m.
\]
Using the formulas already established for \(a_{2m-3}^{\mathrm{lin}}\) and \(b_{2m-3}^{\mathrm{lin}}\), together with \eqref{eq:ell_m},  we obtain
\[
c_{2m-1}^{\mathrm{lin}}
=
\bigl(4\cdot 3^{m-2}-2\bigr)+2+\bigl(2\cdot 3^{m-2}-1\bigr)+\frac{3^m-1}{2}
=
\frac{5\cdot 3^{m-1}-3}{2}.
\]

 \textbf{Objective \textup{\texttt{(c)}}, even case.}
Let \(n=2m\). First solve objective \textup{\texttt{(c)}} on the \(2m-1\) smallest discs, which costs \(c_{2m-1}^{\mathrm{lin}}\). Then move disc \(2m\) from \(N_1\) to \(E\), which costs \(1\) move. In order to move disc \(2m\) from \(E\) to \(N_2\), the smaller even discs must be evacuated from \(N_2\) to \(N_1\); since the odd discs are fixed on \(O\), this is an end-to-end transfer of a monochromatic tower of \(m-1\) even discs, and therefore costs \(l_{m-1}\). After moving disc \(2m\) from \(E\) to \(N_2\), the same monochromatic transfer must be performed again to bring the smaller even discs from \(N_1\) to \(N_2\), which costs another \(l_{m-1}\). Hence
\begin{equation}
    c_{2m}^{\mathrm{lin}}
=
c_{2m-1}^{\mathrm{lin}}+2l_{m-1}+2.
\end{equation}
Using the already proved odd case of \eqref{eq:linear_thoerem_c} and \eqref{eq:l_m}, we obtain
\[
c_{2m}^{\mathrm{lin}}
=
\frac{5\cdot 3^{m-1}-3}{2}+2(3^{m-1}-1)+2
=
\frac{9\cdot 3^{m-1}-3}{2}
=
\frac{3^{m+1}-3}{2}.
\]

 \textbf{Objective \textup{\texttt{(d)}}, odd case.}
Let \(n=2m-1\). The beginning of the construction is the same as for objective \textup{\texttt{(c)}}: first transfer the \(2m-3\) smallest discs from \(N_1\) to \(N_2\), then move discs \(2m-2\) and \(2m-1\) to their parity pegs, and finally transfer the \(2m-3\) smallest discs from \(N_2\) to the parity-separated configuration. This costs $a_{2m-3}^{\mathrm{lin}}+2+b_{2m-3}^{\mathrm{lin}}$.

At that stage, the odd discs form a monochromatic tower of \(m\) discs on \(O\). To complete objective \textup{\texttt{(d)}}, this tower must be transferred from \(O\) to \(N_2\), which is an end-to-end transfer in the linear three-peg setting, and therefore costs \(l_m\). Hence
\begin{equation}
    d_{2m-1}^{\mathrm{lin}}
=
a_{2m-3}^{\mathrm{lin}}+2+b_{2m-3}^{\mathrm{lin}}+l_m.
\end{equation}
Using \eqref{eq:l_m},  we obtain
\[
d_{2m-1}^{\mathrm{lin}}
=
\bigl(4\cdot 3^{m-2}-2\bigr)+2+\bigl(2\cdot 3^{m-2}-1\bigr)+(3^m-1)
=
\frac{7\cdot 3^{m-1}-3}{2}.
\]

 \textbf{Objective \textup{\texttt{(d)}}, even case.}
For \(n=2m\), the same decomposition as in the even case of objective \textup{\texttt{(c)}} applies after exchanging the roles of the odd and even monochromatic subtowers. This yields to the even case of \eqref{eq:linear_thoerem_d}.

Collecting the four cases completes the proof.\qedhere
\end{proof}

A convenient recurrence form follows immediately.

\begin{corollary}
The four sequences satisfy the recurrences 
\begin{align}
    &a_1^{\mathrm{lin}}=2,\qquad a_2^{\mathrm{lin}}=4,\qquad \forall n\ge 3:a_n^{\mathrm{lin}}=3a_{n-2}^{\mathrm{lin}}+4,\\
   & b_1^{\mathrm{lin}}=1,\qquad b_2^{\mathrm{lin}}=2,\qquad \forall n\ge 3:b_n^{\mathrm{lin}}=3b_{n-2}^{\mathrm{lin}}+2,\\
    &c_1^{\mathrm{lin}}=1,\qquad c_2^{\mathrm{lin}}=3,\qquad \forall n\ge 3:c_n^{\mathrm{lin}}=3c_{n-2}^{\mathrm{lin}}+3,\\
    &d_1^{\mathrm{lin}}=2,\qquad d_2^{\mathrm{lin}}=3, \qquad \forall n\ge 3:d_n^{\mathrm{lin}}=3d_{n-2}^{\mathrm{lin}}+3.
\end{align}
\end{corollary}

\begin{proof}
The recurrences follow directly by substituting the explicit formulas of Theorem~\ref{theo:LinearExplicit} and simplifying separately according to the parity of $n$.\qedhere
\end{proof}

\begin{corollary}
For every integer \(n \ge 0\), the following chain of inequalities holds:
\begin{equation}
    b_n^{\mathrm{lin}} \le c_n^{\mathrm{lin}} \le d_n^{\mathrm{lin}} \le a_n^{\mathrm{lin}} \le \ell_n \le l_n .
\end{equation}
\end{corollary}

\begin{corollary}
For the linear variant, the following limits hold:
\begin{equation}
    \lim_{m\to\infty}\frac{a_{2m}^{\mathrm{lin}}}{a_{2m-1}^{\mathrm{lin}}}
=
\lim_{m\to\infty}\frac{b_{2m}^{\mathrm{lin}}}{b_{2m-1}^{\mathrm{lin}}}
=\frac{3}{2},
\end{equation}
\begin{align}
\lim_{m\to\infty}\frac{c_{2m}^{\mathrm{lin}}}{c_{2m-1}^{\mathrm{lin}}}
=\frac{9}{5},\\
\lim_{m\to\infty}\frac{d_{2m}^{\mathrm{lin}}}{d_{2m-1}^{\mathrm{lin}}}
=\frac{9}{7}.
\end{align}
\end{corollary}
 
\section{State Graph}
\label{sec:AssociatedGraph}

We now introduce the graph naturally induced by the parity-constrained Tower of Hanoi problem with $n$ discs. This graph encodes all feasible configurations and legal transitions under the rules of Section~\ref{sec:ProblemDescription}.

Analogously to the classical Hanoi graph $H_{m}^{n}$ (the state graph of the $m$-peg Tower of Hanoi with $n$ discs; see~\cite[Section~5.6]{HinzMythsMaths2018}), the four-peg parity-constrained problem gives rise to a family of graphs, which we call the \emph{parity-constrained Hanoi graphs}.

We code the four pegs by $N_{1}=0$, $E=1$, $O=2$, and $N_{2}=3$. We denote again  by $p(d)$ and $\overline{p}(d)$  the allowed parity peg and the forbidden parity peg of disc $d$, respectively. Thus an even disc may not occupy peg $2$, and an odd disc may not occupy peg $1$.

\subsection*{Vertices}
Let $Q=\{0,1,2,3\}$, and $\Pi_{d}=\{0,p(d),3\}$. A state is encoded by the word $s=s_{n}\cdots s_{1}\in Q^{n}$, where $s_{d}$ is the peg occupied by disc $d$. The vertex set of the \emph{parity-constrained Hanoi graph} $\widetilde{H}_4^{n}$ is

\begin{equation}
    V(\widetilde{H}_4^{n})
=
\Bigl\{
s\in Q^{n}\,\Bigm|\,
\bigl(d \text{ even} \Rightarrow s_{d}\neq 2\bigr)
\ \wedge\
\bigl(d \text{ odd} \Rightarrow s_{d}\neq 1\bigr)
\Bigr\}.
\end{equation}

\subsection*{Edges}
Two states are adjacent if they differ in exactly one coordinate corresponding to a legal move. Formally,
\begin{equation} 
    E(\widetilde{H}_4^{n})
=
\Bigl\{
\{\underline{s}\,i\,\overline{s},\,\underline{s}\,j\,\overline{s}\}
\,\Bigm|\,  \underline{s}\,i\,\overline{s},\,\underline{s}\,j\,\overline{s}\in V(\widetilde{H}_4^{n}),\,
d\in[n],\ i\neq j,\ i,j\in \Pi_{d},\ 
\overline{s}_{k}\notin\{i,j\}\ \forall k<d
\Bigr\},
\end{equation}

i.e., disc $d$ moves between two allowed pegs $i$ and $j$ while no smaller disc $k<d$ occupies either of those pegs. (This is the parity-restricted analogue of~\cite[Eq.\,(5.44)]{HinzMythsMaths2018}.)

Figure~\ref{fig:graphs_123} illustrates $\widetilde{H}_4^{1}$, $\widetilde{H}_4^{2}$, and $\widetilde{H}_4^{3}$. Red edges indicate moves of the largest disc, the red vertex marks $0^{n}$ (the initial state), green vertices mark the four targets of objectives~\textup{\texttt{(a)}}–\textup{\texttt{(d)}}, and dashed edges depict an optimal path for objective~\textup{\texttt{(a)}}.

\subsection*{Recursive structure}
The graph $\widetilde{H}_4^{n}$ decomposes into three copies of $\widetilde{H}_4^{n-1}$ according to the peg occupied by the largest disc $n$:

\begin{equation}
    \widetilde{H}_4^{n} \cong 0\widetilde{H}_4^{n-1}\ \cup\ p(n)\widetilde{H}_4^{n-1}\ \cup\ 3\widetilde{H}_4^{n-1},
\end{equation}

where $i\widetilde{H}_4^{n-1}$ denotes the induced subgraph on states with $s_{n}=i\in \Pi_{d}$. Thus

\begin{equation}
    V(\widetilde{H}_4^{n})=V(0\widetilde{H}_4^{n-1})\ \cup\ V(p(n)\widetilde{H}_4^{n-1})\ \cup\ V(3\widetilde{H}_4^{n-1}).
\end{equation}

Edges are of two types: (i) \emph{internal} edges inside each copy $i\widetilde{H}_4^{n-1}$ (moves of a disc $<n$), and (ii) \emph{bridges} between distinct copies (moves of disc $n$). Therefore, we have the following recursive definition of $E(\widetilde{H}_4^{n})$.

\begin{equation}
\begin{split} E(\widetilde{H}_4^{0})=\varnothing,\;\forall n\geq1:\;
        E(\widetilde{H}_4^{n})=&\{ \{is,it\} \mid i\in  \Pi_{d}, \{s,t\}\in E(\widetilde{H}_4^{n-1})  \}\\ \cup& \{ \{ir,jr\}\mid i,j\in  \Pi_{d}, i\neq j, r \in V(\widetilde{H}_4^{n-1}), r_{k}\neq i,j\ \forall k <n  \}
    \end{split}
\end{equation}

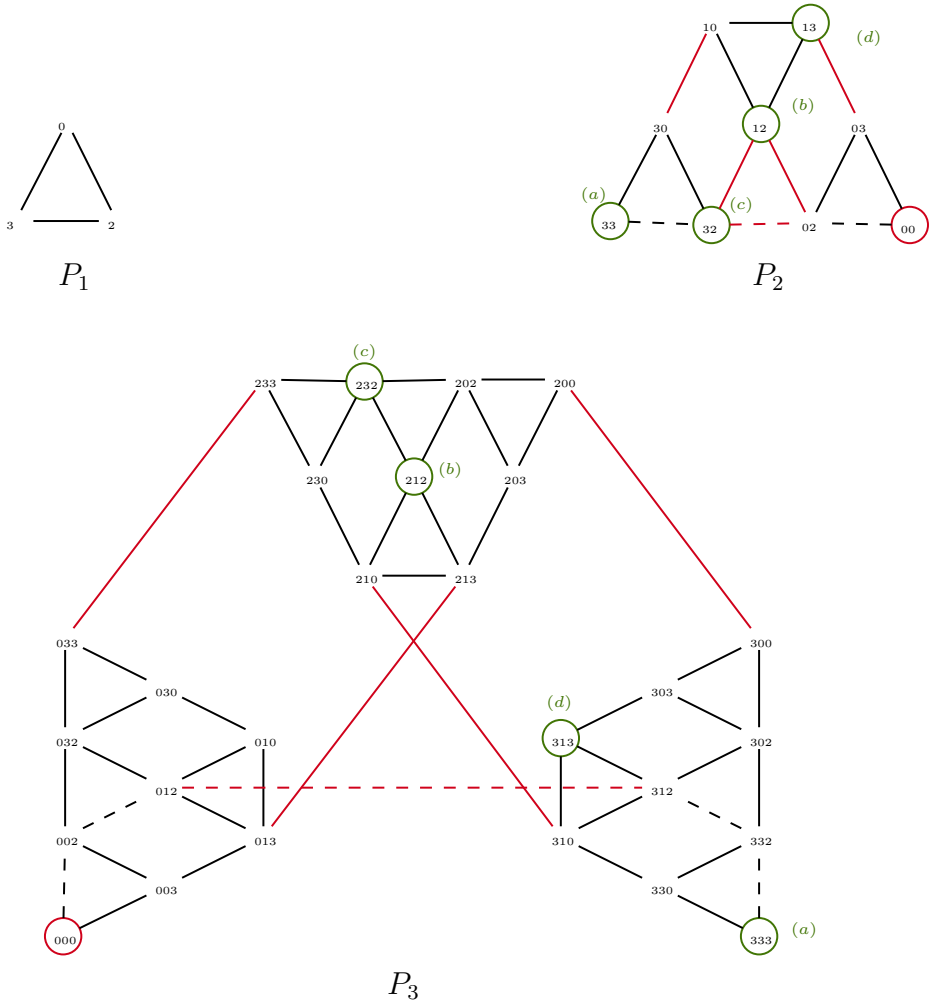
\begin{figure}[ht]
    \centering

\tikzset{every picture/.style={line width=0.75pt}} 
{

\begin{tikzpicture}[x=0.75pt,y=0.75pt,yscale=-1,xscale=1]

\draw (377,182.4) node [anchor=north west][inner sep=0.75pt]  [font=\fontsize{0.35em}{0.42em}\selectfont,]  {$200$};
\draw (352,231.4) node [anchor=north west][inner sep=0.75pt]  [font=\fontsize{0.35em}{0.42em}\selectfont,]  {$203$};
\draw (327,182.4) node [anchor=north west][inner sep=0.75pt]  [font=\fontsize{0.35em}{0.42em}\selectfont,]  {$202$};
\draw (252,231.4) node [anchor=north west][inner sep=0.75pt]  [font=\fontsize{0.35em}{0.42em}\selectfont,]  {$230$};
\draw (226,182.4) node [anchor=north west][inner sep=0.75pt]  [font=\fontsize{0.35em}{0.42em}\selectfont,]  {$233$};
\draw  [color={rgb, 255:red, 65; green, 117; blue, 5 }  ,draw opacity=1 ]  (283, 184.5) circle [x radius= 9.18, y radius= 9.18]   ;
\draw (277,183.4) node [anchor=north west][inner sep=0.75pt]  [font=\fontsize{0.35em}{0.42em}\selectfont,]  {$232$};
\draw (277,281.4) node [anchor=north west][inner sep=0.75pt]  [font=\fontsize{0.35em}{0.42em}\selectfont,]  {$210$};
\draw (327,281.4) node [anchor=north west][inner sep=0.75pt]  [font=\fontsize{0.35em}{0.42em}\selectfont,]  {$213$};
\draw  [color={rgb, 255:red, 65; green, 117; blue, 5 }  ,draw opacity=1 ]  (308, 232.5) circle [x radius= 9.18, y radius= 9.18]   ;
\draw (302,231.4) node [anchor=north west][inner sep=0.75pt]  [font=\fontsize{0.35em}{0.42em}\selectfont,]  {$212$};
\draw (476,313.4) node [anchor=north west][inner sep=0.75pt]  [font=\fontsize{0.35em}{0.42em}\selectfont,]  {$300$};
\draw (426,338.4) node [anchor=north west][inner sep=0.75pt]  [font=\fontsize{0.35em}{0.42em}\selectfont,]  {$303$};
\draw (476,363.4) node [anchor=north west][inner sep=0.75pt]  [font=\fontsize{0.35em}{0.42em}\selectfont,]  {$302$};
\draw (426,438.4) node [anchor=north west][inner sep=0.75pt]  [font=\fontsize{0.35em}{0.42em}\selectfont,]  {$330$};
\draw  [color={rgb, 255:red, 65; green, 117; blue, 5 }  ,draw opacity=1 ]  (482, 464.5) circle [x radius= 9.18, y radius= 9.18]   ;
\draw (476,463.4) node [anchor=north west][inner sep=0.75pt]  [font=\fontsize{0.35em}{0.42em}\selectfont,]  {$333$};
\draw (476,413.4) node [anchor=north west][inner sep=0.75pt]  [font=\fontsize{0.35em}{0.42em}\selectfont,]  {$332$};
\draw (376,413.4) node [anchor=north west][inner sep=0.75pt]  [font=\fontsize{0.35em}{0.42em}\selectfont,]  {$310$};
\draw  [color={rgb, 255:red, 65; green, 117; blue, 5 }  ,draw opacity=1 ]  (382, 364.5) circle [x radius= 9.18, y radius= 9.18]   ;
\draw (376,363.4) node [anchor=north west][inner sep=0.75pt]  [font=\fontsize{0.35em}{0.42em}\selectfont,]  {$313$};
\draw (426,388.4) node [anchor=north west][inner sep=0.75pt]  [font=\fontsize{0.35em}{0.42em}\selectfont,]  {$312$};
\draw (126,313.4) node [anchor=north west][inner sep=0.75pt]  [font=\fontsize{0.35em}{0.42em}\selectfont,]  {$033$};
\draw (176,338.4) node [anchor=north west][inner sep=0.75pt]  [font=\fontsize{0.35em}{0.42em}\selectfont,]  {$030$};
\draw (126,363.4) node [anchor=north west][inner sep=0.75pt]  [font=\fontsize{0.35em}{0.42em}\selectfont,]  {$032$};
\draw (176,438.4) node [anchor=north west][inner sep=0.75pt]  [font=\fontsize{0.35em}{0.42em}\selectfont,]  {$003$};
\draw  [color={rgb, 255:red, 208; green, 2; blue, 27 }  ,draw opacity=1 ]  (131, 464.5) circle [x radius= 9.18, y radius= 9.18]   ;
\draw (125,463.4) node [anchor=north west][inner sep=0.75pt]  [font=\fontsize{0.35em}{0.42em}\selectfont,]  {$000$};
\draw (126,413.4) node [anchor=north west][inner sep=0.75pt]  [font=\fontsize{0.35em}{0.42em}\selectfont,]  {$002$};
\draw (226,413.4) node [anchor=north west][inner sep=0.75pt]  [font=\fontsize{0.35em}{0.42em}\selectfont,]  {$013$};
\draw (226,363.4) node [anchor=north west][inner sep=0.75pt]  [font=\fontsize{0.35em}{0.42em}\selectfont,]  {$010$};
\draw (176,388.4) node [anchor=north west][inner sep=0.75pt]  [font=\fontsize{0.35em}{0.42em}\selectfont,]  {$012$};
\draw  [color={rgb, 255:red, 208; green, 2; blue, 27 }  ,draw opacity=1 ]  (558, 105.5) circle [x radius= 9.18, y radius= 9.18]   ;
\draw (552,104.4) node [anchor=north west][inner sep=0.75pt]  [font=\fontsize{0.35em}{0.42em}\selectfont,color={rgb, 255:red, 0; green, 0; blue, 0 }  ,opacity=1 ,]  {$00$};
\draw (527,53.4) node [anchor=north west][inner sep=0.75pt]  [font=\fontsize{0.35em}{0.42em}\selectfont,]  {$03$};
\draw (502,103.4) node [anchor=north west][inner sep=0.75pt]  [font=\fontsize{0.35em}{0.42em}\selectfont,]  {$02$};
\draw (427,53.4) node [anchor=north west][inner sep=0.75pt]  [font=\fontsize{0.35em}{0.42em}\selectfont,]  {$30$};
\draw  [color={rgb, 255:red, 65; green, 117; blue, 5 }  ,draw opacity=1 ]  (407, 103.5) circle [x radius= 9.18, y radius= 9.18]   ;
\draw (401,102.4) node [anchor=north west][inner sep=0.75pt]  [font=\fontsize{0.35em}{0.42em}\selectfont,]  {$33$};
\draw  [color={rgb, 255:red, 65; green, 117; blue, 5 }  ,draw opacity=1 ]  (458, 105.5) circle [x radius= 9.18, y radius= 9.18]   ;
\draw (452,104.4) node [anchor=north west][inner sep=0.75pt]  [font=\fontsize{0.35em}{0.42em}\selectfont,]  {$32$};
\draw (452,2.4) node [anchor=north west][inner sep=0.75pt]  [font=\fontsize{0.35em}{0.42em}\selectfont,]  {$10$};
\draw  [color={rgb, 255:red, 65; green, 117; blue, 5 }  ,draw opacity=1 ]  (508, 3.5) circle [x radius= 9.18, y radius= 9.18]   ;
\draw (502,2.4) node [anchor=north west][inner sep=0.75pt]  [font=\fontsize{0.35em}{0.42em}\selectfont,]  {$13$};
\draw  [color={rgb, 255:red, 65; green, 117; blue, 5 }  ,draw opacity=1 ]  (483, 54.5) circle [x radius= 9.18, y radius= 9.18]   ;
\draw (477,53.4) node [anchor=north west][inner sep=0.75pt]  [font=\fontsize{0.35em}{0.42em}\selectfont,]  {$12$};
\draw (127,52.4) node [anchor=north west][inner sep=0.75pt]  [font=\fontsize{0.35em}{0.42em}\selectfont,]  {$0$};
\draw (101,102.4) node [anchor=north west][inner sep=0.75pt]  [font=\fontsize{0.35em}{0.42em}\selectfont,]  {$3$};
\draw (152,102.4) node [anchor=north west][inner sep=0.75pt]  [font=\fontsize{0.35em}{0.42em}\selectfont,]  {$2$};
\draw (127,123.4) node [anchor=north west][inner sep=0.75pt]  []  {$\widetilde{H}_4^{1}$};
\draw (477,123.4) node [anchor=north west][inner sep=0.75pt]  []  {$\widetilde{H}_4^{2}$};
\draw (293,481.4) node [anchor=north west][inner sep=0.75pt]  []  {$\widetilde{H}_4^{3}$};
\draw (496.67,456.11) node [anchor=north west][inner sep=0.75pt]  [font=\tiny,color={rgb, 255:red, 65; green, 117; blue, 5 }  ,opacity=1 ,]  {$( a)$};
\draw (318.67,223.45) node [anchor=north west][inner sep=0.75pt]  [font=\tiny,color={rgb, 255:red, 65; green, 117; blue, 5 }  ,opacity=1 ,]  {$( b)$};
\draw (275.33,164.11) node [anchor=north west][inner sep=0.75pt]  [font=\tiny,color={rgb, 255:red, 65; green, 117; blue, 5 }  ,opacity=1 ,]  {$( c)$};
\draw (373.33,341.45) node [anchor=north west][inner sep=0.75pt]  [font=\tiny,color={rgb, 255:red, 65; green, 117; blue, 5 }  ,opacity=1 ,]  {$( d)$};
\draw (390.67,84.78) node [anchor=north west][inner sep=0.75pt]  [font=\tiny,color={rgb, 255:red, 65; green, 117; blue, 5 }  ,opacity=1 ,]  {$( a)$};
\draw (496.67,40.11) node [anchor=north west][inner sep=0.75pt]  [font=\tiny,color={rgb, 255:red, 65; green, 117; blue, 5 }  ,opacity=1 ,]  {$( b)$};
\draw (465.33,90.11) node [anchor=north west][inner sep=0.75pt]  [font=\tiny,color={rgb, 255:red, 65; green, 117; blue, 5 }  ,opacity=1 ,]  {$( c)$};
\draw (529.33,5.45) node [anchor=north west][inner sep=0.75pt]  [font=\tiny,color={rgb, 255:red, 65; green, 117; blue, 5 }  ,opacity=1 ,]  {$( d)$};
\draw [color={rgb, 255:red, 0; green, 0; blue, 0 }  ,draw opacity=1 ]   (380.19,189) -- (360.81,227) ;
\draw [color={rgb, 255:red, 0; green, 0; blue, 0 }  ,draw opacity=1 ]   (355.19,227) -- (335.81,189) ;
\draw [color={rgb, 255:red, 0; green, 0; blue, 0 }  ,draw opacity=1 ]   (374,183.5) -- (342,183.5) ;
\draw [color={rgb, 255:red, 0; green, 0; blue, 0 }  ,draw opacity=1 ]   (241,183.68) -- (273.82,184.32) ;
\draw [color={rgb, 255:red, 0; green, 0; blue, 0 }  ,draw opacity=1 ]   (278.76,192.64) -- (260.86,227) ;
\draw [color={rgb, 255:red, 0; green, 0; blue, 0 }  ,draw opacity=1 ]   (255.08,227) -- (234.92,189) ;
\draw [color={rgb, 255:red, 0; green, 0; blue, 0 }  ,draw opacity=1 ]   (330.25,277) -- (312.11,240.71) ;
\draw [color={rgb, 255:red, 0; green, 0; blue, 0 }  ,draw opacity=1 ]   (303.89,240.71) -- (285.75,277) ;
\draw [color={rgb, 255:red, 0; green, 0; blue, 0 }  ,draw opacity=1 ]   (292,282.5) -- (324,282.5) ;
\draw [color={rgb, 255:red, 0; green, 0; blue, 0 }  ,draw opacity=1 ]   (292.18,184.32) -- (324,183.68) ;
\draw [color={rgb, 255:red, 0; green, 0; blue, 0 }  ,draw opacity=1 ]   (287.24,192.64) -- (303.76,224.36) ;
\draw [color={rgb, 255:red, 0; green, 0; blue, 0 }  ,draw opacity=1 ]   (260.75,238) -- (280.25,277) ;
\draw [color={rgb, 255:red, 0; green, 0; blue, 0 }  ,draw opacity=1 ]   (335.75,277) -- (355.25,238) ;
\draw [color={rgb, 255:red, 0; green, 0; blue, 0 }  ,draw opacity=1 ]   (330.19,189) -- (312.17,224.32) ;
\draw [color={rgb, 255:red, 0; green, 0; blue, 0 }  ,draw opacity=1 ]   (473,319) -- (441,335) ;
\draw [color={rgb, 255:red, 0; green, 0; blue, 0 }  ,draw opacity=1 ]   (441,344) -- (473,360) ;
\draw [color={rgb, 255:red, 0; green, 0; blue, 0 }  ,draw opacity=1 ]   (482,320) -- (482,359) ;
\draw [color={rgb, 255:red, 0; green, 0; blue, 0 }  ,draw opacity=1 ] [dash pattern={on 4.5pt off 4.5pt}]  (482,455.32) -- (482,420) ;
\draw [color={rgb, 255:red, 0; green, 0; blue, 0 }  ,draw opacity=1 ]   (473,419) -- (441,435) ;
\draw [color={rgb, 255:red, 0; green, 0; blue, 0 }  ,draw opacity=1 ]   (441,444) -- (473.79,460.39) ;
\draw [color={rgb, 255:red, 0; green, 0; blue, 0 }  ,draw opacity=1 ]   (390.21,368.61) -- (423,385) ;
\draw [color={rgb, 255:red, 0; green, 0; blue, 0 }  ,draw opacity=1 ]   (423,394) -- (391,410) ;
\draw [color={rgb, 255:red, 0; green, 0; blue, 0 }  ,draw opacity=1 ]   (382,409) -- (382,373.68) ;
\draw [color={rgb, 255:red, 0; green, 0; blue, 0 }  ,draw opacity=1 ]   (482,409) -- (482,370) ;
\draw [color={rgb, 255:red, 0; green, 0; blue, 0 }  ,draw opacity=1 ] [dash pattern={on 4.5pt off 4.5pt}]  (473,410) -- (441,394) ;
\draw [color={rgb, 255:red, 0; green, 0; blue, 0 }  ,draw opacity=1 ]   (423,435) -- (391,419) ;
\draw [color={rgb, 255:red, 0; green, 0; blue, 0 }  ,draw opacity=1 ]   (390.21,360.39) -- (423,344) ;
\draw [color={rgb, 255:red, 0; green, 0; blue, 0 }  ,draw opacity=1 ]   (473,369) -- (441,385) ;
\draw [color={rgb, 255:red, 208; green, 2; blue, 27 }  ,draw opacity=1 ]   (387.16,189) -- (477.84,309) ;
\draw [color={rgb, 255:red, 208; green, 2; blue, 27 }  ,draw opacity=1 ]   (287.13,288) -- (377.88,409) ;
\draw [color={rgb, 255:red, 0; green, 0; blue, 0 }  ,draw opacity=1 ]   (141,319) -- (173,335) ;
\draw [color={rgb, 255:red, 0; green, 0; blue, 0 }  ,draw opacity=1 ]   (173,344) -- (141,360) ;
\draw [color={rgb, 255:red, 0; green, 0; blue, 0 }  ,draw opacity=1 ]   (132,320) -- (132,359) ;
\draw [color={rgb, 255:red, 0; green, 0; blue, 0 }  ,draw opacity=1 ] [dash pattern={on 4.5pt off 4.5pt}]  (131.18,455.32) -- (131.89,420) ;
\draw [color={rgb, 255:red, 0; green, 0; blue, 0 }  ,draw opacity=1 ]   (141,419) -- (173,435) ;
\draw [color={rgb, 255:red, 0; green, 0; blue, 0 }  ,draw opacity=1 ]   (173,443.91) -- (139.24,460.46) ;
\draw [color={rgb, 255:red, 0; green, 0; blue, 0 }  ,draw opacity=1 ]   (223,369) -- (191,385) ;
\draw [color={rgb, 255:red, 0; green, 0; blue, 0 }  ,draw opacity=1 ]   (191,394) -- (223,410) ;
\draw [color={rgb, 255:red, 0; green, 0; blue, 0 }  ,draw opacity=1 ]   (232,409) -- (232,370) ;
\draw [color={rgb, 255:red, 0; green, 0; blue, 0 }  ,draw opacity=1 ]   (132,409) -- (132,370) ;
\draw [color={rgb, 255:red, 0; green, 0; blue, 0 }  ,draw opacity=1 ] [dash pattern={on 4.5pt off 4.5pt}]  (141,410) -- (173,394) ;
\draw [color={rgb, 255:red, 0; green, 0; blue, 0 }  ,draw opacity=1 ]   (191,435) -- (223,419) ;
\draw [color={rgb, 255:red, 0; green, 0; blue, 0 }  ,draw opacity=1 ]   (223,360) -- (191,344) ;
\draw [color={rgb, 255:red, 0; green, 0; blue, 0 }  ,draw opacity=1 ]   (141,369) -- (173,385) ;
\draw [color={rgb, 255:red, 208; green, 2; blue, 27 }  ,draw opacity=1 ] [dash pattern={on 4.5pt off 4.5pt}]  (191,389.5) -- (423,389.5) ;
\draw [color={rgb, 255:red, 208; green, 2; blue, 27 }  ,draw opacity=1 ]   (236.21,409) -- (328.79,288) ;
\draw [color={rgb, 255:red, 208; green, 2; blue, 27 }  ,draw opacity=1 ]   (136.2,309) -- (227.8,189) ;
\draw [color={rgb, 255:red, 0; green, 0; blue, 0 }  ,draw opacity=1 ]   (553.96,97.26) -- (535.7,60) ;
\draw [color={rgb, 255:red, 0; green, 0; blue, 0 }  ,draw opacity=1 ]   (530.25,60) -- (510.75,99) ;
\draw [color={rgb, 255:red, 0; green, 0; blue, 0 }  ,draw opacity=1 ] [dash pattern={on 4.5pt off 4.5pt}]  (548.82,105.32) -- (517,104.68) ;
\draw [color={rgb, 255:red, 0; green, 0; blue, 0 }  ,draw opacity=1 ] [dash pattern={on 4.5pt off 4.5pt}]  (416.17,103.86) -- (448.83,105.14) ;
\draw [color={rgb, 255:red, 0; green, 0; blue, 0 }  ,draw opacity=1 ]   (453.96,97.26) -- (435.7,60) ;
\draw [color={rgb, 255:red, 0; green, 0; blue, 0 }  ,draw opacity=1 ]   (430.08,60) -- (411.3,95.39) ;
\draw [color={rgb, 255:red, 0; green, 0; blue, 0 }  ,draw opacity=1 ]   (503.96,11.74) -- (487.04,46.26) ;
\draw [color={rgb, 255:red, 0; green, 0; blue, 0 }  ,draw opacity=1 ]   (478.96,46.26) -- (460.7,9) ;
\draw [color={rgb, 255:red, 0; green, 0; blue, 0 }  ,draw opacity=1 ]   (467,3.5) -- (498.82,3.5) ;
\draw [color={rgb, 255:red, 208; green, 2; blue, 27 }  ,draw opacity=1 ] [dash pattern={on 4.5pt off 4.5pt}]  (467.18,105.32) -- (499,104.68) ;
\draw [color={rgb, 255:red, 208; green, 2; blue, 27 }  ,draw opacity=1 ]   (462.04,97.26) -- (478.96,62.74) ;
\draw [color={rgb, 255:red, 208; green, 2; blue, 27 }  ,draw opacity=1 ]   (435.7,49) -- (455.3,9) ;
\draw [color={rgb, 255:red, 208; green, 2; blue, 27 }  ,draw opacity=1 ]   (512.04,11.74) -- (530.3,49) ;
\draw [color={rgb, 255:red, 208; green, 2; blue, 27 }  ,draw opacity=1 ]   (505.25,99) -- (487.11,62.71) ;
\draw    (109.86,98) -- (130.14,59) ;
\draw    (135.75,59) -- (155.25,98) ;
\draw    (149,103.5) -- (116,103.5) ;

\end{tikzpicture}
}
     \caption{The graphs $\widetilde{H}_4^1$, $\widetilde{H}_4^2$, and $\widetilde{H}_4^3$.
    Red edges represent moves of the largest disc.
    Dashed edges show the optimal path for objective~\textup{\texttt{(a)}}.
    The red circle marks the initial state;
    green circles mark the final states of objectives~\textup{\texttt{(a)}}--\textup{\texttt{(d)}}.}
    \label{fig:graphs_123}
\end{figure}

\subsection*{Counting vertices and edges}

\begin{proposition}[Vertex count]\label{prop:vertex_count}
For all $n\ge 0$,

\begin{equation}
    |V(\widetilde{H}_4^{n})|=3^{n}.
\end{equation}

\end{proposition}

\begin{proof}
Each disc $d$ may occupy exactly three pegs: $0$, $p(d)$, or $3$. Hence there are $3^{n}$ admissible words $s\in Q^{n}$, one for each state.\qedhere
\end{proof}

\begin{proposition}[Edge-count recurrence]\label{prop:edge_recurrence}
The number of edges in $\widetilde{H}_4^{n}$, satisfy the following recurrence. 

\begin{equation}
    |E(\widetilde{H}_4^{0})|=0,\quad\forall n\geq 1:|E(\widetilde{H}_4^{n})| \;=\; 3\,|E(\widetilde{H}_4^{\,n-1})| \;+\; 2^{\left\lfloor \frac{n}{2}\right\rfloor+1} \;+\; 1.
\end{equation}

\end{proposition}

\begin{proof}
Internal edges: for each fixed position of disc $n$ there is an induced copy of $\widetilde{H}_4^{n-1}$, contributing $|E(\widetilde{H}_4^{n-1})|$ edges; three such copies give $3\,|E(\widetilde{H}_4^{n-1})|$.

Bridges (moves of disc $n$): disc $n$ may move between any two of its three allowed pegs $ \Pi_{d}$, provided no smaller disc lies on either of those two pegs.
\begin{itemize}
\item Between $0$ and $p(n)$: peg $p(n)$ must be empty. All smaller discs of the \emph{same} parity as $n$ must therefore be on peg $3$, while the opposite-parity smaller discs may be anywhere among their two allowed pegs; this yields $2^{\lfloor \frac{n}{2}\rfloor}$ bridge edges.
\item Between $p(n)$ and $3$: symmetrically, peg $3$ must be empty; again we obtain $2^{\lfloor \frac{n}{2}\rfloor}$ bridge edges.
\item Between $0$ and $3$: both pegs $0$ and $3$ must be free of smaller discs, forcing a unique parity-separated arrangement of the $n-1$ smaller discs on pegs $1$ and $2$; hence exactly one bridge edge.
\end{itemize}
Summing gives $2\cdot 2^{\lfloor \frac{n}{2}\rfloor}+1 = 2^{\lfloor \frac{n}{2}\rfloor+1}+1$ bridges, proving the recurrence.\qedhere
\end{proof}

\begin{proposition}[Edge-count closed form]\label{prop:edge_closed}
For all $n\ge 0$,
\begin{equation}\label{eq:HnmEdgesParity}
    |E(\widetilde{H}_4^{n})|=
\begin{cases}
\dfrac{3^{\,n+3}-20\cdot 2^{\frac{n}{2}}-7}{14}, & \text{$n$ even},\\[10pt]
\dfrac{3^{\,n+3}-32\cdot 2^{\frac{n-1}{2}}-7}{14}, & \text{$n$ odd}.
\end{cases}
\end{equation}
\end{proposition}

\begin{proof}
Unfold the linear inhomogeneous recurrence of Proposition~\ref{prop:edge_recurrence} separately for even and odd $n$, using $|E(\widetilde{H}_4^{0})|=0$ and the geometric sums
\(
\sum 3^{k} \) and \( \sum 2^{\lfloor \frac{k}{2}\rfloor}.
\)
A routine calculation yields the stated expressions.\qedhere
\end{proof}

The number of edges in Hanoi graphs $H_{m}^{n}$ has been computed using various approaches in \cite{KlavzarCombinatoricstopmost, HINZ202219, mehiri2024, ArettColoring}. In Proposition~\ref{prop:edges_compa}, we compare the number of edges in $\widetilde{H}_4^{n}$ with those of the three-peg and four-peg Hanoi graphs.

\begin{proposition}[Edges number comparison]
\label{prop:edges_compa}
For all $n \ge 0$,
\begin{equation}
    |E(H_{3}^{n})| \;\le\; |E(\widetilde{H}_4^{n})| \;\le\; |E(H_{4}^{n})|.
\end{equation}
\end{proposition}

\begin{proof}
The left inequality follows by induction on $n$, using that in $\widetilde{H}_4^n$ each disc may occupy three pegs, whereas in $H_3^n$ it may occupy only two. The right inequality holds since $\widetilde{H}_4^n$ is obtained from $H_{4}^{n}$ by removing all states violating the parity constraints (see Theorem \ref{theo:containment}).\qedhere
\end{proof}

Table~\ref{tab:firstvaluesEdgenbr} lists the first values of the three edge-count sequences.

\begin{table}
\caption{First eleven values of $|E(H_{3}^{n})|$, $|E(H_{4}^{n})|$, and $|E(\widetilde{H}_4^{n})|$.}
\label{tab:firstvaluesEdgenbr}
\resizebox{\columnwidth}{!}{
\begin{tabular*}{\linewidth}{@{\extracolsep{\fill}}c|lllllllllll@{}}
\toprule
$n$ & 0 & 1 & 2 & 3 & 4 & 5 & 6 & 7 & 8 & 9 & 10 \\\midrule
$|E(H_{3}^{n})|$ & 0 & 3 & 12 & 39 & 120 & 363 & 1092 & 3279 & 9840 & 29523 & 88572 \\
$|E(\widetilde{H}_4^{n})|$ & 0 & 3 & 14 & 47 & 150 & 459 & 1394 & 4199 & 12630 & 37923 & 113834\\
$|E(H_{4}^{n})|$ & 0 & 6 & 36 & 168 & 720 & 2976 & 12096 & 48768 & 195840 & 784896 & 3142656 \\
\bottomrule
\end{tabular*}
}

\end{table}

 Using the classical formula
\begin{equation}\label{eq:HnmEdges}
    |E(H_m^n)|=\frac{1}{4}\,m(m-1)\bigl(m^n-(m-2)^n\bigr)
\end{equation}
for the number of edges of the \(p\)-peg Hanoi graph (see \cite[Proposition~5.43]{HinzMythsMaths2018}), together with Proposition~\ref{prop:edge_closed}, we obtain the following asymptotic comparisons.

\begin{proposition}[Asymptotic comparison of the number of edges]
\label{prop:edge-asymptotics}
The parity-constrained four-peg Hanoi graph satisfies
\begin{equation}\label{eq:edge-asymptotics}
\lim_{n\to\infty}\frac{|E(\widetilde{H}_4^n)|}{|E(H_3^n)|}=\frac{9}{7}.
\end{equation}
In particular, asymptotically, \(\widetilde{H}_4^n\) has about \(28.57\%\) more edges than the classical three-peg Hanoi graph \(H_3^n\).

Moreover,
\begin{equation}
    \lim_{n\to\infty}\frac{|E(H_4^n)|}{|E(\widetilde{H}_4^n)|}=\lim_{n\to\infty}\frac{14}{9}\left(\frac{4}{3}\right)^n=+\infty.
\end{equation}
\end{proposition}

\begin{proof}
For \(p=3\), Formula \eqref{eq:HnmEdges} gives $|E(H_3^n)|=\frac{3}{2}(3^n-1)$. In both cases of \eqref{eq:HnmEdgesParity},
\[
|E(\widetilde{H}_4^n)|\sim \frac{3^{n+3}}{14}.
\]
Therefore,
\[
\frac{|E(\widetilde{H}_4^n)|}{|E(H_3^n)|}
\sim
\frac{3^{n+3}/14}{(3/2)(3^n-1)}
\sim
\frac{3^{n+3}/14}{(3/2)3^n}
=
\frac{27/14}{3/2}
=
\frac{9}{7}\approx 1.2857.
\]
For \(p=4\), Formula \eqref{eq:HnmEdges} gives $|E(H_4^n)|=3(4^n-2^n)$. Since $|E(H_4^n)|\sim 3\cdot 4^n$ we get
\[
\lim_{n\to\infty}\frac{|E(H_4^n)|}{|E(\widetilde{H}_4^n)|}=\lim_{n\to\infty}\frac{14}{9}\left(\frac{4}{3}\right)^n=+\infty.\qedhere
\]
\end{proof}

\subsection*{Degrees}

\begin{proposition}[Minimum and maximum degrees]\label{prop:degrees}
For all $n\ge 1$,
\begin{equation}
    \delta(\widetilde{H}_4^n)=2,
\end{equation}
\begin{equation}
    \Delta(\widetilde{H}_4^n)=
\begin{cases}
2,& n=1,\\[2pt]
4,& n=2,\\[2pt]
5,& n\ge 3.
\end{cases}
\end{equation}
\end{proposition}

\begin{proof}
Clearly $\delta(\widetilde{H}_4^n)\neq 1$. On the other hand, in any perfect state (all discs on a single peg), the only movable disc is $1$. By parity, disc $1$ cannot use peg $E$ (peg~$1$), so it has exactly two legal moves. Hence $\delta(\widetilde{H}_4^n)=2$.

 For $n=1$, the graph is a $3$-cycle, so $\Delta(\widetilde{H}_4^1)=2$. For $n=2$, one checks directly that the maximum degree is $\Delta(\widetilde{H}_4^2)=4$ from Figure \ref{fig:graphs_123}. For $n\ge 3$, consider a configuration where discs $1,2,3$ are simultaneously unblocked and occupy distinct allowed pegs; in $H_4^n$ this gives $2+2+2=6$ moves, but the parity restriction removes one move for disc~$1$, namely, the move between pegs $1$ and $2$, leaving $2+2+1=5$ total allowed moves in a such state. Such a state exists, so $\Delta(\widetilde{H}_4^n)=5$ for $n\ge 3$.\qedhere
\end{proof}

\begin{corollary}\label{cor:degree-compare}
For all $n\ge 1$,
\begin{align}
\delta(\widetilde{H}_4^n)&=\delta(H_3^n),\\
\Delta(\widetilde{H}_4^n)&<\Delta(H_4^n).
\end{align}
\end{corollary}

\begin{proposition}[Average degree]\label{prop:avg-degree}
Let $\overline{d}(\widetilde{H}_4^n)$ denote the average degree. Then, using $|V(\widetilde{H}_4^n)|=3^n$ and Proposition~\ref{prop:edge_closed},
\begin{equation}
 \overline{d}(\widetilde{H}_4^n)=\frac{2\,|E(\widetilde{H}_4^n)|}{|V(\widetilde{H}_4^n)|}
=
\begin{cases}
\displaystyle \frac{27}{7}\;-\;\frac{20}{7}\,\frac{2^{\frac{n}{2}}}{3^{n}}\;-\;\frac{1}{3^{n}}, & \text{$n$ even},\\[10pt]
\displaystyle \frac{27}{7}\;-\;\frac{32}{7}\,\frac{2^{\frac{n-1}{2}}}{3^{n}}\;-\;\frac{1}{3^{n}}, & \text{$n$ odd}.
\end{cases}
\end{equation}

In particular,
\begin{equation}
    \lim_{n\to\infty}\overline{d}(\widetilde{H}_4^n)=\frac{27}{7}\approx 3.8571.
\end{equation}
\end{proposition}

\begin{proof}
A straightforward computation using the closed form of $|E(\widetilde{H}_4^n)|$ from Proposition~\ref{prop:edge_closed}, together with $|V(\widetilde{H}_4^n)|=3^n$, yields the claim.\qedhere
\end{proof}

\subsection*{Planarity}

It is known that the only planar Hanoi graphs $H_{m}^{n}$ are $H_{3}^{n}$ (for all $n \ge 0$), together with $H_{4}^{0}$, $H_{4}^{1}$, and $H_{4}^{2}$; whereas $H_{4}^{n}$ is non-planar for all $n \ge 3$ \cite{HINZ2002263}. In Theorem~\ref{theo:planarity}, we show that, similarly to the case of $H_{4}^{n}$, the graphs $\widetilde{H}_4^{n}$ are planar only for $n \le 2$, and become non-planar for all $n \ge 3$.

\begin{theorem}[Planarity]
\label{theo:planarity}
The graphs $\widetilde{H}_4^0$, $\widetilde{H}_4^1$, and $\widetilde{H}_4^2$ are planar, whereas $\widetilde{H}_4^n$ is non-planar for all $n \ge 3$.
\end{theorem}

\begin{proof}
Figure~\ref{fig:graphs_123} shows planar embeddings of $\widetilde{H}_4^1$ and $\widetilde{H}_4^2$, and the planarity of $\widetilde{H}_4^0$ is trivial.

For $n \ge 3$, observe that $\widetilde{H}_4^3 \subset \widetilde{H}_4^n$; thus, it suffices to prove that $\widetilde{H}_4^3$ is non-planar.
Consider the subgraph of $\widetilde{H}_4^3$ induced by the vertex set
\[
\{032,012,013,002,010,213,003,033,030,210,212,232,233,310,312\}.
\]
After deleting the four edges $\{210,212\}$, $\{013,012\}$, $\{012,032\}$, and $\{030,033\}$, 
the resulting subgraph contains a subdivision of the complete bipartite graph $K_{3,3}$.
Figure~\ref{fig:planarity_prooft} illustrates two embeddings of this subdivision: 
one aligned with the vertex positions used in Figure~\ref{fig:graphs_123}, and another using a standard layout of $K_{3,3}$. The red and blue vertices represent the two partite sets of $K_{3,3}$, while the green vertices correspond to subdivision vertices.\qedhere
\end{proof}

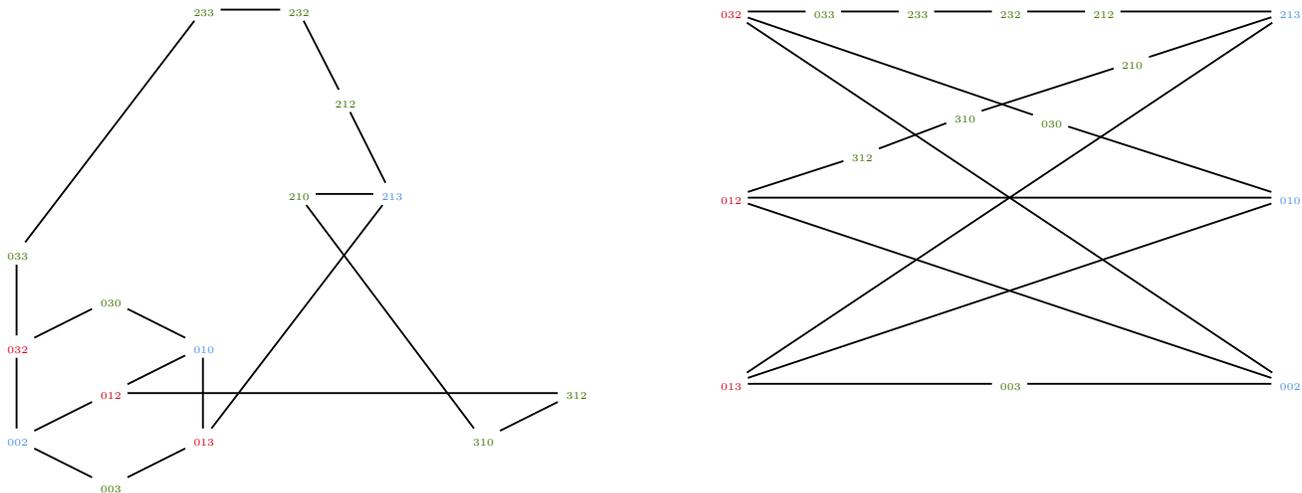
\begin{figure}[ht]
    \centering

\tikzset{every picture/.style={line width=0.75pt}} 
\resizebox{\columnwidth}{!}{

\begin{tikzpicture}[x=0.75pt,y=0.75pt,yscale=-1,xscale=1]

\draw (101,2.4) node [anchor=north west][inner sep=0.75pt]  [font=\fontsize{0.35em}{0.42em}\selectfont,color={rgb, 255:red, 65; green, 117; blue, 5 }  ,opacity=1 ,]  {$233$};
\draw (152,2.4) node [anchor=north west][inner sep=0.75pt]  [font=\fontsize{0.35em}{0.42em}\selectfont,color={rgb, 255:red, 65; green, 117; blue, 5 }  ,opacity=1 ,]  {$232$};
\draw (152,101.4) node [anchor=north west][inner sep=0.75pt]  [font=\fontsize{0.35em}{0.42em}\selectfont,color={rgb, 255:red, 65; green, 117; blue, 5 }  ,opacity=1 ,]  {$210$};
\draw (202,101.4) node [anchor=north west][inner sep=0.75pt]  [font=\fontsize{0.35em}{0.42em}\selectfont,color={rgb, 255:red, 74; green, 144; blue, 226 }  ,opacity=1 ,]  {$213$};
\draw (177,51.4) node [anchor=north west][inner sep=0.75pt]  [font=\fontsize{0.35em}{0.42em}\selectfont,color={rgb, 255:red, 65; green, 117; blue, 5 }  ,opacity=1 ,]  {$212$};
\draw (251,233.4) node [anchor=north west][inner sep=0.75pt]  [font=\fontsize{0.35em}{0.42em}\selectfont,color={rgb, 255:red, 65; green, 117; blue, 5 }  ,opacity=1 ,]  {$310$};
\draw (301,208.4) node [anchor=north west][inner sep=0.75pt]  [font=\fontsize{0.35em}{0.42em}\selectfont,color={rgb, 255:red, 65; green, 117; blue, 5 }  ,opacity=1 ,]  {$312$};
\draw (1,133.4) node [anchor=north west][inner sep=0.75pt]  [font=\fontsize{0.35em}{0.42em}\selectfont,color={rgb, 255:red, 65; green, 117; blue, 5 }  ,opacity=1 ,]  {$033$};
\draw (51,158.4) node [anchor=north west][inner sep=0.75pt]  [font=\fontsize{0.35em}{0.42em}\selectfont,color={rgb, 255:red, 65; green, 117; blue, 5 }  ,opacity=1 ,]  {$030$};
\draw (1,183.4) node [anchor=north west][inner sep=0.75pt]  [font=\fontsize{0.35em}{0.42em}\selectfont,color={rgb, 255:red, 208; green, 2; blue, 27 }  ,opacity=1 ,]  {$032$};
\draw (51,258.4) node [anchor=north west][inner sep=0.75pt]  [font=\fontsize{0.35em}{0.42em}\selectfont,color={rgb, 255:red, 65; green, 117; blue, 5 }  ,opacity=1 ,]  {$003$};
\draw (1,233.4) node [anchor=north west][inner sep=0.75pt]  [font=\fontsize{0.35em}{0.42em}\selectfont,color={rgb, 255:red, 74; green, 144; blue, 226 }  ,opacity=1 ,]  {$002$};
\draw (101,233.4) node [anchor=north west][inner sep=0.75pt]  [font=\fontsize{0.35em}{0.42em}\selectfont,color={rgb, 255:red, 208; green, 2; blue, 27 }  ,opacity=1 ,]  {$013$};
\draw (101,183.4) node [anchor=north west][inner sep=0.75pt]  [font=\fontsize{0.35em}{0.42em}\selectfont,color={rgb, 255:red, 74; green, 144; blue, 226 }  ,opacity=1 ,]  {$010$};
\draw (51,208.4) node [anchor=north west][inner sep=0.75pt]  [font=\fontsize{0.35em}{0.42em}\selectfont,color={rgb, 255:red, 208; green, 2; blue, 27 }  ,opacity=1 ,]  {$012$};
\draw (484,3.4) node [anchor=north west][inner sep=0.75pt]  [font=\fontsize{0.35em}{0.42em}\selectfont,color={rgb, 255:red, 65; green, 117; blue, 5 }  ,opacity=1 ,]  {$233$};
\draw (534,3.4) node [anchor=north west][inner sep=0.75pt]  [font=\fontsize{0.35em}{0.42em}\selectfont,color={rgb, 255:red, 65; green, 117; blue, 5 }  ,opacity=1 ,]  {$232$};
\draw (599.33,30.97) node [anchor=north west][inner sep=0.75pt]  [font=\fontsize{0.35em}{0.42em}\selectfont,color={rgb, 255:red, 65; green, 117; blue, 5 }  ,opacity=1 ,]  {$210$};
\draw (684,3.4) node [anchor=north west][inner sep=0.75pt]  [font=\fontsize{0.35em}{0.42em}\selectfont,color={rgb, 255:red, 74; green, 144; blue, 226 }  ,opacity=1 ,]  {$213$};
\draw (584,3.4) node [anchor=north west][inner sep=0.75pt]  [font=\fontsize{0.35em}{0.42em}\selectfont,color={rgb, 255:red, 65; green, 117; blue, 5 }  ,opacity=1 ,]  {$212$};
\draw (509.5,59.8) node [anchor=north west][inner sep=0.75pt]  [font=\fontsize{0.35em}{0.42em}\selectfont,color={rgb, 255:red, 65; green, 117; blue, 5 }  ,opacity=1 ,]  {$310$};
\draw (454.33,80.63) node [anchor=north west][inner sep=0.75pt]  [font=\fontsize{0.35em}{0.42em}\selectfont,color={rgb, 255:red, 65; green, 117; blue, 5 }  ,opacity=1 ,]  {$312$};
\draw (434,3.4) node [anchor=north west][inner sep=0.75pt]  [font=\fontsize{0.35em}{0.42em}\selectfont,color={rgb, 255:red, 65; green, 117; blue, 5 }  ,opacity=1 ,]  {$033$};
\draw (556,62.13) node [anchor=north west][inner sep=0.75pt]  [font=\fontsize{0.35em}{0.42em}\selectfont,color={rgb, 255:red, 65; green, 117; blue, 5 }  ,opacity=1 ,]  {$030$};
\draw (384,3.4) node [anchor=north west][inner sep=0.75pt]  [font=\fontsize{0.35em}{0.42em}\selectfont,color={rgb, 255:red, 208; green, 2; blue, 27 }  ,opacity=1 ,]  {$032$};
\draw (684,203.4) node [anchor=north west][inner sep=0.75pt]  [font=\fontsize{0.35em}{0.42em}\selectfont,color={rgb, 255:red, 74; green, 144; blue, 226 }  ,opacity=1 ,]  {$002$};
\draw (384,203.4) node [anchor=north west][inner sep=0.75pt]  [font=\fontsize{0.35em}{0.42em}\selectfont,color={rgb, 255:red, 208; green, 2; blue, 27 }  ,opacity=1 ,]  {$013$};
\draw (684,103.4) node [anchor=north west][inner sep=0.75pt]  [font=\fontsize{0.35em}{0.42em}\selectfont,color={rgb, 255:red, 74; green, 144; blue, 226 }  ,opacity=1 ,]  {$010$};
\draw (384,103.4) node [anchor=north west][inner sep=0.75pt]  [font=\fontsize{0.35em}{0.42em}\selectfont,color={rgb, 255:red, 208; green, 2; blue, 27 }  ,opacity=1 ,]  {$012$};
\draw (534,203.4) node [anchor=north west][inner sep=0.75pt]  [font=\fontsize{0.35em}{0.42em}\selectfont,color={rgb, 255:red, 65; green, 117; blue, 5 }  ,opacity=1 ,]  {$003$};
\draw [color={rgb, 255:red, 0; green, 0; blue, 0 }  ,draw opacity=1 ]   (117,4) -- (149,4) ;
\draw [color={rgb, 255:red, 0; green, 0; blue, 0 }  ,draw opacity=1 ]   (205.5,97) -- (186.5,59) ;
\draw [color={rgb, 255:red, 0; green, 0; blue, 0 }  ,draw opacity=1 ]   (168,103) -- (199,103) ;
\draw [color={rgb, 255:red, 0; green, 0; blue, 0 }  ,draw opacity=1 ]   (161.56,10) -- (180.44,47) ;
\draw [color={rgb, 255:red, 0; green, 0; blue, 0 }  ,draw opacity=1 ]   (298,214.75) -- (267,230.25) ;
\draw    (163,109) -- (253,229) ;
\draw [color={rgb, 255:red, 0; green, 0; blue, 0 }  ,draw opacity=1 ]   (48,164.75) -- (17,180.25) ;
\draw [color={rgb, 255:red, 0; green, 0; blue, 0 }  ,draw opacity=1 ]   (7.5,141) -- (7.5,179) ;
\draw [color={rgb, 255:red, 0; green, 0; blue, 0 }  ,draw opacity=1 ]   (17,239.75) -- (48,255.25) ;
\draw [color={rgb, 255:red, 0; green, 0; blue, 0 }  ,draw opacity=1 ]   (98,189.75) -- (67,205.25) ;
\draw [color={rgb, 255:red, 0; green, 0; blue, 0 }  ,draw opacity=1 ]   (107.5,229) -- (107.5,191) ;
\draw [color={rgb, 255:red, 0; green, 0; blue, 0 }  ,draw opacity=1 ]   (7.5,229) -- (7.5,191) ;
\draw [color={rgb, 255:red, 0; green, 0; blue, 0 }  ,draw opacity=1 ]   (17,230.25) -- (48,214.75) ;
\draw [color={rgb, 255:red, 0; green, 0; blue, 0 }  ,draw opacity=1 ]   (67,255.25) -- (98,239.75) ;
\draw [color={rgb, 255:red, 0; green, 0; blue, 0 }  ,draw opacity=1 ]   (98,180.25) -- (67,164.75) ;
\draw    (67,210) -- (298,210) ;
\draw    (112.09,229) -- (203.91,109) ;
\draw    (12.08,129) -- (102.92,10) ;
\draw [color={rgb, 255:red, 0; green, 0; blue, 0 }  ,draw opacity=1 ]   (500,5) -- (531,5) ;
\draw [color={rgb, 255:red, 0; green, 0; blue, 0 }  ,draw opacity=1 ]   (681,5) -- (600,5) ;
\draw [color={rgb, 255:red, 0; green, 0; blue, 0 }  ,draw opacity=1 ]   (615.33,29.47) -- (681,8.09) ;
\draw [color={rgb, 255:red, 0; green, 0; blue, 0 }  ,draw opacity=1 ]   (550,5) -- (581,5) ;
\draw [color={rgb, 255:red, 0; green, 0; blue, 0 }  ,draw opacity=1 ]   (470.33,78.65) -- (506.5,64.99) ;
\draw [line width=0.75]    (596.33,35.62) -- (525.5,58.35) ;
\draw [color={rgb, 255:red, 0; green, 0; blue, 0 }  ,draw opacity=1 ]   (553,60.49) -- (400,8.24) ;
\draw [color={rgb, 255:red, 0; green, 0; blue, 0 }  ,draw opacity=1 ]   (431,5) -- (400,5) ;
\draw [color={rgb, 255:red, 0; green, 0; blue, 0 }  ,draw opacity=1 ][line width=0.75]    (681,105) -- (400,105) ;
\draw [color={rgb, 255:red, 0; green, 0; blue, 0 }  ,draw opacity=1 ][line width=0.75]    (400,201.83) -- (681,108.17) ;
\draw [color={rgb, 255:red, 0; green, 0; blue, 0 }  ,draw opacity=1 ][line width=0.75]    (681.5,199) -- (399.5,11) ;
\draw [color={rgb, 255:red, 0; green, 0; blue, 0 }  ,draw opacity=1 ][line width=0.75]    (681,201.83) -- (400,108.17) ;
\draw [color={rgb, 255:red, 0; green, 0; blue, 0 }  ,draw opacity=1 ]   (681,101.94) -- (572,66.8) ;
\draw [line width=0.75]    (400,101.93) -- (451.33,85.31) ;
\draw [line width=0.75]    (399.5,199) -- (681.5,11) ;
\draw    (450,5) -- (481,5) ;
\draw    (400,205) -- (531,205) ;
\draw    (550,205) -- (681,205) ;

\end{tikzpicture}

}
     \caption{Two embeddings of a subdivision of $K_{3,3}$ contained in $\widetilde{H}_4^{3}$.}
    \label{fig:planarity_prooft}
\end{figure}

\subsection*{Connectivity}

Many properties of the parity-constrained Hanoi graphs $\widetilde{H}_4^{n}$ can be deduced from their recursive structure, such as their connectedness or, more precisely, their $2$-connectedness.

\begin{theorem}[Vertex connectivity]
\label{theo:connectivity}
For all $n \ge 1$, the vertex connectivity of $\widetilde{H}_4^{n}$ satisfies

\begin{equation}
    \kappa(\widetilde{H}_4^{n}) = 2.
\end{equation}

\end{theorem}

\begin{proof}

Any perfect-state vertex has degree $2$, since only disc $1$ can move from such a state,
and disc $1$ has exactly two allowed destinations. Hence, $\kappa(\widetilde{H}_4^n)\le \delta(\widetilde{H}_4^n)= 2$.

To show that the deletion of only one vertex is not sufficient to disconnect the graph, we proceed by induction on~$n$. For $n=1$, the graph $\widetilde{H}_4^1 \cong K_3$, which is $2$-connected.

Assume now that $\widetilde{H}_4^{n-1}$ is $2$-connected for some $n \ge 2$. 
In the recursive construction of $\widetilde{H}_4^n$, the graph is composed of three copies 
$0\widetilde{H}_4^{n-1}$, $p(n)\widetilde{H}_4^{n-1}$, and $3\widetilde{H}_4^{n-1}$, which are interconnected by edges 
corresponding to moves of the largest disc $n$, called a \textit{bridge}.

Removing a single vertex that has all its neighbors within the same copy of $\widetilde{H}_4^{n-1}$ does not disconnect that copy by the induction hypothesis. 
Thus, it remains to consider the removal of a \emph{frontier vertex}, i.e., a vertex that has at least one neighbor in another copy.

Between any two adjacent copies, say $0\widetilde{H}_4^{n-1}$ and $p(n)\widetilde{H}_4^{n-1}$, there exist at least two bridges:  
one where peg $\overline{p}(n)$ is empty, and another where peg $\overline{p}(n)$ is not empty 
(the only exception is the case $n=1$, which has already been handled as the base case). 
A similar argument applies to the pair $p(n)\widetilde{H}_4^{n-1}$ and $3\widetilde{H}_4^{n-1}$.

Consequently, removing a frontier vertex corresponds to deleting exactly one bridge, which does not disconnect the three copies. Hence, the graph $\widetilde{H}_4^{n}$ remains connected after the removal of any single vertex.

Therefore, no single vertex removal can disconnect the graph, and we conclude that $\kappa(\widetilde{H}_4^n) \ge 2$.
Combining both inequalities yields $\kappa(\widetilde{H}_4^n) = 2$.\qedhere
\end{proof}

\begin{theorem}[Edge connectivity]
\label{theo:edgeconnectivity}
For all $n \ge 1$, the edge connectivity of $\widetilde{H}_4^{n}$ satisfies
\begin{equation}
    \lambda(\widetilde{H}_4^{n}) = 2.
\end{equation}
\end{theorem}

\begin{proof}
From Proposition~\ref{prop:degrees}, we know that $\delta(\widetilde{H}_4^{n}) = 2$, 
and by Theorem~\ref{theo:connectivity}, we have $\kappa(\widetilde{H}_4^{n}) = 2$.
Since, for any graph $G$, it holds that
\[
\kappa(G) \le \lambda(G) \le \delta(G),
\]
we obtain
\[
2 = \kappa(\widetilde{H}_4^{n}) \le \lambda(\widetilde{H}_4^{n}) \le \delta(\widetilde{H}_4^{n}) = 2.
\]
Hence, $\lambda(\widetilde{H}_4^{n}) = 2$.\qedhere
\end{proof}

Since $\kappa(\widetilde{H}_4^{n}) = \lambda(\widetilde{H}_4^{n}) = \delta(\widetilde{H}_4^{n})$, we may state the following result.

\begin{proposition}
For all $n \ge 1$, the parity-constrained Hanoi graph $\widetilde{H}_4^{n}$ is maximally connected; that is, it is both maximally vertex-connected and maximally edge-connected.
\end{proposition}

\subsection*{Subgraph relations to Hanoi graphs}

Theorem~\ref{theo:containment} establishes the subgraph relationships between $\widetilde{H}_4^{n}$ and the classical Hanoi graphs on three and four pegs.

\begin{theorem}[Subgraph inclusion]
\label{theo:containment}
For all $n \ge 0$,

\begin{equation}
    H^{\left\lceil \frac{n}{2} \right\rceil}_{3} \;\subseteq\; \widetilde{H}_4^{n} \;\subseteq\; H^{n}_{4}.
\end{equation}

\end{theorem}

\begin{proof}
The right inclusion is immediate: $\widetilde{H}_4^{n}$ is obtained from $H^{n}_{4}$ by deleting states
and edges that violate parity constraints, hence $\widetilde{H}_4^{n}$ is a subgraph of $H^{n}_{4}$.

For the left inclusion, fix the majority parity among $\{1,\dots,n\}$; there are
$\left\lceil \frac{n}{2} \right\rceil$ discs of that parity. These discs may move freely among the three pegs
$\{0,3,p\}$ with $p\in\{1,2\}$ matching that parity. The induced subgraph on these discs
(with the others kept fixed) is isomorphic to $H^{\left\lceil \frac{n}{2} \right\rceil}_{3}$.\qedhere
\end{proof}

\begin{corollary}[Largest embedded 3-peg sub-Hanoi graph]
\label{cor:largestsub}
For $n \ge 1$, a largest subgraph of $\widetilde{H}_4^{n}$ isomorphic to $H^{\left\lceil \frac{n}{2} \right\rceil}_{3}$ is given by

\begin{equation}
    V\!\left(H^{\left\lceil \frac{n}{2} \right\rceil}_{3}\right)
=
\begin{cases}
\{\, s \in V(\widetilde{H}_4^{n}) \mid s_{d}=2 \text{ for all odd } d \,\},  & \text{$n$ even}\\[4pt]
\{\, s \in V(\widetilde{H}_4^{n}) \mid s_{d}=1 \text{ for all even } d \,\}, & \text{$n$ odd},.
\end{cases}
\end{equation}

\end{corollary}

Figure~\ref{fig:subHanoiGraphinP3} illustrates the largest Hanoi subgraph of $\widetilde{H}_4^{3}$ that is isomorphic to $H_{3}^{2}$.

\begin{figure}[ht]
    \centering

\tikzset{every picture/.style={line width=0.75pt}} 

{

\begin{tikzpicture}[x=0.75pt,y=0.75pt,yscale=-1,xscale=1]

\draw (278,72.4) node [anchor=north west][inner sep=0.75pt]  [font=\fontsize{0.35em}{0.42em}\selectfont,]  {$210$};
\draw (328,72.4) node [anchor=north west][inner sep=0.75pt]  [font=\fontsize{0.35em}{0.42em}\selectfont,]  {$213$};
\draw (303,22.4) node [anchor=north west][inner sep=0.75pt]  [font=\fontsize{0.35em}{0.42em}\selectfont,]  {$212$};
\draw (377,204.4) node [anchor=north west][inner sep=0.75pt]  [font=\fontsize{0.35em}{0.42em}\selectfont,]  {$310$};
\draw (377,154.4) node [anchor=north west][inner sep=0.75pt]  [font=\fontsize{0.35em}{0.42em}\selectfont,]  {$313$};
\draw (427,179.4) node [anchor=north west][inner sep=0.75pt]  [font=\fontsize{0.35em}{0.42em}\selectfont,]  {$312$};
\draw (227,204.4) node [anchor=north west][inner sep=0.75pt]  [font=\fontsize{0.35em}{0.42em}\selectfont,]  {$013$};
\draw (227,154.4) node [anchor=north west][inner sep=0.75pt]  [font=\fontsize{0.35em}{0.42em}\selectfont,]  {$010$};
\draw (177,179.4) node [anchor=north west][inner sep=0.75pt]  [font=\fontsize{0.35em}{0.42em}\selectfont,]  {$012$};
\draw [color={rgb, 255:red, 0; green, 0; blue, 0 }  ,draw opacity=1 ]   (331.25,68) -- (311.75,29) ;
\draw [color={rgb, 255:red, 0; green, 0; blue, 0 }  ,draw opacity=1 ]   (306.25,29) -- (286.75,68) ;
\draw [color={rgb, 255:red, 0; green, 0; blue, 0 }  ,draw opacity=1 ]   (293,73.5) -- (325,73.5) ;
\draw [color={rgb, 255:red, 0; green, 0; blue, 0 }  ,draw opacity=1 ]   (392,160) -- (424,176) ;
\draw [color={rgb, 255:red, 0; green, 0; blue, 0 }  ,draw opacity=1 ]   (424,185) -- (392,201) ;
\draw [color={rgb, 255:red, 0; green, 0; blue, 0 }  ,draw opacity=1 ]   (383,200) -- (383,161) ;
\draw [color={rgb, 255:red, 0; green, 0; blue, 0 }  ,draw opacity=1 ]   (288.13,79) -- (378.88,200) ;
\draw [color={rgb, 255:red, 0; green, 0; blue, 0 }  ,draw opacity=1 ]   (224,160) -- (192,176) ;
\draw [color={rgb, 255:red, 0; green, 0; blue, 0 }  ,draw opacity=1 ]   (192,185) -- (224,201) ;
\draw [color={rgb, 255:red, 0; green, 0; blue, 0 }  ,draw opacity=1 ]   (233,200) -- (233,161) ;
\draw [color={rgb, 255:red, 0; green, 0; blue, 0 }  ,draw opacity=1 ]   (192,180.5) -- (424,180.5) ;
\draw [color={rgb, 255:red, 0; green, 0; blue, 0 }  ,draw opacity=1 ]   (237.21,200) -- (329.79,79) ;

\end{tikzpicture}

}
\caption{Largest Hanoi subgraph (isomorphic to $H^{\lceil n/2\rceil}_3$) inside $\widetilde{H}_4^3$.}
    \label{fig:subHanoiGraphinP3}
\end{figure}
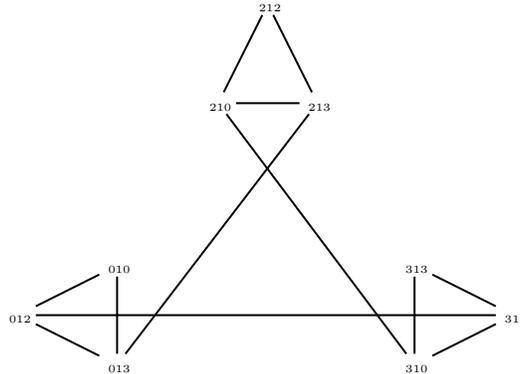

\begin{proposition}[Disconnection after removal]
\label{prop:disconexion}
If $n \ge 4$, then the graph $\widetilde{H}_4^{n} \setminus H^{\left\lceil \frac{n}{2} \right\rceil}_{3}$ is disconnected; otherwise, it is connected.

\end{proposition}

\begin{proof}
The claim is trivial for $n=0$ and can be checked directly for $n=1,2,3$ (see Figure \ref{fig:graphs_123}).
Assume $n \ge 4$. By Corollary~\ref{cor:largestsub}, removing $H^{\left\lceil \frac{n}{2} \right\rceil}_{3}$
forbids all states in which any disc of the majority parity sits on its parity peg.
Consequently, such discs can only move between the two neutral pegs $\{0,3\}$ in the remaining graph.

For $n$ even, fix any $(n-3)$-prefix $s$ and consider the six states
\[
s000,\ s003,\ s013,\ s010,\ s030,\ s033 .
\]
In $\widetilde{H}_4^{n}\setminus H^{\left\lceil \frac{n}{2} \right\rceil}_{3}$ (with $n$ even, odd discs cannot use peg $2$), these
form an induced path with no admissible continuation. Hence the graph is disconnected.

For $n$ odd, a similar argument uses parity reversal: even discs cannot use peg $1$,
so there is no path connecting $s0000$ to $s3333$ once $H^{\left\lceil \frac{n}{2} \right\rceil}_{3}$ is removed,
as moving  disc $4$ requires two intermediate stacks, stack including only disc $2$ on peg $1$, and another including discs $1$ and $3$ on peg $2$, which is impossible, sine disc $2$  is forbidden to use peg $1$.
Thus $\widetilde{H}_4^{n}\setminus H^{\left\lceil \frac{n}{2} \right\rceil}_{3}$ is disconnected for all $n \ge 4$.\qedhere
\end{proof}

\begin{proposition}[Number of largest sub-Hanoi graphs in $\widetilde{H}_4^n$]
For $n$ even, the graph $\widetilde{H}_4^{n}$ contains two distinct subgraphs isomorphic to $H^{\frac{n}{2}}_{3}$, each corresponding to configurations where exactly $\frac{n}{2}$ discs (either all even or all odd) are free to move under unconstrained three-peg rules. 

For $n$ odd, the graph $\widetilde{H}_4^{n}$ contains one subgraph isomorphic to $H^{\frac{n+1}{2}}_{3}$, corresponding to states where all even discs are placed on their parity peg $1$, and another subgraph isomorphic to $H^{\frac{n-1}{2}}_{3}$, corresponding to states where all odd discs are placed on their parity peg $2$.

In both cases, these subgraphs share a unique common vertex, namely the configuration in which all discs are placed on their respective parity pegs and are fully separated by parity.
\end{proposition}

\begin{proof}
The number of discs represented in each subgraph follows immediately from Corollary~\ref{cor:largestsub}. 
The uniqueness of the shared vertex is also clear: this state is the only configuration belonging to both vertex sets identified in Corollary~\ref{cor:largestsub}, as it is the sole configuration in which full parity separation occurs, allowing it to lie simultaneously at the base of both sub-Hanoi structures.\qedhere
\end{proof}

\subsection*{Diameter}

We performed breadth-first-search computations on graphs $\widetilde{H}_4^{n}$ in order to estimate their diameter. 
For $1\le n\le 8$, the diameter was computed exactly by exhaustive exploration of the whole graph. 
The obtained values are
\[
1,\,3,\,5,\,9,\,15,\,23,\,35,\,53,
\]
which coincide with the values of the full-tower transfer sequence $(a_n)$.

Additional heuristic computations for larger values of $n$ were also consistent with the identity $\mathrm{diam}(\widetilde{H}_4^{n})=a_n$. This leads to the following numerical table up to $n=20$.

\begin{table}
 \caption{Computed and predicted values of $\mathrm{diam}(\widetilde{H}_4^{n})$ for $1\le n\le 20$.}
\label{tab:diameter_hf4}
 \footnotesize
 \setlength{\tabcolsep}{3pt}
\begin{tabular*}{\linewidth}{@{\extracolsep{\fill}}c|llllllll|llllllllllll@{}}
\toprule
$n$&1&2&3&4&5&6&7&8&9&10&11&12&13&14&15&16&17&18&19&20\\
\midrule
$\mathrm{diam}(\widetilde{H}_4^{n})$
&1&3&5&9&15&23&35&53&77&113&163&235&335&481&681&973&1375&1959&2763&3933\\
\bottomrule
\end{tabular*}

\end{table}

\begin{conjecture}\label{conj:diameter_equals_a_n}
For every integer $n\ge 1$, the diameter of the state graph $\widetilde{H}_4^{n}$ is equal to the optimal number of moves for the full-tower transfer objective:
\begin{equation}
    \mathrm{diam}(\widetilde{H}_4^{n})=A_n=a_n.
\end{equation}
Equivalently, the maximum distance between two feasible states of $\widetilde{H}_4^{n}$ is attained between the two perfect states and is equal to the minimum number of moves required to transfer the full tower from $N_1$ to $N_2$.
\end{conjecture}

\subsection*{Automorphism}

\begin{proposition}[Automorphism group]\label{prop:aut}
For all $n \ge 1$, the automorphism group of $\widetilde{H}_4^{n}$ is isomorphic to $S_{2}$, generated by the swap of the two neutral pegs $0 \leftrightarrow 3$.
\end{proposition}

\begin{proof}
Swapping $0$ and $3$ preserves both legality of moves and parity constraints, yielding a nontrivial automorphism.
No other peg permutation is permitted: swapping $1$ and $2$ breaks parity feasibility, while a neutral peg cannot be exchanged with a parity-restricted one. Hence $\mathrm{Aut}(\widetilde{H}_4^{n}) \cong S_{2}$.\qedhere
\end{proof}

\subsection*{Hamiltonicity}
It is known that the Hanoi graphs $H_{m}^{n}$ are Hamiltonian for all $m \ge 3$ and $n \ge 1$ \cite{HINZ2002263}. 
In the next three theorems, namely Theorems~\ref{theo:hamiltonian_path1}, \ref{theo:hamiltonian_path2}, and \ref{theo:hamiltoniancycle}, we establish Hamiltonicity properties of the parity-constrained Hanoi graphs $\widetilde{H}_4^{n}$.

\begin{theorem}[Hamiltonian path between perfect states]
\label{theo:hamiltonian_path1}
For all $n \ge 1$, $\widetilde{H}_4^{n}$ has a Hamiltonian path between $0^{n}$ and $3^{n}$.
\end{theorem}

\begin{proof}
By induction on $n$. For $n=1$, $\widetilde{H}_4^{1}$ is a 3-cycle. Suppose the result holds for $n-1$. Then the copies $0\widetilde{H}_4^{n-1}$, $p(n)\widetilde{H}_4^{n-1}$, and $3\widetilde{H}_4^{n-1}$ are linked by moves of disc $n$, allowing concatenation of Hamiltonian paths:
\[
0^{n} \leadsto 0\,3^{n-1} \to p(n)\,3^{n-1} \leadsto p(n)\,0^{n-1} \to 3\,0^{n-1} \leadsto 3^{n}.
\]
This produces a Hamiltonian path.\qedhere
\end{proof}

\begin{theorem}[Hamiltonian path from perfect state to parity-separated configuration]
\label{theo:hamiltonian_path2}
For all $n \ge 1$, there is a Hamiltonian path from $i^{n}$, where $i \in \{0,3\}$, to the parity-separated state in which all even discs lie on peg $1$ and all odd discs on peg $2$.
\end{theorem}

\begin{proof}
Without loss of generality, consider the initial and final states $I$ and $F$, respectively, of objective \textup{\texttt{(b)}}, as illustrated in Figure~\ref{fig:objectives}. Let $\overline{b}_{n}$ denote the maximum number of moves required to reach $F$ from $I$ without visiting any state more than once. The goal is to construct an elementary path from $I$ to $F$ of maximum possible length, namely $|V(\widetilde{H}_4^{n})|-1 = 3^{n}-1$.

We present a feasible solution based on two moves of the largest disc $n$. The steps are as follows:
1. Move the $n-1$ smaller discs to pegs $1$ and $2$, forming two subtowers separated by parity. This requires $\overline{b}_{n-1}$ moves.
2. Move disc $n$ from peg $0$ to peg $3$.
3. Move the $n-1$ smaller discs from pegs $1$ and $2$ to peg $0$, again using $\overline{b}_{n-1}$ moves.
4. Move disc $n$ from peg $3$ to the parity peg $p(n)$.
5. Finally, move the $n-1$ discs from peg $0$ to pegs $1$ and $2$ in parity-separated configuration, using another $\overline{b}_{n-1}$ moves.

For $n$ even, this sequence may be represented as
\[
\begin{split}
([n],\varnothing,\varnothing,\varnothing)
&\xrightarrow[\overline{b}_{n-1}]{} (\{n\},[n-1]^{0},[n-1]^{1},\varnothing)
\xrightarrow[1]{} (\varnothing,[n-1]^{0},[n-1]^{1},\{n\})\\
&\xrightarrow[\overline{b}_{n-1}]{} ([n-1],\varnothing,\varnothing,\{n\})
\xrightarrow[1]{} ([n-1],\{n\},\varnothing,\varnothing)
\xrightarrow[\overline{b}_{n-1}]{} (\varnothing,[n]^{0},[n]^{1},\varnothing).
\end{split}
\]
For $n$ odd, the fourth configuration changes to $([n-1],\varnothing,\{n\},\varnothing)$.

Thus, the total number of moves in this construction is
\[
\overline{b}_n = 3\overline{b}_{n-1} + 2.
\]
Since $\overline{b}_{0} = 0$, solving the recurrence yields $\overline{b}_{n} = 3^{n}-1$.

Moreover, this construction is elementary: no state is repeated at any stage, because each subprocedure recursively performs an elementary traversal over a distinct $\widetilde{H}_4^{n-1}$ configuration space, and transitions involving disc $n$ always lead to new state classes distinguished by the position of the largest disc.

Since $|V(\widetilde{H}_4^n)| = 3^n$, the path constructed has maximum possible length $3^n - 1$, proving that it is Hamiltonian. Therefore, this procedure provides a Hamiltonian path between any two states of the form $s \in V(\widetilde{H}_4^n)$ such that $s_0 \in \{0,3\}$, $s_d = 1$ for even $d$, and $s_d = 2$ for odd $d$.\qedhere
\end{proof}

\begin{theorem}[Hamiltonicity]
\label{theo:hamiltoniancycle}
For all $n \ge 1$, the graph $\widetilde{H}_4^{n}$ contains a Hamiltonian cycle.
\end{theorem}

\begin{proof}
Consider the vertex 
\[
0s \in V(\widetilde{H}_4^{n}), \quad\text{where } s \in V(\widetilde{H}_4^{n-1}) \text{ is defined by } 
s_d = \begin{cases}
1, & d \text{ even},\\
2, & d \text{ odd},
\end{cases}
\]
which represents the configuration where disc $n$ is on peg $0$ and the smaller discs are distributed on pegs $1$ and $2$ according to parity. 
Define similarly 
\[
3s \in V(\widetilde{H}_4^{n})
\]
as the state where disc $n$ is on peg $3$ and the smaller discs remain in the same configuration.

From Theorem~\ref{theo:hamiltonian_path2}, there exists a Hamiltonian path in $0\widetilde{H}_4^{n-1}$ from $0s$ to $03^{n-1}$. 
Similarly, by Theorem~\ref{theo:hamiltonian_path1}, there exists a Hamiltonian path in $p(n)\widetilde{H}_4^{n-1}$ from $p(n)3^{n-1}$ to $p(n)0^{n-1}$, and again by Theorem~\ref{theo:hamiltonian_path2}, there exists a Hamiltonian path in $3\widetilde{H}_4^{n-1}$ from $30^{n-1}$ to $3s$.

These three Hamiltonian paths can be concatenated via the three edges
\[
\{03^{n-1}, p(n)3^{n-1}\}, \quad
\{p(n)0^{n-1}, 30^{n-1}\}, \quad
\{3s, 0s\},
\]
which correspond to legal moves of disc $n$ between the three subgraphs $0\widetilde{H}_4^{n-1}$, $p(n)\widetilde{H}_4^{n-1}$, and $3\widetilde{H}_4^{n-1}$. This yields a closed walk that visits every vertex exactly once, that is, a Hamiltonian cycle in $\widetilde{H}_4^{n}$.\qedhere
\end{proof}

\subsection*{Maximum clique }
The clique number of Hanoi graphs $H_{m}^{n}$ was established in \cite{Hinz2013}, namely $\omega(H_{m}^{n}) = m$. 
In Theorem~\ref{theo:clique}, we show that $\omega(\widetilde{H}_4^{n}) = \omega(H_{3}^{n}) = 3 < \omega(H_{4}^{n})$.

\begin{theorem}\label{theo:clique}
    Every complete subgraph of $\widetilde{H}_4^{n}$, $n \ge 1$, is induced by edges corresponding to moves of a single disc, and this disc is necessarily either disc $1$ or disc $2$. In particular, 
    \begin{equation}
        \omega(\widetilde{H}_4^{n}) = 3.
    \end{equation}
\end{theorem}

\begin{proof}
Consider a vertex $s$ that is adjacent to two vertices $s'$ and $s''$ via moves of two different discs. Then the positions of these two discs differ between $s'$ and $s''$. Since two vertices in $\widetilde{H}_4^{n}$ can only be adjacent if they differ in exactly one coordinate (i.e., in the position of a single disc), it follows that $s'$ and $s''$ cannot be adjacent. Hence, any complete subgraph must consist of vertices obtained by successive moves of the same disc. This proves the first claim.

We now show that only discs $1$ and $2$ can generate a maximal clique. In any state $s^{(1)}$ of $\widetilde{H}_4^{n}$, disc $1$ can move to two different pegs, giving rise to three mutually adjacent vertices (including $s^{(1)}$), forming a $3$-clique. Similarly, in any state $s^{(2)}$ where disc $1$ is on peg $2$, disc $2$ can move to two different pegs, again forming a $3$-clique.

However, no disc $d \ge 3$ can generate a $3$-clique. Indeed, although disc $d$ may be placed on three pegs in the underlying unconstrained four-peg Hanoi graph, in order for disc $d$ to move freely among its two possible target pegs, both of these pegs must be free of all smaller discs. This would require all discs smaller than $d$ to be placed on the forbidden parity peg for disc $d$ (the peg whose parity is opposite to that of disc $d$). But this is impossible under the parity constraint, since discs $1$ and $2$ cannot be simultaneously located on the same (forbidden) parity peg. Therefore, no disc $d \ge 3$ can generate a $3$-clique.

Consequently, the only maximal complete subgraphs arise from moves of disc $1$ or disc $2$, each giving a triangle. Hence, $\omega(\widetilde{H}_4^{n}) = 3$.\qedhere
\end{proof}

\subsection*{Coloring}

Next, in Theorems~\ref{theo:chromatic_number} and~\ref{theo:chromatic_index}, we present coloring results, namely the chromatic number and the chromatic index of $\widetilde{H}_4^n$. 
The exact values of these two parameters for the Hanoi graphs $H_{m}^{n}$ were established in \cite{ArettColoring, HINZ20121521}.

\begin{theorem}[Chromatic number]\label{theo:chromatic_number}
For all $n \ge 1$, the chromatic number is
\begin{equation}
    \chi(\widetilde{H}_4^{n})=3.
\end{equation}

\end{theorem}

\begin{proof}
It is known that $\chi(H_{4}^{n}) = 4$ (see \cite{HINZ20121521}). 
Consider such a proper $4$-coloring of $H_{4}^{n}$. 
Since $\widetilde{H}_4^{n}$ is obtained from $H_{4}^{n}$ by forbidding moves of discs between the two parity pegs $1$ and $2$, the states connected by such forbidden moves form disjoint pairs that can be merged into a single color class without introducing any monochromatic edges. 
By identifying the colors of each such pair of states, we obtain a valid $3$-coloring of $\widetilde{H}_4^{n}$. Hence, $\chi(\widetilde{H}_4^{n}) \le 3$.

On the other hand, from Theorem~\ref{theo:clique}, we know that $\widetilde{H}_4^{n}$ contains a $3$-clique. Therefore, $\chi(\widetilde{H}_4^{n}) \ge 3$.

Combining both bounds gives $\chi(\widetilde{H}_4^{n}) = 3$.\qedhere
\end{proof}

\begin{theorem}[Chromatic index]\label{theo:chromatic_index}
For all $n \ge 2$, we have 
\begin{equation}
    \chi'(\widetilde{H}_4^{n}) = \Delta(\widetilde{H}_4^{n})=5.
\end{equation}
\end{theorem}

\begin{proof}
We know that, for any graph $G$, we have
\[
\Delta(G) \le \chi'(G) \le \Delta(G) + 1.
\]
Thus, it suffices to exhibit a proper edge coloring of $\widetilde{H}_4^{n}$ using exactly $\Delta(\widetilde{H}_4^{n}) = 5$ colors (see Proposition~\ref{prop:degrees}) for $n \ge 3$. 

The case $n = 2$ is easily verified directly, as $\widetilde{H}_4^{2}$ is $4$-regular and $\chi'(\widetilde{H}_4^{2})=4$.

Now assume $n \ge 3$. We define a partition of the edge set as follows:
\[
E(\widetilde{H}_4^{n}) = E^{1}(\widetilde{H}_4^{n}) \cup E^{2}(\widetilde{H}_4^{n}) \cup E^{3}(\widetilde{H}_4^{n}) \cup E^{4}(\widetilde{H}_4^{n}) \cup E^{5}(\widetilde{H}_4^{n}),
\]
where the classes are pairwise disjoint and given by:
\begin{align*}
E^{1}(\widetilde{H}_4^{n}) &= \{e \in E(\widetilde{H}_4^{n}) \mid \text{$e$ corresponds to a move between pegs $0$ and $1$} \},\\
E^{2}(\widetilde{H}_4^{n}) &= \{e \in E(\widetilde{H}_4^{n}) \mid \text{$e$ corresponds to a move between pegs $0$ and $2$} \},\\
E^{3}(\widetilde{H}_4^{n}) &= \{e \in E(\widetilde{H}_4^{n}) \mid \text{$e$ corresponds to a move between pegs $0$ and $3$} \},\\
E^{4}(\widetilde{H}_4^{n}) &= \{e \in E(\widetilde{H}_4^{n}) \mid \text{$e$ corresponds to a move between pegs $1$ and $3$} \},\\
E^{5}(\widetilde{H}_4^{n}) &= \{e \in E(\widetilde{H}_4^{n}) \mid \text{$e$ corresponds to a move between pegs $2$ and $3$} \}.
\end{align*}

Note that, due to the parity constraint, no edge in $\widetilde{H}_4^{n}$ corresponds to a move between pegs $1$ and $2$.

We now claim that each set $E^{f}(\widetilde{H}_4^{n})$, $f \in [5]$, is an induced matching; i.e., no two distinct edges in $E^{f}(\widetilde{H}_4^{n})$ are adjacent. Indeed, suppose two distinct edges $e, e' \in E^{f}(\widetilde{H}_4^{n})$ were adjacent. Then they would both correspond to moves of the same disc between the same two pegs, implying $e = e'$, a contradiction. Thus, the edges in each $E^{f}(\widetilde{H}_4^{n})$ form a matching.

Therefore, we may assign a single color to each $E^{f}(\widetilde{H}_4^{n})$, for $f \in [5]$. This yields a proper edge coloring of $\widetilde{H}_4^{n}$ using exactly $5$ colors.

Since $\Delta(\widetilde{H}_4^{n}) = 5$ for $n \ge 3$ (Proposition~\ref{prop:degrees}), we conclude that $\chi'(\widetilde{H}_4^{n}) = \Delta(\widetilde{H}_4^{n}).$ Combined with the case $n=2$, the result follows for all $n \ge 2$.\qedhere
\end{proof}

\subsection*{State graph of the linear variant}\label{subsec:stategraph linear}

\medskip

We end this section with Figure~\ref{fig:linear_state_graphs}, which illustrates the state graphs
\(\widetilde{H}_{4}^{\mathrm{lin},2}\) and \(\widetilde{H}_{4}^{\mathrm{lin},3}\) associated with the linear variant of the
parity-constrained four-peg Tower of Hanoi. In these graphs, red edges correspond to moves of the
largest disc, while the dashed edges indicate an optimal path for objective~\textup{\texttt{(a)}}. As the figure
shows, the graphs \(\widetilde{H}_{4}^{\mathrm{lin},n}\) exhibit a particularly rigid and visually appealing
structure, with a strong layered organization and a recursive pattern that appears much more
regular than in the non-linear case. Their shape suggests that these graphs may admit a rich
combinatorial description, possibly leading to explicit results on their order, size, diameter,
Hamiltonian structure, and recursive decomposition. A systematic study of the family
\(\bigl(\widetilde{H}_{4}^{\mathrm{lin},n}\bigr)_{n\ge 1}\) therefore seems to be an interesting direction for
future research.

\begin{figure}[ht]
    \centering

\tikzset{every picture/.style={line width=0.75pt}} 

\tikzset{every picture/.style={line width=0.75pt}} 

\begin{tikzpicture}[x=0.75pt,y=0.75pt,yscale=-1,xscale=1]

\draw (89,164.4) node [anchor=north west][inner sep=0.75pt]  []  {$\widetilde{H}_{4}^{2,\text{lin}}$};
\draw (494,139.4) node [anchor=north west][inner sep=0.75pt]  [font=\fontsize{0.35em}{0.42em}\selectfont,]  {$200$};
\draw (394,139.4) node [anchor=north west][inner sep=0.75pt]  [font=\fontsize{0.35em}{0.42em}\selectfont,]  {$203$};
\draw (444,139.4) node [anchor=north west][inner sep=0.75pt]  [font=\fontsize{0.35em}{0.42em}\selectfont,]  {$202$};
\draw (494,38.4) node [anchor=north west][inner sep=0.75pt]  [font=\fontsize{0.35em}{0.42em}\selectfont,]  {$230$};
\draw (394,38.4) node [anchor=north west][inner sep=0.75pt]  [font=\fontsize{0.35em}{0.42em}\selectfont,]  {$233$};
\draw  [color={rgb, 255:red, 65; green, 117; blue, 5 }  ,draw opacity=1 ]  (450, 39.5) circle [x radius= 9.18, y radius= 9.18]   ;
\draw (444,38.4) node [anchor=north west][inner sep=0.75pt]  [font=\fontsize{0.35em}{0.42em}\selectfont,]  {$232$};
\draw (494,89.4) node [anchor=north west][inner sep=0.75pt]  [font=\fontsize{0.35em}{0.42em}\selectfont,]  {$210$};
\draw (394,89.4) node [anchor=north west][inner sep=0.75pt]  [font=\fontsize{0.35em}{0.42em}\selectfont,]  {$213$};
\draw  [color={rgb, 255:red, 65; green, 117; blue, 5 }  ,draw opacity=1 ]  (450, 90.5) circle [x radius= 9.18, y radius= 9.18]   ;
\draw (444,89.4) node [anchor=north west][inner sep=0.75pt]  [font=\fontsize{0.35em}{0.42em}\selectfont,]  {$212$};
\draw (544,139.4) node [anchor=north west][inner sep=0.75pt]  [font=\fontsize{0.35em}{0.42em}\selectfont,]  {$300$};
\draw (644,139.4) node [anchor=north west][inner sep=0.75pt]  [font=\fontsize{0.35em}{0.42em}\selectfont,]  {$303$};
\draw (594,139.4) node [anchor=north west][inner sep=0.75pt]  [font=\fontsize{0.35em}{0.42em}\selectfont,]  {$302$};
\draw (544,38.4) node [anchor=north west][inner sep=0.75pt]  [font=\fontsize{0.35em}{0.42em}\selectfont,]  {$330$};
\draw  [color={rgb, 255:red, 65; green, 117; blue, 5 }  ,draw opacity=1 ]  (650, 39.5) circle [x radius= 9.18, y radius= 9.18]   ;
\draw (644,38.4) node [anchor=north west][inner sep=0.75pt]  [font=\fontsize{0.35em}{0.42em}\selectfont,]  {$333$};
\draw (594,38.4) node [anchor=north west][inner sep=0.75pt]  [font=\fontsize{0.35em}{0.42em}\selectfont,]  {$332$};
\draw (544,89.4) node [anchor=north west][inner sep=0.75pt]  [font=\fontsize{0.35em}{0.42em}\selectfont,]  {$310$};
\draw  [color={rgb, 255:red, 65; green, 117; blue, 5 }  ,draw opacity=1 ]  (650, 90.5) circle [x radius= 9.18, y radius= 9.18]   ;
\draw (644,89.4) node [anchor=north west][inner sep=0.75pt]  [font=\fontsize{0.35em}{0.42em}\selectfont,]  {$313$};
\draw (594,89.4) node [anchor=north west][inner sep=0.75pt]  [font=\fontsize{0.35em}{0.42em}\selectfont,]  {$312$};
\draw (344,38.4) node [anchor=north west][inner sep=0.75pt]  [font=\fontsize{0.35em}{0.42em}\selectfont,]  {$033$};
\draw (244,38.4) node [anchor=north west][inner sep=0.75pt]  [font=\fontsize{0.35em}{0.42em}\selectfont,]  {$030$};
\draw (294,38.4) node [anchor=north west][inner sep=0.75pt]  [font=\fontsize{0.35em}{0.42em}\selectfont,]  {$032$};
\draw (344,139.4) node [anchor=north west][inner sep=0.75pt]  [font=\fontsize{0.35em}{0.42em}\selectfont,]  {$003$};
\draw  [color={rgb, 255:red, 208; green, 2; blue, 27 }  ,draw opacity=1 ]  (250, 140.5) circle [x radius= 9.18, y radius= 9.18]   ;
\draw (244,139.4) node [anchor=north west][inner sep=0.75pt]  [font=\fontsize{0.35em}{0.42em}\selectfont,]  {$000$};
\draw (294,139.4) node [anchor=north west][inner sep=0.75pt]  [font=\fontsize{0.35em}{0.42em}\selectfont,]  {$002$};
\draw (344,89.4) node [anchor=north west][inner sep=0.75pt]  [font=\fontsize{0.35em}{0.42em}\selectfont,]  {$013$};
\draw (244,89.4) node [anchor=north west][inner sep=0.75pt]  [font=\fontsize{0.35em}{0.42em}\selectfont,]  {$010$};
\draw (294,89.4) node [anchor=north west][inner sep=0.75pt]  [font=\fontsize{0.35em}{0.42em}\selectfont,]  {$012$};
\draw (657,16.4) node [anchor=north west][inner sep=0.75pt]  [font=\tiny,color={rgb, 255:red, 65; green, 117; blue, 5 }  ,opacity=1 ,]  {$( a)$};
\draw (457,71.4) node [anchor=north west][inner sep=0.75pt]  [font=\tiny,color={rgb, 255:red, 65; green, 117; blue, 5 }  ,opacity=1 ,]  {$( b)$};
\draw (657,71.4) node [anchor=north west][inner sep=0.75pt]  [font=\tiny,color={rgb, 255:red, 65; green, 117; blue, 5 }  ,opacity=1 ,]  {$( d)$};
\draw (457,16.4) node [anchor=north west][inner sep=0.75pt]  [font=\tiny,color={rgb, 255:red, 65; green, 117; blue, 5 }  ,opacity=1 ,]  {$( c)$};
\draw (116,72.28) node [anchor=north west][inner sep=0.75pt]  [font=\tiny,color={rgb, 255:red, 65; green, 117; blue, 5 }  ,opacity=1 ,]  {$( b)$};
\draw (166,72.28) node [anchor=north west][inner sep=0.75pt]  [font=\tiny,color={rgb, 255:red, 65; green, 117; blue, 5 }  ,opacity=1 ,]  {$( d)$};
\draw  [color={rgb, 255:red, 65; green, 117; blue, 5 }  ,draw opacity=1 ]  (160.67, 39.5) circle [x radius= 9.18, y radius= 9.18]   ;
\draw (154.67,38.4) node [anchor=north west][inner sep=0.75pt]  [font=\fontsize{0.35em}{0.42em}\selectfont,]  {$33$};
\draw (53.67,38.4) node [anchor=north west][inner sep=0.75pt]  [font=\fontsize{0.35em}{0.42em}\selectfont,]  {$30$};
\draw  [color={rgb, 255:red, 65; green, 117; blue, 5 }  ,draw opacity=1 ]  (109.67, 39.5) circle [x radius= 9.18, y radius= 9.18]   ;
\draw (103.67,38.4) node [anchor=north west][inner sep=0.75pt]  [font=\fontsize{0.35em}{0.42em}\selectfont,]  {$32$};
\draw (153.67,139.4) node [anchor=north west][inner sep=0.75pt]  [font=\fontsize{0.35em}{0.42em}\selectfont,]  {$03$};
\draw  [color={rgb, 255:red, 208; green, 2; blue, 27 }  ,draw opacity=1 ]  (59.67, 140.5) circle [x radius= 9.18, y radius= 9.18]   ;
\draw (53.67,139.4) node [anchor=north west][inner sep=0.75pt]  [font=\fontsize{0.35em}{0.42em}\selectfont,]  {$00$};
\draw (103.67,139.4) node [anchor=north west][inner sep=0.75pt]  [font=\fontsize{0.35em}{0.42em}\selectfont,]  {$02$};
\draw  [color={rgb, 255:red, 65; green, 117; blue, 5 }  ,draw opacity=1 ]  (159.67, 90.5) circle [x radius= 9.18, y radius= 9.18]   ;
\draw (153.67,89.4) node [anchor=north west][inner sep=0.75pt]  [font=\fontsize{0.35em}{0.42em}\selectfont,]  {$13$};
\draw (53.67,89.4) node [anchor=north west][inner sep=0.75pt]  [font=\fontsize{0.35em}{0.42em}\selectfont,]  {$10$};
\draw  [color={rgb, 255:red, 65; green, 117; blue, 5 }  ,draw opacity=1 ]  (109.67, 90.5) circle [x radius= 9.18, y radius= 9.18]   ;
\draw (103.67,89.4) node [anchor=north west][inner sep=0.75pt]  [font=\fontsize{0.35em}{0.42em}\selectfont,]  {$12$};
\draw (166,16.4) node [anchor=north west][inner sep=0.75pt]  [font=\tiny,color={rgb, 255:red, 65; green, 117; blue, 5 }  ,opacity=1 ,]  {$( a)$};
\draw (116,16.4) node [anchor=north west][inner sep=0.75pt]  [font=\tiny,color={rgb, 255:red, 65; green, 117; blue, 5 }  ,opacity=1 ,]  {$( c)$};
\draw (426,164.4) node [anchor=north west][inner sep=0.75pt]  []  {$\widetilde{H}_{4}^{3,\text{lin}}$};
\draw [color={rgb, 255:red, 0; green, 0; blue, 0 }  ,draw opacity=1 ]   (410,141) -- (441,141) ;
\draw [color={rgb, 255:red, 0; green, 0; blue, 0 }  ,draw opacity=1 ]   (491,141) -- (460,141) ;
\draw [color={rgb, 255:red, 0; green, 0; blue, 0 }  ,draw opacity=1 ]   (410,40) -- (440.64,40) ;
\draw [color={rgb, 255:red, 0; green, 0; blue, 0 }  ,draw opacity=1 ]   (460.36,40) -- (491,40) ;
\draw [color={rgb, 255:red, 0; green, 0; blue, 0 }  ,draw opacity=1 ] [dash pattern={on 4.5pt off 4.5pt}]  (410,91) -- (440.64,91) ;
\draw [color={rgb, 255:red, 0; green, 0; blue, 0 }  ,draw opacity=1 ] [dash pattern={on 4.5pt off 4.5pt}]  (460.36,91) -- (491,91) ;
\draw [color={rgb, 255:red, 0; green, 0; blue, 0 }  ,draw opacity=1 ]   (500.5,46) -- (500.5,85) ;
\draw [color={rgb, 255:red, 0; green, 0; blue, 0 }  ,draw opacity=1 ]   (400.5,97) -- (400.5,135) ;
\draw [color={rgb, 255:red, 0; green, 0; blue, 0 }  ,draw opacity=1 ]   (641,141) -- (610,141) ;
\draw [color={rgb, 255:red, 0; green, 0; blue, 0 }  ,draw opacity=1 ]   (560,141) -- (591,141) ;
\draw [color={rgb, 255:red, 0; green, 0; blue, 0 }  ,draw opacity=1 ] [dash pattern={on 4.5pt off 4.5pt}]  (640.64,40) -- (610,40) ;
\draw [color={rgb, 255:red, 0; green, 0; blue, 0 }  ,draw opacity=1 ]   (591,40) -- (560,40) ;
\draw [color={rgb, 255:red, 0; green, 0; blue, 0 }  ,draw opacity=1 ]   (640.64,91) -- (610,91) ;
\draw [color={rgb, 255:red, 0; green, 0; blue, 0 }  ,draw opacity=1 ] [dash pattern={on 4.5pt off 4.5pt}]  (591,91) -- (560,91) ;
\draw [color={rgb, 255:red, 0; green, 0; blue, 0 }  ,draw opacity=1 ] [dash pattern={on 4.5pt off 4.5pt}]  (600.5,46) -- (600.5,85) ;
\draw [color={rgb, 255:red, 0; green, 0; blue, 0 }  ,draw opacity=1 ]   (550.5,46) -- (550.5,85) ;
\draw [color={rgb, 255:red, 0; green, 0; blue, 0 }  ,draw opacity=1 ]   (650.5,100.86) -- (650.5,135) ;
\draw [color={rgb, 255:red, 0; green, 0; blue, 0 }  ,draw opacity=1 ]   (600.5,135) -- (600.5,97) ;
\draw [color={rgb, 255:red, 208; green, 2; blue, 27 }  ,draw opacity=1 ]   (510,141) -- (541,141) ;
\draw [color={rgb, 255:red, 208; green, 2; blue, 27 }  ,draw opacity=1 ] [dash pattern={on 4.5pt off 4.5pt}]  (510,91) -- (541,91) ;
\draw [color={rgb, 255:red, 0; green, 0; blue, 0 }  ,draw opacity=1 ]   (260,40) -- (291,40) ;
\draw [color={rgb, 255:red, 0; green, 0; blue, 0 }  ,draw opacity=1 ]   (341,40) -- (310,40) ;
\draw [color={rgb, 255:red, 0; green, 0; blue, 0 }  ,draw opacity=1 ] [dash pattern={on 4.5pt off 4.5pt}]  (260.36,141) -- (291,141) ;
\draw [color={rgb, 255:red, 0; green, 0; blue, 0 }  ,draw opacity=1 ]   (310,141) -- (341,141) ;
\draw [color={rgb, 255:red, 0; green, 0; blue, 0 }  ,draw opacity=1 ]   (260,91) -- (291,91) ;
\draw [color={rgb, 255:red, 0; green, 0; blue, 0 }  ,draw opacity=1 ] [dash pattern={on 4.5pt off 4.5pt}]  (310,91) -- (341,91) ;
\draw [color={rgb, 255:red, 0; green, 0; blue, 0 }  ,draw opacity=1 ] [dash pattern={on 4.5pt off 4.5pt}]  (300.5,135) -- (300.5,97) ;
\draw [color={rgb, 255:red, 0; green, 0; blue, 0 }  ,draw opacity=1 ]   (350.5,135) -- (350.5,97) ;
\draw [color={rgb, 255:red, 0; green, 0; blue, 0 }  ,draw opacity=1 ]   (250.5,85) -- (250.5,46) ;
\draw [color={rgb, 255:red, 0; green, 0; blue, 0 }  ,draw opacity=1 ]   (300.5,46) -- (300.5,85) ;
\draw [color={rgb, 255:red, 208; green, 2; blue, 27 }  ,draw opacity=1 ] [dash pattern={on 4.5pt off 4.5pt}]  (360,91) -- (391,91) ;
\draw [color={rgb, 255:red, 208; green, 2; blue, 27 }  ,draw opacity=1 ]   (360,40) -- (391,40) ;
\draw [color={rgb, 255:red, 0; green, 0; blue, 0 }  ,draw opacity=1 ]   (69.67,40) -- (100.31,40) ;
\draw [color={rgb, 255:red, 0; green, 0; blue, 0 }  ,draw opacity=1 ] [dash pattern={on 4.5pt off 4.5pt}]  (151.31,40) -- (120.03,40) ;
\draw [color={rgb, 255:red, 0; green, 0; blue, 0 }  ,draw opacity=1 ] [dash pattern={on 4.5pt off 4.5pt}]  (70.03,141) -- (100.67,141) ;
\draw [color={rgb, 255:red, 0; green, 0; blue, 0 }  ,draw opacity=1 ]   (119.67,141) -- (150.67,141) ;
\draw [color={rgb, 255:red, 0; green, 0; blue, 0 }  ,draw opacity=1 ]   (69.67,91) -- (100.31,91) ;
\draw [color={rgb, 255:red, 0; green, 0; blue, 0 }  ,draw opacity=1 ]   (120.03,91) -- (150.31,91) ;
\draw [color={rgb, 255:red, 208; green, 2; blue, 27 }  ,draw opacity=1 ] [dash pattern={on 4.5pt off 4.5pt}]  (110.17,135) -- (110.17,100.86) ;
\draw [color={rgb, 255:red, 208; green, 2; blue, 27 }  ,draw opacity=1 ]   (160.17,135) -- (160.17,100.86) ;
\draw [color={rgb, 255:red, 208; green, 2; blue, 27 }  ,draw opacity=1 ]   (60.17,85) -- (60.17,46) ;
\draw [color={rgb, 255:red, 208; green, 2; blue, 27 }  ,draw opacity=1 ] [dash pattern={on 4.5pt off 4.5pt}]  (110.17,49.86) -- (110.17,81.14) ;

\end{tikzpicture}
    \caption{The state graphs \(\widetilde{H}_{4}^{\mathrm{lin},2}\) and \(\widetilde{H}_{4}^{\mathrm{lin},3}\) of the linear
    variant of the parity-constrained four-peg Tower of Hanoi. Red edges correspond to moves of the
    largest disc, and dashed edges indicate an optimal path for objective~\texttt{(a)}.}
    \label{fig:linear_state_graphs}
\end{figure}

\section{Conclusion, Discussion, and Future Directions}
In this paper, we introduced a parity-constrained variant of the four-peg Tower of Hanoi problem, in which two pegs are neutral and the two others are reserved for even- and odd-labelled discs. We considered four natural transfer objectives and proposed four recursive algorithms for them. The corresponding candidate move counts \(a_n,b_n,c_n,d_n\) satisfy a coupled system of recurrences, from which we derived simplified recurrences, higher-order recurrences, and explicit closed formulas. These formulas reveal a periodic behavior modulo \(6\) and show that the four candidate sequences have the same half-exponential order of growth, namely \(\Theta((\sqrt{2})^n)\). The central point of the paper is now formulated as the main optimality conjecture: the candidate values \(a_n,b_n,c_n,d_n\) are conjectured to coincide with the true optimal values \(A_n,B_n,C_n,D_n\). We also investigated the number of shortest solutions, proposed conjectural recurrence formulas for these multiplicities, examined a linear version of the problem in which only adjacent moves are allowed, and introduced the parity-constrained Hanoi state graph \(\widetilde{H}_4^n\), for which several structural and graph-theoretic properties were studied.

Several directions remain open. The first and most important one is to prove or disprove the uniqueness of the largest disc move conjecture and  the  main optimality conjecture, namely,  Conjectures~\ref{conj:largest-disc-once} and~\ref{conj:mainconj}. Such a proof would require understanding whether the canonical intermediate configurations used by the recursive algorithms are unavoidable in shortest solutions, a difficulty reminiscent of the structural issues arising in the Frame--Stewart problem. Beyond this central question, the conjectural aspects of the present work also deserve further study: Conjecture~\ref{conj:bounds} on the comparison between the parity-constrained lengths and the classical four-peg optimum, Conjecture~\ref{conj:number-optimal-solutions} on the recurrence structure of the numbers of shortest solutions, and Conjecture~\ref{conj:diameter_equals_a_n} asserting that the diameter of the state graph \(\widetilde{H}_4^n\) is equal to the candidate value \(a_n\). It would also be natural to generalize the model to more than four pegs, to study parity-constrained Hanoi problems associated with other movement digraphs, such as cyclic and star digraphs, and to compare how the geometry of allowed moves affects lengths, multiplicities, and graph structure; see \cite[Chapter 8]{HinzMythsMaths2018} for related Tower of Hanoi variants with movement digraphs. Finally, many questions about the state graphs remain open, including perfectness, metric dimension, domination-type parameters, spectral properties, and the structure of the linear state graphs \(\widetilde{H}_{4}^{\mathrm{lin},n}\), whose first examples suggest a particularly regular geometry.

\label{sec:Conclusion}

\section*{Acknowledgments}

 The author is very grateful to Professor Andreas M.~Hinz (Munich) for his careful reading of an earlier version of this manuscript and for his thoughtful comments on the proposed recursive arguments. In particular, his criticism of the attempted proof of optimality (Conjectures~\ref{conj:largest-disc-once} and \ref{conj:mainconj}) was essential in clarifying the main structural difficulty of the problem. These comments led to the present, more cautious formulation of the recursive algorithms as conjecturally optimal and helped improve the overall exposition of the paper.


\bibliographystyle{plainurl} 
\bibliography{bibou}

\end{document}